\documentclass[twoside,12pt]{article}
\usepackage{amsmath,amssymb,amsthm,amsfonts}
\usepackage{paralist}
\usepackage{subfigure}
\usepackage[title]{appendix}

\usepackage{graphics} 
\usepackage{epsfig} 
\usepackage{graphicx}
\usepackage{caption}
\usepackage{epstopdf}
\usepackage[colorlinks=true]{hyperref}
\usepackage{float}
\usepackage{pdfpages}
\usepackage{amsfonts}
\usepackage{amsmath,amsthm}
\usepackage{multirow}
\usepackage{amsmath}
\usepackage{txfonts}
\usepackage{enumerate}
\usepackage[numbers,sort&compress]{natbib}
\usepackage[font=scriptsize,labelfont=bf,width=0.95\textwidth]{caption}

\numberwithin{equation}{section}
\newcommand{\R}{{\mathbb R}}

\newcommand{\be}{\begin{equation}}
\newcommand{\ee}{\end{equation}}

\newcommand{\ben}{\begin{eqnarray*}}
\newcommand{\enn}{\end{eqnarray*}}

\newtheorem{proposition}{Proposition}[section]
\newtheorem{theorem}{\textbf Theorem}[section]
\newtheorem{lemma}{\textbf Lemma}[section]

\newtheorem{corollary}{Corollary}

 \numberwithin{equation}{section}
\newtheorem{remark}{Remark}[section]
\renewcommand{\theequation}{\arabic{section}.\arabic{equation}}

\usepackage{microtype}

\begin{document}

\title{\textbf{
Critical Mass Phenomena and Blow-up Behaviors of Ground States in Stationary Second Order Mean-field Games Systems with Decreasing Cost }}
\author{Marco Cirant \thanks{Dipartimento di Matematica, Tullio ``Levi-Civita”, Universit\`{a} di Padova,
Via Trieste 63, 35121 Padova, Italy; cirant@math.unipd.it},
Fanze Kong \thanks{Department of Mathematics, University of British Columbia, Vancouver, BC, V6T 1Z2, Canada; fzkong@math.ubc.ca},
Juncheng Wei \thanks{Department of Mathematics, Chinese University of Hong Kong
 Shatin, NT, Hong Kong; wei@math.cuhk.edu.hk}
and Xiaoyu Zeng \thanks{Department of Mathematics, Wuhan University of Technology, Wuhan 430070, China; xyzeng@whut.edu.cn}
        }
\date{\today}
\maketitle
\abstract{
This paper is devoted to the study of Mean-field Games (MFG) systems in the mass critical exponent case. We firstly establish the optimal Gagliardo-Nirenberg type inequality associated with the potential-free MFG system.  Then, under some mild assumptions on the potential function, we show that there exists a critical mass $M^*$ such that the MFG system admits a least energy solution if and only if the total mass of population density $M$ satisfies $M<M^*$.  Moreover, the blow-up behavior of energy minimizers are captured as $M\nearrow M^*$.  In particular, given the precise asymptotic expansions of the potential, we establish the refined blow-up behavior of ground states as $M\nearrow M^*.$  While studying the existence of least energy solutions, we establish new local $W^{2,p}$ estimates of solutions to Hamilton-Jacobi equations with superlinear gradient terms}. 



\medskip


\maketitle

  \hypersetup{linkcolor=black}

 \tableofcontents

\section{Introduction}
A very important topic in the fields of economics, finance and management is how participants in the market maximize their utilities and minimize costs.  Whereas, in various economical and financial problems, researchers have to consider a huge number of players, which brings many challenges to numerical and theoretical studies since they are forced to tackle numerous coupled equations from the analytic perspective.  More specifically, it is very difficult to determine the value function of each player one by one. In this situation, it is necessary to propose some simplified model that can be investigated qualitatively by using mathematical tools.  Motivated by this, scholars borrowed ideas from statistical physics and developed the theory of Mean-field games (MFG) \cite{Lasry,LST,Gue091,Gue092}.

\subsection{Mean-field Games Theory and Systems}
As mentioned above, MFG models proposed by Lasry et al. \cite{Lasry} and Huang et al. \cite{Huang} independently in 2007, are well-used to describe complex decision processes involving a huge number of homogeneous agents, which are a class of backward-forward parabolic equations consisting of Hamilton-Jacobi equations and Fokker-Planck equations. In the setting described below, their mathematical form reads as
 \begin{equation}\label{MFG-time}
\left\{
\begin{array}{ll}
u_t= -\Delta u+H(\nabla u)-V[x]-f(m), &x \in \mathbb R^n,t>0,\\
m_t=\Delta m +\nabla\cdot (\nabla H(\nabla u)m),&x \in \mathbb R^n, t>0,\\
u\vert_{t=T}=u_T, m|_{t=0}=m_0,&x\in \mathbb R^n,
\end{array}
\right.
\end{equation}
where $m$ represents the population density and $u$ denotes the value function of a typical player.  Here $m_0$ is the initial data of density and $u_T$ is the terminal data of value function $u$.  In particular, $H:\mathbb R^n\rightarrow \mathbb R$ is the so-called Hamiltonian and $V(x)$ denotes a potential function.  


To illustrate the MFG theory, we assume that for $i=1,\cdots,N$, the dynamics of the $i$-th participant is governed by the following controlled stochastic differential equation (SDE):
\begin{align}\label{game-process-dXti}
dX_t^i=-\gamma^i_t dt+\sqrt{2}dB_t^i, \ \ X_0^i=x^i\in\mathbb R^n,
\end{align}
where $x^i$ are the initial states, $\gamma^i_t$ represent the control terms and $B_t^i$ are independent Brownian motions.  For simplification, we assume that all agents are indistinguishable and follow the same game process, then drop ``$i$" in \eqref{game-process-dXti}.  The goal of each player is to minimize the following average cost :
\begin{align}\label{longsenseexpectation}
J(\gamma_t):=\mathbb E\int_0^T[L(\gamma_t)+V(X_t)+f(m(X_t))] dt + u_T(X_T),
\end{align}
where $L$ is the so-called Lagrangian, which is the Legendre transform of $H$ satisfying $H(p)=\sup_{\gamma\in\mathbb R^n}(p\gamma-L(\gamma))$.  Lasry et al. applied the dynamic programming principle to study (\ref{longsenseexpectation}) and formulated the time-dependent system (\ref{MFG-time}).  

One popular research topic in the study of (\ref{MFG-time}) is the global well-posedness and long-time dynamics of (\ref{MFG-time}), see \cite{Car12,Car13,CGMT13,cirantgoffi2021,GPM12,GPM13,GPV13}.  Numerical techniques including the finite difference method are also useful on the analysis of solutions to the backward-forward system (\ref{MFG-time}) and we refer the readers to \cite{ACD10,achdou2012mean,CS12,CS13}.

We next consider the stationary problem of (\ref{MFG-time}), which serves as a paradigm to describe the distribution of Nash equilibria of infinite-horizon differential games among numerous players. 
The mathematical form of the stationary problem of (\ref{MFG-time}) is 
 \begin{equation}\label{MFG-SS}
\left\{
\begin{array}{ll}
-\Delta u+H(\nabla u)+\lambda=f(m)+V(x) , &x \in \mathbb R^n,     \\
 \Delta m+\nabla\cdot (m\nabla H(\nabla u))=0,&x \in \mathbb R^n,\\
\int_{\mathbb R^N} mdx=M>0,
\end{array}
\right.
\end{equation}
where $\lambda\in \mathbb R$ is the Lagrange multiplier.  Here a triple $(m,u,\lambda)$ is defined as the solution to (\ref{MFG-SS}) and constant $M>0$ is the total mass of population.  System (\ref{MFG-SS}) models in the market, the long time collective behavior of a huge number of homogeneous agents whose goals are to minimize a cost that depends on the population distribution $m$.  One can understand the solution to (\ref{MFG-SS}) as follows:  firstly, the optimal control of a typical agent is inputted into the Fokker-Planck equation in (\ref{MFG-SS}), yielding $m$. In the long-time regime, this has to coincide with the population density.  Then, the optimal control is outputted via the HJB equation.  In particular, the Nash equilibrium is attained when the input coincides with the output.  

We would like to point out that Hamiltonian $H$ is convex, and we assume that it has the following typical form:
\begin{align}\label{MFG-H}
H:={C_H}\vert p\vert^{r'},~~\exists r'>1,~C_H>0.
\end{align}
By using the definition, the corresponding Lagrangian is given by 
\begin{align}\label{MFG-L}
L={C_L}\vert\gamma\vert^r,~~r=\frac{r'}{r'-1}>1,~C_L=\frac{1}{r}(r'C_H)^{\frac{1}{1-r'}}>0,
\end{align}
where $r$ is the conjugate number of $r'.$

Focusing on the stationary problem (\ref{MFG-SS}), some works appeared in the last few years \cite{cesaroni2018concentration,GM15,gomes2016regularity,meszaros2015variational,cirant2016stationary}.  It is worthy mentioning that if the interaction between agents is assumed to be of congestion type, i.e. the cost is monotone increasing, then the uniqueness of the stationary solution can be proved \cite{Lasry,GM15,gomes2016regularity}.  However, if the cost is monotone decreasing, (\ref{MFG-SS}) may admit many solutions, and the analysis of existence becomes more challenging if the cost is also unbounded.  Motivated by this, Cesaroni and the first author \cite{cesaroni2018concentration} employed the variational structure possessed by (\ref{MFG-SS}) and showed the existence and concentration behaviors of ground states via the direct method under some assumptions on $f$, $V$ and $H$. 

Before stating the technical conditions on the functions $f$, $V$ and $H$, we discuss the strong relationship between (\ref{MFG-SS}) and nonlinear Schr\"{o}dinger equations.  In fact, whenever it is possible to trivialize the Fokker-Planck equation in (\ref{MFG-SS}) as
\begin{align}\label{FPeqpartially}
\nabla m+mC_H|\nabla u|^{r'-2}\nabla u=0~~\text{a.e.,}~~x\in\mathbb R^n,
\end{align}
then setting $v:=m^{\frac{1}{r}}$, one finds from the $u$-equation in (\ref{MFG-SS}) that
\begin{align}\label{nonlinear-Schrodinger}
\left\{\begin{array}{ll}
-\mu\Delta_r v+[f(v^r)+V(x)-\lambda]v^{r-1}=0,~x\in\mathbb R^n,\\
\int_{\mathbb R^n} v^r\,dx=M,~v>0,~\mu=\big(\frac{r}{C_H}\big)^{r-1},
\end{array}
\right.
\end{align}
where $\Delta_r$ is the $r$-Laplacian and given by $\Delta_rv=\nabla\cdot(|\nabla v|^{r-2}\nabla v)$.  The $v$-equation is the famous Schr\"{o}dinger equation with the mass constraint and there are rich literatures \cite{guo2014massRobert,guo2016energy} on the study of stationary problem (\ref{nonlinear-Schrodinger}) via the following useful variational structures:
\begin{align}\label{variation-schrodinger}
\mathcal F(v):=\int_{\mathbb R^n}\bigg[\frac{\mu}{r}|\nabla v|^r+F(v)+\frac{1}{r}V(x)v^r\bigg]\, dx,
\end{align}
where $F(v)$ denotes the anti-derivative of $f(v^r)v^{r-1}.$  We would like to mention that when $r=2$ and $f(v^r)=-v^{r\alpha}$ in (\ref{nonlinear-Schrodinger}), the properties of the ground states to (\ref{nonlinear-Schrodinger}) associated with (\ref{variation-schrodinger}) are well-understood.  As shown in \cite{kwong1989uniqueness}, it is well-known that under the subcritical Sobolev exponent case, the solution to (\ref{nonlinear-Schrodinger}) is radially symmetric, unique up to scaling and translation, has the exponential decay property if $\lambda<0$, etc.  It is necessary to point out that \eqref{FPeqpartially} is not expected to hold in general \cite{CirantHC}, but it is true for instance in the quadratic case $r' = 2$ or if one considers radial solutions. This strong connection when $r'=2$ is seen also at the variational level: the ground states to (\ref{MFG-SS}) and (\ref{nonlinear-Schrodinger}) are the same under some general assumptions on $f(m)$ and $V$  given in \eqref{MFG-SS}, see the detailed discussion in Appendix \ref{appendixA}. 



Inspired by the discussion above, Cirant et al. \cite{cesaroni2018concentration} proved the existence of least energy solutions to (\ref{MFG-SS}) via the variational method with some technical assumptions imposed on $H$, $V$ and $f$.  In detail, Hamiltonian $H$ is assumed to satisfy $H\in C^2(\mathbb R^n\backslash\{0\})$; moreover, there exist $C_H>0$, $K>0$ and $r'>1$ such that
\begin{align}\label{cirant-H}
C_H |p|^{r'}-K\leq H\leq C_H|p|^{r'},
\end{align}
and 
\begin{align}\label{cirant-H2}
\nabla H(p)\cdot p-H(p)\geq K^{-1}|p|^{\gamma}-K\text{~and~}|\nabla H(p)|\leq K|p|^{\gamma-1}.
\end{align}
In addition, they suppose that the potential $V$ is locally H\"{o}lder continuous and satisfy 
\begin{align}\label{cirant-V}
 C_V^{-1}(\max\{|x|-C_V,0\})^b\leq V(x)\leq C_V(1+|x|)^b,~~\exists b,C_V>0.
 \end{align}
On the other hand, the local coupling $f:[0,+\infty)\rightarrow \mathbb R$ is imposed to be locally Lipschitz continuous, which satisfies
\begin{align}\label{cirant-f}
-C_fm^{\alpha}-K\leq f(m)\leq -C_fm^{\alpha}+K,~~\exists C_f,K,\alpha>0.
\end{align}
In particular, the author were concerned with the mass subcritical exponent case, which is
\begin{align}\label{subcritical-exponent}
0<\alpha<\frac{r}{n}.
\end{align}

To illustrate (\ref{subcritical-exponent}), we firstly state the following energy functional associated with (\ref{MFG-SS}):
\begin{align}\label{energy-dual}
\mathcal E(m,w):=
\int_{\mathbb R^n} \left[mL\bigg(-\frac{w}{m}\bigg)+V(x)m+F(m)\right]\, dx,
\end{align}
where $F(m):=\int_0^m f(s)ds$ for $m\geq 0$ and $F(m)=0$ for $m\leq 0.$  As shown in \cite{cesaroni2018concentration}, Lagrangian $L$ is defined by
\begin{align}\label{general-Lagrangian}
L\bigg(-\frac{w}{m}\bigg):=\left\{\begin{array}{ll}
\sup\limits_{p\in\mathbb R^n}\big(-\frac{pw}{m}-H(p)\big),&m>0,\\
0,&(m,w)=(0,0),\\
+\infty,&\text{otherwise},
\end{array}
\right.
\end{align}
and $mL\big(-\frac{w}{m}\big)$ is the Legendre transform of $mH(p)$.  In addition, the admissible set $\mathcal K_{M}$ is defined as
\begin{align}\label{constraint-set-K}
\mathcal K_{ M}:=\Big\{&(m,w)\in (L^1(\mathbb R^n)\cap W^{1,\hat q}(\mathbb R^n))\times L^{1}(\mathbb R^n)\nonumber\\
~&\text{s. t. }\int_{\mathbb R^n}\nabla m\cdot\nabla\varphi\,dx=\int_{\mathbb R^n}w\cdot\nabla\varphi\, dx,\forall \varphi\in C_c^{\infty}(\mathbb R^n),\nonumber\\
~&\int_{\mathbb R^n} V(x)m\,dx<+\infty,~
\int_{\mathbb R^n}m\,dx=M>0,~m\geq 0\text{~a.e.~}\Big\},
\end{align}
where $\hat q$ is given by
\begin{equation}\label{hatqconstraint}
\hat q:=\begin{cases}
\frac{n}{n-r+1} &\text{ if }r<n,\\
\in\left(\frac{2n}{n+2},n\right)&\text{ if }r=n,\\
r &\text{ if }r> n.
\end{cases}
\end{equation}

With these preliminaries, to find the ground states of (\ref{MFG-SS}) associated with energy (\ref{energy-dual}), we consider the following problem: 
\begin{align}\label{ealphaM-117}
e_{\alpha,M}:=\inf_{(m,w)\in \mathcal K_{M}}\mathcal E(m,w).
\end{align}

We would like to remark that $e_{\alpha,M}<+\infty.$  To show this, one sets $(m_s,w_s)=\Big(ce^{-|x|},-\frac{xe^{-|x|}}{|x|}\Big)$ with $c$ determined by $\int_{\mathbb R^n}m\,dx=M$ and $w_s=\nabla m_s$, then can straightforward check $(m_s,w_s)\in\mathcal K_{\alpha,M}$ and $\mathcal E(m_s,w_s)<+\infty.$  With the upper bound of $e_{\alpha,M},$ we claim that (\ref{subcritical-exponent}) is the necessary condition to guarantee that $\mathcal E(m,w)$ is bounded from below for all $M>0.$  In this case, $e_{\alpha,M}$ given in (\ref{ealphaM-117}) is well-defined.  Indeed, if $\alpha>\frac{r}{n},$ for any $(\bar m,\bar w)\in \mathcal K_{\alpha, M},$
\begin{align*}
\mathcal E(\bar m_{\delta},\bar w_{\delta})\rightarrow -\infty,
\end{align*}
as $\delta\rightarrow 0,$ in which $(\bar m_{\delta},\bar w_{\delta}):=(\delta^{-n}\bar m(\delta^{-1}x),\delta^{-(n+1)}\bar w(\delta^{-1}x))\in \mathcal K_{\alpha,M}$ since $\delta^{-n}\int_{\mathbb R^n}\bar m(\delta^{-1}x)\, dx\equiv M.$  It follows that $e_{\alpha,M}$ is not well-defined when $\alpha>\frac{r}{n}$.  As a consequence, Cirant et al. \cite{cesaroni2018concentration} discussed the ground states to (\ref{MFG-SS}) associated with energy (\ref{energy-dual}) under condition (\ref{subcritical-exponent}).  In particular, the authors \cite{cesaroni2018concentration} assumed $H$, $V$ and $f$ satisfy (\ref{cirant-H}), (\ref{cirant-H2}), (\ref{cirant-V}) and (\ref{cirant-f}), respectively. 
 Moreover, they showed the ground states to (\ref{MFG-SS}) are in fact classical solutions by applying the regularization techniques and standard elliptic estimates.  The concentration behaviors of the ground states were also established in the vanishing viscosity limit sense under the mass subcritical exponent case \eqref{subcritical-exponent} in  \cite{cesaroni2018concentration}.

In this paper, we shall investigate the qualitative behaviors of ground states to (\ref{MFG-SS}) with $n\geq 2$ and focus on the \textit{mass critical exponent case}  $\alpha=\alpha^*:=\frac{r}{n}.$  More specifically, we aim to discuss the existence and asymptotic profiles of ground states to (\ref{energy-dual}).  This critical case is much more delicate than the subcritical one, and many arguments used in \cite{cesaroni2018concentration} break down.  It is worthy mentioning that there exists the other critical exponent $\alpha=\frac{r}{n-r}$ arising from Sobolev embedding on analyzing the regularity of $m$-component to (\ref{MFG-SS}).  Recently, some results in the mass supercritical case up to the Sobolev critical exponent were completed by the first author and his collaborators \cite{cirant2023ergodic}, in which the domain is assumed to be bounded and Neumann boundary conditions are imposed on $(m,u)$.  More specifically, it is shown that energy (\ref{energy-dual}) admits local minimizers with the mass supercritical exponent.    

Now, we state our main results in the following subsection.

\subsection{Main results}\label{mainresults11}

This paper extends the results in \cite{cesaroni2018concentration} to the mass critical case, and allows for much more general potentials $V$.  In detail, we consider $\alpha=\alpha^*:=\frac{r}{n}$, $H$ is given by (\ref{MFG-H}), $f(m):=- m^{\alpha}$ and locally H\"{o}lder continuous potential $V$ satisfies some of the following conditions
\begin{center}
\begin{itemize}
\item[(V1). ] $\inf\limits_{x\in\mathbb R^n}V(x)=0\leq V(x)\in L^\infty_{\text{loc}}(\mathbb R^n)$.   
\item[(V2). ] there exist constants $C_1,C_2,K>0$ and $b, \delta>0$  such that 
\begin{subequations}\label{V2mainasumotiononv}
\begin{align}
    &C_1(1+|x|^b)\leq V(x)\leq C_2e^{\delta|x|}, \ \ |x|\geq K; \label{V2mainassumption_1}\\
   &0< C_1\leq \frac{V(x+y)}{V(x)}\leq C_2,\text{~for~all~}|x|\geq K \text{~with~}|y|<2; \label{V2mainassumption_2}\\
   &\sup_{\nu\in[0,1]}V(\nu x)\leq C_2V(x)\text{~for~}|x|\geq K.\label{V2mainassumption_3}
\end{align}
\end{subequations}

\item[(V3).]  $|\mathcal Z|=0$ with $\mathcal Z:=\{x\in\mathbb R^n~|~V(x)=0\}$.
\end{itemize}
\end{center}

Then, we are able to classify the ground states to (\ref{MFG-SS}) by using the variational method and introducing some critical mass threshold $M^*.$  A vital step here is to study the regularity of the solution after the minimizer is obtained.  In fact, the regularity of solutions to the Fokker-Planck equation is well understood \cite{bogachev2022fokker} and some useful results are stated in Section \ref{preliminaries}.  Whereas, focusing on the Hamilton-Jacobi equation, the quasi-linear structure may cause difficulties.  There are some references on the regularity of the Hamilton-Jacobi  equation that deserve discussion.  Consider the following equation:
\begin{align}\label{introductioneqmaximalnew}
-\Delta u+|\nabla u|^{r'}=f,
\end{align}
where $r'>1$.  A natural question is the $W^{2,p}$ regularity of the solution to (\ref{introductioneqmaximalnew}) provided that $f\in L^{p}$ for some $p>1.$   
P.-L. Lions conjectured that for all $p>\frac{n}{r},$ the $L^p$ estimate should  hold for the $D^2 u$ of (\ref{introductioneqmaximalnew}) when $f$ is assumed to satisfy $f\in L^p.$  This conjecture is confirmed when the periodic boundary condition is imposed on equation (\ref{introductioneqmaximalnew}) \cite{cirant2021problem}.  For the case of $1<r'\leq 2$, the global estimates were discussed in \cite{grenon2014priori, betta2015gradient, goffi2023optimal, goffipediconi2023} with Dirichlet or Neumann boundary conditions.  There are a few results on the local $L^p$ estimates for $D^2 u$; these have been obtained by the first author \cite{cirant2022local} and Verzini, in the superquadratic case $r'>2.$  However, the local $L^p$ regularity in the case of $r'\leq 2$ is open to our knowledge. It turns out that in our analysis of the MFG system in the mass critical case, the local maximal regularity estimate plays a crucial role. Motivated by this, we obtain the following result: 

\begin{theorem}\label{thmmaximalregularity}
Let $p>\frac{n}{r}$,  $ r'\geq \frac{n}{n-1}$, $f\in L^p(\Omega)$ and assume $u\in W^{2,p}(\Omega)$ solves 
$$-\Delta u+C_H|\nabla u|^{r'}= f\text{~in~}\Omega,$$
in the strong sense.  Then for each $M>0$ and $\Omega'\subset\subset \Omega$, we have
{$$ \Vert \nabla u\Vert_{L^p(\Omega')} + \Vert D^2 u\Vert_{L^p(\Omega')}\leq C,$$}
where $\Vert f\Vert_{L^p(\Omega)}\leq M $ and the constant $C=C(M,\mathrm{dist}(\Omega',\partial\Omega),n,p, C_H,r)>0$.

\end{theorem}

Theorem \ref{thmmaximalregularity} is an analogue of Calder\'{o}n-Zygmund regularity for the linear second order elliptic equation. We would like to remark that the Lipschitz gradient estimate of the solution can be obtained \cite{lions1985quelques} when $p>n$, see also \cite{} for previous results.  For the proof of Theorem \ref{thmmaximalregularity}, we shall perform a blow-up analysis which is inspired by the strategy shown in \cite{cirant2022local}, but tailored to Morrey spaces rather than H\"older spaces. 

\medskip

Next, we come back to the MFG system, and focus on the existence of ground states under the critical mass exponent case. 
To begin with, we consider the following constrained minimization problem:  
\begin{align}\label{minimization-problem-critical}
e_{\alpha^*,M}=\inf_{(m,w)\in \mathcal K_{M}} \mathcal E(m,w), \ \ \alpha^*:=\frac{r}{n},
\end{align}
where the energy $\mathcal E(m,w)$ is given by
\begin{align}\label{41engliang}
\mathcal E(m,w)=C_L\int_{\mathbb R^n}\Big|\frac{w}{m}\Big|^{r}m\,dx+\int_{\mathbb R^n}V(x) m\,dx-\frac{n}{n+r}\int_{\mathbb R^n}m^{1+\frac{r}{n}}\, dx,
\end{align}
and $\mathcal K_{M}$ is given in \eqref{constraint-set-K}.  Before studying the existence of minimizers to (\ref{minimization-problem-critical}), one must establish the optimal Gagliardo-Nirenberg type inequality subject to the mass critical exponent involving the Lagrangian term in energy (\ref{energy-dual}). This is in fact crucial to understand whether the infimum in \eqref{minimization-problem-critical} is finite or not.
 More precisely, we focus on the following minimization problem:
\begin{align}\label{sect2-equivalence-scaling}
\Gamma_{\alpha}:=\inf_{(m,w)\in\mathcal A}\frac{\Big(C_L\int_{\mathbb R^n}m\big|\frac{w}{m}\big|^{r}\,dx\Big)^{\frac{n\alpha}{r}}\Big(\int_{\mathbb R^n}m\,dx\Big)^{\frac{(\alpha+1)r-n\alpha}{r}}}{\int_{\mathbb R^n}m^{\alpha+1}\, dx}, \ \  \alpha\in\bigg(0,\frac{r}{n}\bigg],
\end{align}
where 
\begin{align}\label{mathcalA-equivalence}
\mathcal A:=\Big\{&(m,w)\in (L^1(\mathbb R^n)\cap W^{1,\hat q}(\mathbb R^n))\times L^{1}(\mathbb R^n)\nonumber\\
~&\text{s. t. }\int_{\mathbb R^n}\nabla m\cdot\nabla\varphi\,dx=\int_{\mathbb R^n}w\cdot\nabla\varphi\, dx,\forall \varphi\in C_c^{\infty}(\mathbb R^n),~\int_{\mathbb R^n} |x|^bm\,dx<+\infty,~m\geq,\not\equiv 0\text{~a.e.~}\Big\},
\end{align}
with $\hat q$ defined as (\ref{hatqconstraint}) and
$b>0$.  We would like to remark that problem (\ref{sect2-equivalence-scaling}) is  scaling invariant  under the scaling $(t^{\beta}m(tx),t^{\beta+1}w(tx))$ for any $t>0$ and $\beta>0.$ As a consequence, it is easy to see that (\ref{sect2-equivalence-scaling}) is equivalent to the following minimization problem 
\begin{align}\label{optimal-inequality-sub}
\Gamma_{\alpha}=\inf_{(m,w)\in { {\mathcal A}_M}}\frac{\Big(C_L\int_{\mathbb R^n}m\big|\frac{w}{m}\big|^{r}\,dx\Big)^{\frac{n\alpha}{r}}\Big(\int_{\mathbb R^n}m\,dx\Big)^{\frac{(\alpha+1)r-n\alpha}{r}}}{\int_{\mathbb R^n}m^{\alpha+1}\, dx},\ \ \alpha\in\bigg(0,\frac{r}{n}\bigg],
\end{align}
where
\begin{align}\label{mathcalAMbegining}
\mathcal A_M:=\Big\{(m,w)\in  \mathcal{A}, \int_{\mathbb R^n}m\,dx=M>0\Big\}.
\end{align}


We first prove there exist minimizers to (\ref{optimal-inequality-sub}) for any $\alpha\in(0,\frac{r}{n})$ thanks to Theorem 1.3 in \cite{cesaroni2018concentration}.  Then, using an approximation argument, we show that $\Gamma_{\alpha}$ can be also attained when $\alpha=\frac{r}{n}$, \textit{provided that $M=M^*$ is suitably chosen}. The result is is summarized as 
\begin{theorem}\label{thm11-optimal}
For any $r>1$, assume that $\alpha=\alpha^*:=\frac{r}{n}$ in (\ref{sect2-equivalence-scaling}), then we have $\Gamma_{\alpha^*}$ is finite and attained by some minimizer $(m_{\alpha^*},w_{\alpha^*})\in\mathcal A$.  Correspondingly, there exists a solution {$\big(m_{\alpha^*},u_{\alpha^*}\big)\in W^{1,p}(\mathbb R^n)\times C^2(\mathbb R^n)$}, $\forall p>1,$ to the following potential-free MFG systems:
\begin{align}\label{limitingproblemminimizercritical0}
\left\{\begin{array}{ll}
-\Delta u+C_H|\nabla u|^{r'}-\frac{r}{nM^*}=-m^{\alpha^*},&x\in\mathbb R^n,\\
\Delta m+C_Hr'\nabla\cdot (m|\nabla u|^{r'-2}\nabla u)=0,&x\in\mathbb R^n,\\
w=-C_Hr'm\big|\nabla u\big|^{r'-2}\nabla u,\ \int_{\mathbb R^n}m\,dx=M^*,
\end{array}
\right.
\end{align}
where
\begin{align}\label{Mstar-critical-mass}
M^*:=[{\Gamma_{\alpha^*}(\alpha^*+1)}]^{\frac{1}{\alpha^*}}.
\end{align}
In particular, there exists some positive constants $c_1$ and $c_2$ such that $0<  m_{\alpha^*}(x)\leq c_1e^{-c_2|x|}$.
\end{theorem}

Note that working on the whole Euclidean space here involves non trivial compactness issues. These were solved in \cite{cesaroni2018concentration} by a concentration-compactness argument; here we follow a different strategy based on Pohozaev type identities.
With the aid of Theorem \ref{thm11-optimal}, we analyze the boundedness of $e_{\alpha^*,M}$ from below and obtain the sharp existence of minimizers for \eqref{minimization-problem-critical}, which is shown in the following theorem:
\begin{theorem}\label{thm11}
Assume that $V$ satisfies (V1)-(V2)  and $M^*$ is defined by (\ref{Mstar-critical-mass}),
then the following alternatives hold: 
\begin{itemize}
    \item[(i).] 
    If $0<M<M^*$, problem (\ref{minimization-problem-critical}) admits at least one minimizer $(m_M,w_M)\in W^{1,p}(\mathbb R^n)\times L^{p}(\mathbb R^n)$  $\forall p>1$, which satisfies for some $\lambda_M\in \mathbb R$,
    \begin{align}\label{125potentialfreesystem}
\left\{\begin{array}{ll}
-\Delta u_{M}+C_H|\nabla u_{M}|^{r'}+\lambda_M=V(x)- m_{M}^{\frac{r}{n}},\\
\Delta m_{M}+C_Hr'\nabla\cdot (m_{M}|\nabla u_{M}|^{r'-2}\nabla u_{M})=0,\\
w_{M}=-C_Hr'm_{M}|\nabla u_{M}|^{r'-2}\nabla u_{M}, \ \int_{\mathbb R^n}m_{M}\,dx=M<M^*.
\end{array}
\right.
\end{align}

    \item[(ii).] If $M> M^*$, there is no minimizer of problem (\ref{minimization-problem-critical}).
    \item [(iii).] If $M=M^*$, when $1<r\leq n$ and potential $V$ satisfies also (V3), or when $r>n$, then problem (\ref{minimization-problem-critical}) has no minimizer.
\end{itemize}
\end{theorem}

It is worthy mentioning that while discussing the case of $M<M^*$ in Theorem \ref{thm11}, we need to apply the maximal regularity shown in Theorem \ref{thmmaximalregularity} to obtain the uniformly boundedness of $m$-component in $L^\infty$ and the detailed discussion is shown in Section \ref{sect4-criticalmass}. 
 Theorem \ref{thm11} implies that in the mass critical exponent case, the existence of minimizers to (\ref{minimization-problem-critical}) depends on the total mass of population density $m$ and there exists a mass threshold $M^*$ explicitly given by (\ref{Mstar-critical-mass}).  In particular, when $M=M^*,$ the constrained problem (\ref{minimization-problem-critical}) is not attained.
 
 To further understand the delicate case $M=M^*$, we explore the blow-up behaviors of minimizers as $M\nearrow M^*$ and obtain that
\begin{theorem}\label{thm13basicbehavior} Assume that $V(x)$ satisfies either $(V1)-(V3)$ for $1<r\leq n$, or $(V1)-(V2)$ for $r\geq n$. 
Let $(m_M,w_M)$ be the minimizer of $e_{\alpha^*,M}$ obtained in Theorem \ref{thm11} with $0<M<M^*$ and $\mathcal Z:=\{x\in\mathbb R^n~|~V(x)=0\}$ with $|\mathcal Z|=0$.  Then, we have 
\begin{itemize}
    \item[(i).] 
    \begin{align}\label{thm51property1}
    {\varepsilon}_M={\varepsilon}:=\Big(C_L\int_{\mathbb R^n}\bigg|\frac{w_M}{m_M}\bigg|^{r}m_M\,dx\Big)^{-\frac{1}{r}}\rightarrow 0\text{~as~}M \nearrow M^*.
    \end{align}
    \item[(ii).] 
    Let $\{x_{\varepsilon}\}$ be one of the global minimum points of $u_M$, then $\text{dist}(x_{\varepsilon},\mathcal Z)\rightarrow 0$ as $M \nearrow M^*$, where $\mathcal Z=\{x\in\mathbb R^n~|~V(x)=0\}$. Moreover, 
    \begin{align}\label{thm51property2}
    u_{\varepsilon}:=\varepsilon^{\frac{2-r'}{r'-1}} u_{M}(\varepsilon x+x_{\varepsilon}),~
m_{\varepsilon}:=\varepsilon^n m_{M}(\varepsilon x+x_{\varepsilon}),~ w_{\varepsilon}:=\varepsilon^{n+1}w_M(\varepsilon x+x_{\varepsilon}),
    \end{align}
    satisfies  up to a subsequence, 
     \begin{equation}\label{eq1.40}
      u_{\varepsilon}\rightarrow u_0\text{ in }C^2_{\rm loc}(\mathbb R^n), ~
     m_{\varepsilon}\rightarrow m_0 \text{ in }L^p(\mathbb R^n) ~\forall ~p\in[1,{\hat q}^*),~ \text{ and } w_{\varepsilon}\rightharpoonup  w_0 \text{ in }L^{\hat q}(\mathbb R^n),\end{equation}
    where $(m_0,w_0)$ is  a minimizer of (\ref{sect2-equivalence-scaling}), and  $(u_0,m_0,w_0)$ satisfies \eqref{limitingproblemminimizercritical0}.
In particular, when $V$ satisfies (\ref{cirant-V}), let $\bar x_{\varepsilon}$ be any one of global maximum points of $m_M$, then  
\begin{align}\label{eq141realtionuminmmax}
\limsup_{\varepsilon\rightarrow 0}\frac{|\bar x_{\varepsilon}-x_{\varepsilon}|}{\varepsilon}<+\infty.
\end{align}
\end{itemize}
\end{theorem}
With mild assumptions imposed on potential $V$, Theorem \ref{thm13basicbehavior} exhibits the basic blow-up behaviors of ground states as $M\nearrow M^*$.  The reader may observe that when $r \le n$ we need stronger information on $V$. Moreover, by imposing some more specific local asymptotics on the potential $V(x)$, we can get the refined blow-up behaviors of ground states, which is 
\begin{theorem}\label{thm14preciseblowup}
Suppose all conditions shown in Theorem \ref{thm13basicbehavior} hold.  Assume that $V$ has $l\in\mathbb{N}_+$ distinct zeros given by $\{P_1,\cdots,P_l\}$ and  $\exists \ a_i>0$, $q_i>0$ and $d>0$ such that 
\begin{align*}
V(x)=a_i|x-P_i|^{q_i}+O\big(|x-P_i|^{q_i+1}\big), \ \ 0<|x-P_i|\leq d,\ i=1,\cdots,l.
\end{align*}
Denote
$$Z:=\{P_i~|~q_i=q,\ i=1,\cdots,l\} ~\text{ and }~Z_0:=\{P_i~|~q_i\in Z \text{~and~}\mu_i=\mu,i=1,\cdots,l\},$$  
where $q:=\max\{q_1,\cdots,q_l\}$ and $
\mu:=\min\{\mu_i~|~P_i\in Z, i=1,\cdots,l\}$  with 
\begin{equation*}
\mu_i:=\min\limits_{y\in\mathbb R^n}H_i(y),\ H_i(y):=\int_{\mathbb R^n} a_i|x+y|^{q_i}m_0(x)\,dx\ i=1,\cdots,l.
\end{equation*}
Let $(m_{\varepsilon},w_{\varepsilon},u_{\varepsilon})$ be  the convergent subsequence,  $(m_0,w_0,u_0)$ be  the  corresponding limit and $x_{\varepsilon}$ be the global minimum point of each $u_M$ chosen in Theorem \ref{thm13basicbehavior}. Then $x_{\varepsilon}\rightarrow P_i\in Z_0$.  Moreover,  as $M\nearrow M^*,$
\begin{align*}
\frac{e_{\alpha^*,M}}{\frac{q+r}{q}\Big(\frac{q\mu}{r}\Big)^{\frac{r}{r+1}}\Big[1-\Big(\frac{M}{M^*}\Big)^{\frac{r}{n}}\Big]^{\frac{q}{r+q}}}\rightarrow 1,
\end{align*}
and
\begin{align}\label{131thm14}
\frac{\varepsilon}{\left[\frac{r}{q\mu}\left[1-\big(\frac{M}{M^*}\big)^{\frac{r}{n}}\right]\right]^{\frac{1}{r+q}}}\rightarrow 1, \text{ as } M\nearrow M^*,
\end{align}
where $e_{\alpha^*,M}$ and $\varepsilon=\varepsilon_M$ are defined in (\ref{minimization-problem-critical}) and (\ref{thm51property1}), respectively.  In addition, up to a subsequence,
\begin{align}\label{132thm14}
\frac{x_{\varepsilon}-P_i}{\varepsilon_M}\rightarrow y_0 ~\text{ with }~
P_i\in Z_0 ~\text{ and }~  H_i(y_0)=\inf_{y\in\mathbb R^n} H_i(y)=\mu.
\end{align}
\end{theorem}
Theorem \ref{thm14preciseblowup} captures the precise blow-up behaviors of ground states arising from problem (\ref{minimization-problem-critical}), in which {the flattest minima of V are selected.}

\medskip

The rest of this article is organized as follows: In Section \ref{sectmaximalregular}, we establish maximal regularities of Hamilton-Jacobi equations with subquadratic gradient terms.  Section \ref{preliminaries} is devoted to some preliminaries for the proof of the existence of ground states.  In Section \ref{sect3-optimal}, we formulate the optimal Gagliardo-Nirenberg type inequality, which is Theorem \ref{thm11-optimal}.  Then in Section \ref{sect4-criticalmass}, we applied this essential inequality to derive the critical mass phenomenon shown in Theorem \ref{thm11} under the mass critical exponent case.  Finally, in Section \ref{sect5preciseblowup}, we investigate the blow-up behaviors of ground states obtained in Theorem \ref{thm11} and finish the proof of Theorem \ref{thm13basicbehavior} and \ref{thm14preciseblowup}.  Without of confusion, we define constant $C>0$ is generic, which may change line to line. 

\medskip

\section{Local Maximal Regularity of Hamilton-Jacobi Equations 
}\label{sectmaximalregular}

In this section, we focus on the following elliptic PDEs involving gradient terms:
\begin{align}\label{newsectioneq1}
-\Delta u+C_H|\nabla u|^{r'}=f\text{~in~}\Omega,
\end{align}
where $r'>1$, $C_H>0$ are constants and $\Omega\subset \mathbb R^n$ is a bounded Lipschitz domain. Our goal is to study the $W^{2,p}$ regularities of solutions to (\ref{newsectioneq1}) if we assume that $f\in L^p(\Omega)$ with $p>\frac{n}{r}$. To begin with, we give some preliminary notations and results.

\subsection{Preliminaries}
The Morrey space $L^{r',s}(\mathbb R^n)$ with $r'\in[1,\infty)$ and $s\in[0,n]$ is defined as a functional space consisting of all measurable functions $u:\mathbb R^n\rightarrow \mathbb R$ satisfying
\begin{align}\label{preliminmorreydef}
\Vert u\Vert_{L^{r',s}(\mathbb R^n)}^{r'}:=\sup_{R>0;x\in\mathbb R^n}R^{s}\fint_{B_R(x)}|u|^{r'}\,dy<+\infty,
\end{align}
where 
$$\fint_{B_R(x)}|u|^{r'}\,dy:=\frac{1}{|B_R(x)|}\int_{B_R(x)}|u|^{r'}\,dy.$$
We have facts that when $s=n,$ the Morrey space $L^{r', n}(\mathbb R^n)$ coincides with the Lebesgue space $L^{r'}(\mathbb R^n)$ and the Morrey space $L^{r',0}(\mathbb R^n)$ coincides with $L^\infty(\mathbb R^n).$  Correspondingly, if $\Omega$ is a bounded domain, we define $u\in L^{r',s}(\Omega)$ as the space of $u:\Omega \rightarrow \mathbb R$ satisfying
\begin{align*}
    \Vert u\Vert_{L^{r',s}(\Omega)}^{r'}:=\sup_{0<R<\text{diam}\Omega;x\in\Omega}R^{s}\fint_{B_R(x)\cap \Omega}|u|^{r'}\,dy<+\infty.
\end{align*}

With the definition of Morrey spaces, we recall the following improved Gagliardo-Nirenberg inequality involving Morrey's norm in the whole space $\mathbb R^n$:
\begin{lemma}\label{lemma11maximalregularity}
Let $1\leq p<n$, $1\leq r'< p^*:=\frac{np}{n-p}$, and $q=r'(\frac{n}{p}-1)$.   Then there exists a constant $C>0$ depending on $n$ and $p$ such that for any $\frac{p}{p^*}\leq \theta<1$, 
\begin{align*}
\Vert u\Vert_{L^{p^*}(\mathbb R^n)}\leq C\Vert\nabla u\Vert^{\theta}_{L^p(\mathbb R^n)}\Vert u\Vert^{1-\theta}_{L^{r',q}(\mathbb R^n)} ~\text{ for all }~u \in W^{1,p}(\mathbb R^n).
\end{align*}
\end{lemma}
\begin{proof}
See the proof of Theorem 2 in \cite{palatucci2014improved}.
\end{proof}

With the aid of Lemma \ref{lemma11maximalregularity}, we next formulate the Gagliardo-Nirenberg's inequality involving Morrey norms in the ball $B_R$.  To this end, we recall the following extension Theorem and interpolation inequalities.
\begin{lemma}[C.f. Lemma 1.5 in \cite{li2021note} ]\label{lemma12maximalregularity}
Let $\Omega$ be a Lipschitz bounded domain and assume $f\in W^{1,p}(\Omega)$ with some $1\leq p\leq+\infty.$  Then there exists a bounded linear extension $T: W^{1,p}(\Omega)\rightarrow W^{1,p}(\mathbb R^n)$  such that 
\begin{align*}
\Vert Tf\Vert_{W^{1,p}(\mathbb R^n)}\leq C\Vert f\Vert_{W^{1,p}(\Omega)},~\forall f\in W^{1,p}(\Omega),
\end{align*}
where $C$ is a constant depending on $n$, $p$ and $\Omega.$
\end{lemma}

\begin{lemma}[C.f. Lemma 2.1 in \cite{li2021note}]\label{lemma13maximalregularity}
Let $\Omega$ be a bounded Lipschitz  domain and $u\in W^{1,s}(\Omega)$, $1\leq s<+\infty$.  Then for any $\epsilon>0$, $p\geq 1$ with 
$$\frac{n}{p}>\frac{n}{s}-1,~~q>1,$$
there exists a constant $C$ depending on $n$, $q$, $r$, $\epsilon$, $\Omega$ such that
\begin{align*}
\Vert u\Vert_{L^p(\Omega)}\leq \epsilon\Vert\nabla u\Vert_{L^s(\Omega)}+C\Vert u\Vert_{L^q(\Omega)}.
\end{align*}
\end{lemma}

Now, we are ready to establish the improved Gagliardo-Nirenberg inequality in the bounded domain, which is

\begin{lemma}\label{lemma14maximalregularity}
Let  $p,q,r$ and $\theta$ satisfy the assumptions of Lemma \ref{lemma11maximalregularity}.
Assume $u\in W^{1,p}(B_R)\cap L^{r',q}(B_R)$ with $R>0$,  then there exists a constant $C>0$ depending on $n$ and $p$ such that,
\begin{align}\label{conclusion1inlemma14maximalregularity}
\Vert u\Vert_{L^{p^*}(B_R)}\leq C\Vert\nabla u\Vert^{\theta}_{L^p(B_R)}\Vert u\Vert^{1-\theta}_{L^{r',q}(B_R)}+C\Vert u\Vert_{L^{r',q}(B_R)}.
\end{align}
\end{lemma}
\begin{proof}
We consider the case of $R=1$.  By using Lemma \ref{lemma12maximalregularity}, we extend $u$ to $Tu\in W^{1,p}(\mathbb R^n)\cap L^{r',q}(\mathbb R^n)$ for $u\in W^{1,p}(B_1)$.  Then we invoke Lemma  \ref{lemma11maximalregularity} to get
\begin{equation}\label{eqTu}
\begin{split}
\Vert u\Vert_{L^{p^*}(B_1)}\leq& \Vert Tu\Vert_{L^{p^*}(\mathbb R^n)}\leq C\Vert \nabla Tu\Vert_{L^p(\mathbb R^n)}^{1-\theta}\Vert T u\Vert^{\theta}_{L^{r',q}(\mathbb R^n)}\\
\leq& C\big[\Vert u\Vert_{L^p(B_1)}+\Vert\nabla u\Vert_{L^p(B_1)}\big]^{1-\theta}\Vert u\Vert_{L^{r',q}(B_1)}^{\theta},
\end{split}
\end{equation}
where we have used the fact that $\Vert T u\Vert_{L^{r',q}(\mathbb R^n)}\leq C\Vert u\Vert_{L^{r',q}(B_1)}$. Moreover, from the definition of Morrey space, one can easily deduce that $\Vert  u\Vert_{L^{r'}(B_1)}\leq C\Vert u\Vert_{L^{r',q}(B_1)}$, it then follows from Lemma \ref{lemma13maximalregularity} that 
$$\Vert u\Vert_{L^p(B_1)}\leq \epsilon  \Vert\nabla u\Vert_{L^p(B_1)}+C\Vert u\Vert_{L^{r',q}(B_1)}.$$
This together with \eqref{eqTu} gives that 
\begin{align*}
\Vert u\Vert_{L^{p^*}(B_1)}\leq &C\big[\epsilon \Vert\nabla u\Vert_{L^p(B_1)}+C\Vert u\Vert_{L^{r',q}(B_1)}+\Vert\nabla u\Vert_{L^p(B_1)}\big]^{1-\theta}\Vert u\Vert^{\theta}_{L^{r',q}(B_1)}\\
\leq &C\Vert \nabla u\Vert^{1-\theta}_{L^p(B_1)}\Vert u\Vert_{L^{r',q}(B_1)}^{\theta}+C\Vert u\Vert_{L^{r',q}(B_1)},
\end{align*}
where $C>0$ is some constant depending on $p$. Finally, we perform the scaling argument  to obtain the desired estimate \eqref{conclusion1inlemma14maximalregularity}.
\end{proof}
A vital ingredient in the proof of Theorem \ref{thmmaximalregularity} is the following Harnack type's inequality:
\begin{lemma}\label{lemmavitalingredientmaximalregularity}
Let  $\Omega$ be a bounded Lipschitz  domain and $f\in L^{p}_{\text{loc}}(\Omega)$ for some $p\geq 1.$  Assume that $u\in W^{1,r'}_{\text{loc}}(\Omega)$ is a solution of the following equation in the sense of distributions:
\begin{align}\label{testeqinlemma15maximalregularity}
-\Delta u+|\nabla u|^{r'}\leq f,\text{~in~}\Omega\subset \mathbb R^n,
\end{align}
where $r'>1.$ 
Then for $B_R\subset \Omega,$ we have 
\begin{align*}
\int_{B_{{R/2}}}|\nabla u|^{r'}\,dx\leq K R^{n-\hat r},
\end{align*}
where $\hat r:=\max \big\{\frac{n}{p},r\big\}$ and constant $K$ depends on $r'$, $p$, $n$,  and $\Vert f\Vert_{L^p(B_R)}$.
\end{lemma}
\begin{proof}
{We refer the readers to Lemma 2.3 in \cite{goffi2023holder}.}

\end{proof}
It is necessary to establish the following improved Poincar\'{e} inequality.
\begin{lemma}\label{lemma16newpoincare}
Let $R \geq 1$ and $\gamma >1$, then  for any $v\in W^{1,\gamma}(B_R)$ satisfying
\begin{equation}\label{eq-aver}
\fint_{B_1}v\,dx=0,
\end{equation}
there exists constant  $C=C(R)>0$ such that, 
\begin{align}\label{conclusionholdsinlemma16maxmial}
\Vert v\Vert_{L^{\gamma}(B_R)}\leq C \Vert \nabla v\Vert_{L^{\gamma}(B_R)}.
\end{align}
\end{lemma}

\begin{proof}
First of all, by using the standard Poincar\'{e} inequality, one obtains
\begin{align}\label{onecanobtainlemma16maximalregular}
\Vert v\Vert_{L^{\gamma}(B_1)}\leq C\Vert\nabla v\Vert_{L^{\gamma}(B_1)},
\end{align}
where $C$ is a constant depending on $\gamma.$  Then for some $R_1>1$ which will be chosen later on, we find there exists some  $C>0$ independent of $R_1$ such that
\begin{align*}
\Vert v\Vert_{L^{\gamma}(B_{R_1})}\leq &\bigg\Vert v-\fint_{B_{R_1}}v\,dx\bigg\Vert_{L^{\gamma}(B_{R_1})}+\bigg|\fint_{B_{R_1}}v\,dx\bigg||B_{R_1}|^{\frac{1}{\gamma}}\\
\leq & CR_1\Vert\nabla v\Vert_{L^{\gamma}(B_{R_1})}+\bigg|\fint_{B_{R_1}}v\,dx-\fint_{B_1}v\,dx\bigg||B_{R_1}|^{\frac{1}{\gamma}}+\bigg|\fint_{B_1}v\,dx\bigg||B_{R_1}|^{\frac{1}{\gamma}}\\
=& CR_1\Vert\nabla v\Vert_{L^{\gamma}(B_{R_1})}+\bigg|\fint_{B_{R_1}}v\,dx-\fint_{B_1}v\,dx\bigg||B_{R_1}|^{\frac{1}{\gamma}},
\end{align*}
where we have used the condition \eqref{eq-aver}.
On the other hand , one gets
\begin{align*}
&|B_{R_1}|^{\frac{1}{\gamma}}\bigg|\fint_{B_{R_1}}v\,dx-\fint_{B_1}v\,dx\bigg|\nonumber\\
=&|B_{R_1}|^{\frac{1}{\gamma}}\bigg|\frac{1}{|B_{R_1}|}\int_{B_{R_1}}v\,dx-\frac{1}{|B_{R_1}|}\int_{B_1}v\,dx+\frac{1}{|B_{R_1}|}\int_{B_1}v\,dx-\frac{1}{|B_1|}\int_{B_1}v\,dx\bigg|\nonumber\\
\leq &\frac{1}{|B_{R_1}|^{\frac{1}{\gamma'}}}\Vert v\Vert_{L^{\gamma}(B_{R_1})}|B_{R_1}\backslash B_1|^{\frac{1}{\gamma'}}+\bigg|\frac{1}{|B_{R_1}|}-\frac{1}{|B_1|}\bigg|\Vert v\Vert_{L^{\gamma}(B_1)}|B_{R_1}|\nonumber\\
\leq &\Big(\frac{1}{2}\Big)^\frac{1}{\gamma'} \Vert v\Vert_{L^{\gamma}(B_{R_1})}+\frac{2}{|B_1|}|B_{R_1}|\Vert v\Vert_{L^{\gamma}(B_1)},
\end{align*}
where $R_1$ is chosen as $R_1=2^{\frac{1}{n}}$ such that
$$\frac{R_1^n-1}{R_1^n}\leq \frac{1}{2}.$$
 Thus, we further obtain from \eqref{onecanobtainlemma16maximalregular} that 
\begin{align*}
\Vert v\Vert_{L^{\gamma}(B_{R_1})}\leq& CR_1\Vert\nabla v\Vert_{L^{\gamma}(B_{R_1})}+\Big(\frac{1}{2}\Big)^\frac{1}{\gamma'}\Vert v\Vert_{L^{\gamma}(B_{R_1})}+\frac{2}{|B_1|}|B_{R_1}|\Vert v\Vert_{L^{\gamma}(B_1)}\\
\leq&CR_1\Vert\nabla v\Vert_{L^{\gamma}(B_{R_1})}+\Big(\frac{1}{2}\Big)^\frac{1}{\gamma'}\Vert v\Vert_{L^{\gamma}(B_{R_1})}+C|B_{R_1}|\Vert \nabla v\Vert_{L^{\gamma}(B_1)}.
\end{align*}
Let {$C(R_1):=\frac{C\max\{R_1,|B_{R_1}|\}}{1-(\frac{1}{2})^\frac{1}{\gamma'}}$}, then it follows that
\begin{align}\label{eq-aver1}
\Vert v\Vert_{L^{\gamma}(B_{R_1})}
\leq & C(R_1)\Vert \nabla v\Vert_{L^{\gamma}(B_{R_1})}.
\end{align}
Next, we let $R_2>R_1$ which will be chosen later on such that
\begin{align}
\Vert v\Vert_{L^{\gamma}(B_{R_2})}\leq &\bigg\Vert v-\fint_{B_{R_2}}v\,dx\bigg\Vert_{L^{\gamma}(B_{R_2})}+\bigg\Vert \fint_{B_{R_2}}v\,dx\bigg\Vert_{L^{\gamma}(B_{R_2})}\nonumber\\
\leq& C R_2\Vert \nabla v\Vert_{L^{\gamma}(B_{R_2})}+\bigg|\fint_{B_{R_2}}v\,dx\bigg||B_{R_2}|^{\frac{1}{\gamma}}\nonumber\\
\leq & C R_2 \Vert \nabla v\Vert_{L^{\gamma}(B_{R_2})}+\bigg|\fint_{B_{R_2}}v\,dx-\fint_{B_{R_1}}v\,dx+\fint_{B_{R_1}}v\,dx\bigg||B_{R_2}|^{\frac{1}{\gamma}}\nonumber\\
\leq&C R_2\Vert \nabla v\Vert_{L^{\gamma}(B_{R_2})}+\bigg|\fint_{B_{R_2}}v\,dx-\fint_{B_{R_1}}v\,dx\bigg||B_{R_2}|^{\frac{1}{\gamma}}+\bigg|\fint_{B_{R_1}}v\,dx\bigg||B_{R_2}|^{\frac{1}{\gamma}},\label{eq-aver2}
\end{align}
where $C>0$ is independent of $R_1$ and $R_2$.  In light of \eqref{eq-aver1} and H\"{o}lder's  inequality, one finds
\begin{align}\label{eq-aver3}
\bigg|\fint_{B_{R_1}} v\,dx\bigg||B_{R_2}|^{\frac{1}{\gamma}} 
\leq& \frac{|B_{R_2}|^{\frac{1}{\gamma}} }{|B_{R_1}|} |B_{R_1}|^{\frac{1}{\gamma'}}\|v\|_{L^{\gamma}(B_{R_1})}\leq  \frac{|B_{R_2}| }{|B_{R_1}|} \|v\|_{L^{\gamma}(B_{R_1})}.
\end{align}
Similarly, we have 
\begin{align}
&\bigg|\fint_{B_{R_2}}v\,dx-\fint_{B_{R_1}}v\,dx\bigg||B_{R_2}|^{\frac{1}{\gamma}}\nonumber\\
=&\bigg|\frac{1}{|B_{R_2}|}\int_{B_{R_2}}v\,dx-\frac{1}{|B_{R_2}|}\int_{B_{R_1}}v\,dx+\frac{1}{|B_{R_2}|}\int_{B_{R_1}}v\,dx-\frac{1}{|B_{R_1}|}\int_{B_{R_1}}v\,dx\bigg||B_{R_2}|^{\frac{1}{\gamma}}\nonumber\\
=&\bigg|\frac{1}{|B_{R_2}|^{\frac{1}{\gamma'}}}\int_{B_{R_2}\backslash B_{R_1}}v\,dx\bigg|+\bigg|\frac{1}{|B_{R_2}|}-\frac{1}{|B_{R_1}|}\bigg|\bigg|\int_{B_{R_1}}v\,dx\bigg|| B_{R_2}|^{\frac{1}{\gamma}}\nonumber\\
\leq &\Vert v\Vert_{L^{\gamma}(B_{R_2})}\frac{|B_{R_2}\backslash B_{R_1}|^{\frac{1}{\gamma'}}}{|B_{R_2}|^{\frac{1}{\gamma'}}}+\Vert v\Vert_{L^{\gamma}(B_{R_1})}\frac{2|B_{R_2}|}{|B_{R_1}|}.\label{eq-aver4}
\end{align}
Choosing  $R_2$ such that $R_2^n=2R_1^n $, Then 
$\frac{R_2^n-R_1^n}{R_2^n}\leq \frac{1}{2},$ and  it follows  from \eqref{eq-aver3} and \eqref{eq-aver4} that 
\begin{align*}
\bigg|\fint_{B_{R_2}}v\,dx-\fint_{B_{R_1}}v\,dx\bigg||B_{R_2}|^{\frac{1}{\gamma}}+\bigg|\fint_{B_{R_1}} v\,dx\bigg||B_{R_2}|^{\frac{1}{\gamma}} 
\leq &\Big(\frac{1}{2}\Big)^\frac{1}{\gamma'}\Vert v\Vert_{L^{\gamma}(B_{R_2})}+6\Vert v\Vert_{L^{\gamma}(B_{R_1})}.
\end{align*}
This combines with \eqref{eq-aver1} and \eqref{eq-aver2} indicate that, there exists $C(R_2)>0$ such that 
\begin{align*}
\Vert v\Vert_{L^{\gamma}(B_{R_2})}\leq & C(R_2)\Vert \nabla v\Vert_{L^{\gamma}(B_{R_2})}.
\end{align*}
 By performing the iteration argument, one can obtain for any $R>0$, the conclusion (\ref{conclusionholdsinlemma16maxmial}) holds.
\end{proof}

We collect the following Calder\'{o}n-Zygmund estimates for linear second order elliptic equations:
\begin{lemma}[C.f. \cite{gilbarg1977elliptic}]\label{basicellipticregulartrudinger}
Define $\Omega$ as a bounded domain.  Assume $u\in W^{2,p}_{\text{loc}}(\Omega)\cap L^p(\Omega)$ with $1<p<\infty$ be a strong solution of
$$-\Delta u=g\text{~in~}\Omega.$$
Then we have for each $B_R\subset \subset\Omega$ with $R\leq \delta$ and every $\sigma\in(0,1),$
\begin{align*}
\Vert D^2 u\Vert_{L^p(B_{{\sigma R}})}\leq &\frac{C}{(1-\sigma)^2R^2}(R^2 \Vert g\Vert_{L^p(B_R)}+\Vert u\Vert_{L^p(B_{R})}),
\end{align*}
where constants $\delta>0$ and $C=C(n,p,\delta)$.
\end{lemma}
\subsection{Weighted Morrey Norm Estimates and $W^{2,p}$ Theories}
This subsection is devoted to the local gradient estimates and $W^{2,p}$ regularity of solutions to (\ref{newsectioneq1}).  It is worthy mentioning that for the superquadratic case, i.e. $r'>2,$ the local H\"{o}lder and maximal regularities were established in \cite{cirant2022local}.  We follow the ideas shown in \cite{cirant2022local} to study the case of $1<r'\leq 2.$  More precisely, our strategy is to perform the blow-up analysis and apply the Liouville type's results to derive the contradiction, now through the analysis of weighted Morrey estimates.  In general, by using Lemma \ref{lemmavitalingredientmaximalregularity}, we can obtain that the solution $u$ of equation (\ref{newsectioneq1}) satisfies $\nabla u\in L^{r',r}(\Omega)$ provided that $p>\frac{n}{r}$.  If $r'>2$, i.e. $r'>r$, one further has {$u \in C^{0,\alpha}_{\text{loc}}(\Omega)$ for $\alpha = 2-r$} by Morrey's Lemma.  However, if $1<r'\leq 2$, we need to argue in (weighted) Morrey spaces.  To this end, we first establish the following key lemma:

\begin{lemma}\label{lemma21prop33keylemma}
Let $R>0$, $r'\geq \frac{n}{n-1}$, $g\in L^p(B_{2R})$ with $n>p>\frac{n}{r}$ and $v\in W^{2,p}(B_{2R})$ satisfy
\begin{align}\label{governedeqvsubmaximalregular}
|\Delta v|\leq C_{H}|\nabla v|^{r'}+|g|,\text{~a.e.~in~}B_{2R},
\end{align}
where $C_H>0$ is a constant.
Then if
\begin{align*}
\Vert g\Vert_{L^p(B_{2R})}+\Vert\nabla v\Vert_{L^{r',q}(B_{2R})}\leq K,
\end{align*}
where $q=r'(\frac{n}{p}-1)$, there exists $C>0$ depending on $K$ and $R$ such that 
\begin{align*}
\Vert D^2v\Vert_{L^p(B_{R})}\leq C.
\end{align*}

\end{lemma}
\begin{proof}

For any $0<\rho<2R,$ we let 
$\tilde v=v-\fint_{B_{\rho}} v\,dx$. 
Then we have from \eqref{governedeqvsubmaximalregular} that 
\begin{align*}
|\Delta \tilde v|\leq C_{H}|\nabla\tilde  v|^{r'}+|g|,\text{~a.e.~in~}B_{2R}.
\end{align*}
By employing the Poincar\'{e} inequality and H\"{o}lder's inequality, we obtain 
\begin{align}\label{byusing1inprop33new}
\Vert \tilde v\Vert_{L^p(B_{\rho})}\leq& C\rho \Vert \nabla \tilde v\Vert_{L^p(B_{\rho})}\leq C\rho^2\Vert \nabla \tilde  v\Vert_{L^{p^*}(B_{\rho})},
\end{align}
where $C>0$ is independent of $\rho>0$.

Since $\tilde v\in W^{2,p}(B_{2R})$, we apply Gagliardo-Nirenberg inequality (\ref{conclusion1inlemma14maximalregularity}) to get
\begin{align}\label{byusing2inlemmaprop33new}
\Vert \nabla \tilde v\Vert_{L^{p^*}(B_\rho)}\leq C\Vert D^2\tilde v\Vert^{\theta}_{L^p(B_\rho)}\Vert\nabla \tilde v\Vert^{1-\theta}_{L^{r',q}(B_\rho)}+\Vert \nabla \tilde v\Vert_{L^{r',q}(B_\rho)}  ~\text{ for any  }~\frac{p}{p^*}\leq\theta<1,
\end{align}
where $q=r'(\frac{n}{p}-1)$  and $C>0$ is independent of $\rho.$  From \eqref{byusing1inprop33new} and \eqref{byusing2inlemmaprop33new}, we see that there exists  $C=C(K)>0$  independent of $\rho$ such that  
\begin{align*}
\Vert \tilde v\Vert_{L^p(B_{\rho})}\leq& C\rho^2\Vert D^2\tilde v\Vert^{\theta}_{L^p(B_{2\rho})}\Vert \nabla \tilde v\Vert^{1-\theta}_{L^{r',q}(B_{
\rho})}+\rho^2\Vert \nabla \tilde v\Vert_{L^{r',q}(B_{\rho})}\\
\leq &C\rho^2\Vert D^2\tilde v\Vert_{L^p(B_\rho)}^{\theta r'}+C\rho^2,
\end{align*}
and 
\begin{align*}
\Vert \nabla \tilde v \Vert_{L^{p^*}(B_\rho)}^{r'}\leq& C\Vert D^2\tilde v\Vert_{L^p(B_\rho)}^{\theta r'}\Vert \nabla v\Vert_{L^{r',q}(B_\rho)}^{(1-\theta)r'}+C\Vert \nabla \tilde v\Vert^{r'}_{L^{r',q}(B_\rho)}\\
\leq &C(\Vert D^2\tilde v\Vert_{L^p(B_\rho)}^{\theta r'}+1),
\end{align*}
where we have used $\Vert \nabla \tilde v\Vert_{L^{r',q}(B_{\rho})}\leq K$.

We next fix $R<\rho<2R$ and apply Lemma \ref{basicellipticregulartrudinger} on equation (\ref{governedeqvsubmaximalregular}) to obtain 
\begin{align*}
\Vert D^2\tilde v\Vert_{L^p(B_{\sigma\rho})}\leq &\frac{C}{(1-\sigma)^2R^2}\bigg(R^2\Vert |\nabla \tilde v|^{r'}+g\Vert_{L^p(B_\rho)}+\Vert \tilde v\Vert_{L^p(B_{\rho})}\bigg)\\
\leq &\frac{C}{(1-\sigma)^2}\bigg(R^{r'-\frac{n}{p}(r'-1)}\Vert \nabla \tilde v\Vert_{L^{p^*}(B_\rho)}^{r'}+\frac{1}{R^2}\Vert \tilde v\Vert_{L^p(B_\rho)}+1\bigg)\\
\leq &\frac{C}{(1-\sigma)^2}\bigg(R^{r'-\frac{n}{p}(r'-1)}\Vert D^2\tilde v\Vert_{L^{p}(B_\rho)}^{\theta r'}+R^{r'-\frac{n}{p}(r'-1)}+\Vert D^2\tilde v\Vert_{L^p(B_\rho)}^{\theta r'}+1\bigg)\\
\leq &\frac{C}{(1-\sigma)^2}\bigg(\max\{1,R^{r'-\frac{n}{p}(r'-1)}\}\Vert D^2\tilde v\Vert_{L^p(B_\rho)}^{\theta r'}+1+R^{r'-\frac{n}{p}(r'-1)}\bigg)
\end{align*}
where $\sigma\in(0,1),$ and $\theta$ is chosen such that $0<\theta r'<1.$
Noting that $D^2\tilde v=D^2 v,$ we have
\begin{align*}
\Vert D^2 v\Vert_{L^p(B_{\sigma\rho})}\leq &\frac{C}{(1-\sigma)^2}\bigg(\max\{1,R^{r'-\frac{n}{p}(r'-1)}\}\Vert D^2 v\Vert_{L^p(B_\rho)}^{\theta r'}+1+R^{r'-\frac{n}{p}(r'-1)}\bigg).
\end{align*}
Now, we find if $R$ satisfies
\begin{align*}
R^{r'-\frac{n}{p}(r'-1)}\Vert D^2v \Vert_{L^p(B_\rho)}^{\theta r'}\leq 1+R^{r'-\frac{n}{p}(r'-1)},
\end{align*}
we complete the proof of our desired conclusion.  Otherwise, we arrive at
\begin{align*}
\Vert D^2v\Vert_{L^p(B_{\sigma\rho})}\leq \frac{E}{(1-\sigma)^2}\Vert D^2v\Vert_{L^p(B_{\rho})}^{\theta r'},
\end{align*}
where $0<\theta r'<1$ and $E$ is defined as
\begin{align*}
E:=C\max\Big\{1,R^{r'-\frac{n}{p}(r'-1)}\Big\},~~r'-\frac{n}{p}(r'-1)>0.
\end{align*}
Then we follow the argument below (3.6) in Proposition \cite{cirant2022local} and perform the iteration argument to complete the proof.

\end{proof}
\begin{remark}
We would like to point out that the condition $r'\geq \frac{n}{n-1}$ in Lemma \ref{lemma21prop33keylemma} is imposed to guarantee that $q\leq n$, which satisfies the range of $s$ given in the definition of Morrey space (\ref{preliminmorreydef}). 
 {When $r' < \frac{n}{n-1}$, the Morrey scale is not anymore natural, and we expect $\nabla u$ to be controlled in some H\"older space. We are not going to investigate here this kind of analysis.}
\end{remark}

On the other hand, we require the following Liouville type results:

\begin{lemma}\label{liouvillthmcontramaxi}
Assume $w\in W^{2,p}_{\text{loc}}(\mathbb R^n)$ with $p>\frac{n}{r}$ solves the following equation:
\begin{align}\label{simplfyeqmaximalregular}
-\Delta w+h|\nabla w|^{r'}=0\text{~in~}\mathbb R^n,
\end{align}
where $h\geq 0$ is a constant.  If $w$ satisfies
\begin{align}\label{andsatisfiesprop33}
\nabla w\in L^{r',q}(\mathbb R^n),~\text{with some~}
q\in(0,n],
\end{align}
then we have $w\equiv C$ with $C\in\mathbb{R}$ being some constant.
\end{lemma}
\begin{proof}
    If $h>0$, the conclusion is stated in Lemma 2.5 of \cite{cirant2022local}.  We remark that by the standard bootstrap argument, one has $w\in C^3(\mathbb R^n)$ when $h>0$ since $p>\frac{n}{r}$.  Next, we consider $h=0$ and simplify equation (\ref{simplfyeqmaximalregular}) as
    $$\Delta w=0~~\text{in~~}\mathbb R^n.$$
    
    Next,  we estimate $\nabla w (x_0)$ for arbitrary $x_0\in\mathbb{R}^n$.  To this end, define $\tilde w(x):=w(x)-w(x_0)$, then $\tilde w$ satisfies 
 $$\Delta \tilde w=0~~\text{in~~}\mathbb R^n \ \text{ and }\ \fint_{B_R(x_0)}\tilde w\,dx=\tilde w(x_0)=0\ \text{for any } R>0.$$
    By using the gradient estimates for harmonic functions, we find from H\"{o}lder's inequality, Poincar\'{e}'s inequality  and (\ref{andsatisfiesprop33}) that
\begin{align*}
|\nabla w(x_0)|=|\nabla \tilde w(x_0)|\leq \frac{C}{R^{n+1}}\Vert \tilde w\Vert_{L^1(B_R(x_0))}\leq &\frac{C}{R^{n+1}}\Vert \tilde w\Vert_{L^{r'}(B_{R}(x_0))}R^{\frac{n}{r}}\\
\leq \frac{C}{R^{n+1}}R\Vert \nabla\tilde  w\Vert_{L^{r'}(B_{R}(x_0))}R^{\frac{n}{r}}\leq &\frac{C}{R^{n-\frac{n}{r}}}\bigg(\frac{1}{R^{q-n}}\bigg)^{\frac{1}{r'}}=\frac{C}{R^\frac{q}{r'}},
\end{align*}
where $C>0$ is a constant independent of $R$.  In light of the definition of $q$ in (\ref{andsatisfiesprop33}), one obtain $
|\nabla w(x_0)|= 0$
by letting $R\rightarrow +\infty$ in the above estimate. 
Thus, we have $w\equiv C$ for some constant $C\in \mathbb{R}$ since $x_0$ is arbitrary.
\end{proof}
We are ready to show the weighted Morrey estimates satisfied by solutions of (\ref{newsectioneq1}), which is
\begin{theorem}\label{thm21zuidazhengze}
Let $p>\frac{n}{r}$ and assume $u\in W^{2,p}(\Omega)$ solves 
\begin{align*}
-\Delta u+C_H|\nabla u|^{r'}=f\text{~in~}\Omega\subset \mathbb R^n,
\end{align*}
in the strong sense, where $r'\geq \frac{n}{n-1}$ and there exists a constant $M>0$ such that  $\Vert f\Vert_{L^p(\Omega)}\leq M.$  Then  
   \begin{align*}
    \sup_{B_{2R}(\hat x)\subset \Omega}R^q\fint_{B_R(\hat x)}|\nabla u|^{r'}\,dx \ (\mathrm{dist}(B_{R}(\hat x),\partial\Omega))^{r-q}\leq C, 
    \end{align*}
    where $C=C(M,n,p,r',\Omega)$ is a positive constant.  In particular, if $p<n$, $q=r'(\frac{n}{p}-1)$; otherwise if $p\geq n,$ $q$ is any number satisfying $q\in (0,r)$.
\end{theorem}

    \begin{proof}
   To prove this theorem, we argue by contradiction and assume sequences $(f_k)_k\subset L^q(\Omega)$ and $(u_k)_k\subset W^{2,p}(\Omega) $ satisfying for every $n$,
\begin{align}\label{fromumeqkeylemma}
-\Delta u_k+C_H|\nabla u_k|^{r'}= f_k,~~x\in\Omega,
\end{align}
and $\sup_k\Vert f_k\Vert_{L^p(\Omega)}\leq M.$  Moreover,
\begin{align}\label{moreovernotingthat202443}
  \sup_{B_{2R}(\hat x)\subset \Omega}R^q\fint_{B_R(\hat x)}|\nabla u_k|^{r'}\,dx \ (\mathrm{dist}(B_{R}(\hat x),\partial\Omega))^{r-q}:=L_k\rightarrow +\infty \text{~as~}k\rightarrow +\infty.
  \end{align}
   Consequently, there exist $x_k$ and $R_k$, $k=1,2,\cdots$, such that
\begin{align*}
\frac{L_k}{2}\leq \mathrm{dist}(B_{R_k}(x_k),\partial\Omega)^{r-q}R_k^q\fint_{B_{R_k}(x_k)}|\nabla u_k|^{r'}\,dx\leq L_k,
\end{align*}
where $B_{2R_k}(x_k)\subset \Omega.$  
Define
\begin{align}\label{defwnew202443}
w_k(y):=\frac{R_k^{q/r'}}{R_kM_k^{1/r'}}u_k(x_k+R_k y),
\end{align}
where
\begin{align*}
y\in \frac{\Omega-x_k}{R_k}:=\Omega_k,~~M_k:=R_k^q\fint_{B_{R_k}(x_k)}|\nabla u_k|^{r'}\,dx.
\end{align*}
In light of $p>\frac{n}{r}$, we obtain from  Lemma \ref{lemmavitalingredientmaximalregularity} that 
\begin{equation}\label{eq-2.28}
R_k^r\fint_{B_{R_k}(x_k)}|\nabla u_k|^{r'}\,dx\leq C\end{equation}
for some constant $C>0$ independent of $k$.  Moreover, thanks to (\ref{moreovernotingthat202443}), we obtain
\begin{align}\label{rationrelationweused}
 \bigg[\frac{\text{dist}(B_{R_k}(x_k),\partial\Omega)}{R_k}\bigg]^{r-q}\geq&\frac{1}{2}\frac{L_k}{R_k^{r}\fint_{B_{R_k}(x_k)}|\nabla u_k|^{r'}\,dx}\rightarrow +\infty.
\end{align}
As a consequence, we see that  $R_k\rightarrow 0$ and $M_k\rightarrow +\infty.$ 

We next claim that for any $s>0$ and $\hat y\in \mathbb{R}^n$, the following estimate holds for $k$ large enough: 
\begin{equation}\label{stareqmaximalregular}
    \limsup_k s^q\fint_{B_s(\hat y)}|\nabla w_k(y)|^{r'}\,dy\leq C<\infty,
\end{equation}
where the  constant $C>0$ independent of $k$, $s$ and $\hat y$.  To begin with, one can see from (\ref{rationrelationweused}) that $\Omega_k\rightarrow \mathbb{R}^n$ as $k\rightarrow+\infty$, hence $\hat y\in \Omega_k$ for $k$ large enough. Noting that  $B_{R_k}(x_k)\subset \Omega$,  we further find from (\ref{rationrelationweused}) that $B_{R_ks}(x_k+R_k\hat y)\subset\Omega$, then apply the triangle inequality to get 
$$\text{dist}(B_{R_ks}(x_k+R_k\hat y),\partial\Omega)\geq \text{dist}(B_{R_k}(x_k),\partial\Omega)-(|\hat y|+s-1)R_k.$$
By using (\ref{rationrelationweused})  again,  we deduce that for $k$ large enough,
\begin{align}
s^q&\fint_{B_s(\hat y)}|\nabla w_k(y)|^{r'}\,dy=\frac{(R_ks)^q}{M_k}\fint_{B_{R_ks}(x_k+R_k\hat y)}|\nabla u_k(x)|^{r'}\,dx
\leq \frac{L_k}{M_k[{\text{dist}(B_{R_ks}(x_k+R_k\hat y),\partial\Omega)}]^{r-q}} \nonumber\\
\leq& 2 \bigg(\frac{\text{dist}(B_{R_k}(x_k),\partial\Omega)}{\text{dist}(B_{R_ks}(x_k+R_k\hat y),\partial\Omega)}\bigg)^{r-q} \leq 2 \bigg(\frac{\text{dist}(B_{R_k}(x_k),\partial\Omega)}{\text{dist}(B_{R_k}(x_k),\partial\Omega)-(|\hat y|+s-1)R_k}\bigg)^{r-q}\overset{k}{\to}2,\nonumber
\end{align}
which finishes the proof of claim \eqref{stareqmaximalregular}.


On the other hand, invoking  (\ref{defwnew202443}), we have
\begin{align}\label{finalcontrapart}
\fint_{B_1(0)}|\nabla w_k(y)|^{r'}\,dy =1.
\end{align}
Moreover, one obtains from (\ref{fromumeqkeylemma}) that $w_k$ satisfy
\begin{align}\label{inlightoflemmanewprop33}
-\Delta w_k+\frac{M_k^{{1/r}}}{R_m^{q/r-1}}|\nabla w_k|^{r'}=\frac{R_k^{q/r'+1}}{M_k^{1/r'}}f_k(x_k+R_ky):=\hat f_k,
\end{align}
where 
\begin{align}\label{eq-2.33}
\Bigg(\frac{M_k^{1/r}}{R_k^{q/r-1}}\Bigg)^{r}=\frac{M_k}{R_k^{q-r}}=R_k^r\fint_{B_{R_k}(x_k)}|\nabla u_k|^{r'}\,dx\leq C<\infty 
\end{align}
by \eqref{eq-2.28}, and it follows from  the boundness of $f$ in $L^p(\Omega)$ and $\frac{q}{r'}-\frac{n}{p}+1=0$ that 
\begin{align}\label{fmconverge0strongly}
&\Vert \hat f_k\Vert_{L^p(\Omega_k)}=\bigg\Vert \frac{R_k^{q/r'+1}}{M_k^{1/r'}}f_k(x_k+R_ky)\bigg\Vert_{L^p(\Omega_k)}\nonumber\\
=&\frac{R_k^{q/r'+1}}{M_k^{1/r'}}\bigg(\int|f_k(x_k+R_ky)|^p\,dy\bigg)^{\frac{1}{p}}=\frac{R_k^{q/r'+1}}{M_k^{1/r'}}\frac{1}{R_k^{\frac{n}{p}}}\Vert f_k\Vert_{L^p(\Omega)}\nonumber\\
=&\frac{R_k^{q/r'-\frac{n}{p}+1}}{M_k^{1/r'}}\Vert f_k\Vert_{L^p(\Omega)}=\frac{1}{M_k^{1/r'}}\Vert f_k\Vert_{L^p(\Omega)}\rightarrow 0\text{~as~}k\rightarrow +\infty.
\end{align}
In light of (\ref{inlightoflemmanewprop33}), we obtain for any $R>0$,
\begin{align}\label{eq-2.35}
|\Delta w_k|\leq \frac{M_k^{\frac{1}{r}}}{R_k^{\frac{q}{r}-1}}|\nabla w_k|^{r'}+\frac{R_k^{\frac{q}{r'}+1}}{M_k^{\frac{1}{r'}}}|f_k|,\text{~in~}B_{2R}.
\end{align}
Using (\ref{stareqmaximalregular}), \eqref{eq-2.33} and  \eqref{fmconverge0strongly}, one can invoke Lemma \ref{lemma21prop33keylemma} to get
$$\Vert D^2w_k\Vert_{L^p(B_{R})}\leq C(R)<\infty $$
for some $C(R)>0$ independent of $k.$
Moreover, we can deduce from  Lemma \ref{lemma14maximalregularity}  and \eqref{stareqmaximalregular} that 
\begin{equation}\label{eq-2.36}
  \Vert \nabla w_k\Vert_{L^{p}(B_R)}\leq C \Vert  \nabla w_k\Vert_{L^{p^*}(B_R)}\leq C\Vert D^2 w_k\Vert^{\theta}_{L^p(B_R)}\Vert \nabla w_k\Vert^{1-\theta}_{L^{r',q}(B_R)}+C\Vert \nabla  w_k\Vert_{L^{r',q}(B_R)}\leq C(R)<\infty.
\end{equation}
Without loss of generality, we may  fix $\fint_{B_1(0)}w_k\,dx=0$, and then obtain from  Lemma \ref{lemma16newpoincare} that 
\begin{align}\label{eq-2.37}
\Vert w_k\Vert_{L^{r'}(B_R)}\leq C \Vert \nabla w_k\Vert_{L^{r'}(B_R)}\leq C(R)<\infty.
\end{align}

Therefore, by using a standard diagonal argument, we can obtain from \eqref{eq-2.35}, \eqref{eq-2.36} and \eqref{eq-2.37} that 
    \begin{align*}
    w_k\rightharpoonup w_{\infty}\text{~in~}W^{2,p}_{\text{loc}}(\mathbb R^n).
\end{align*}
Thus $|\nabla w_k|\rightharpoonup |\nabla w_{\infty}|$ in $W^{1,p}_{\text{loc}}(\mathbb R^n)\hookrightarrow\hookrightarrow L^s_{\text{loc}}(\mathbb R^n)$ for all $s\in[p, p^*)$. Noting that $r'p\in(p,p^*)$ by $p>\frac{n}{r}$, we thus deduce that 
\begin{align*}
C_H|\nabla w_k|^{r'}\rightarrow h_{\infty}|\nabla w_{\infty}|^{r'}\text{~in~}L^p_{\text{loc}}(\mathbb R^n) \text{ for some }h_\infty\geq 0,
\end{align*}
where \eqref{eq-2.33} is used.

In summary, one arrives at the following limiting problem
\begin{align*}
-\Delta w_{\infty}+h_{\infty} |\nabla w_{\infty}|^{r'}=0\text{~in~}\mathbb R^n,
\end{align*}
where $w_{\infty}\in W^{2,p}_{\text{loc}}(\mathbb R^n)$ with $p>\frac{n}{r}$, and in particular, $|\nabla w_{\infty}|\in L^{r',q}(\mathbb R^n)$ followed  by  \eqref{stareqmaximalregular}. By invoking Lemma \ref{liouvillthmcontramaxi}, we find $w_{\infty}$ is a constant and thus  $\nabla w_{\infty}\equiv 0$, which contradicts (\ref{finalcontrapart}).  This completes the proof of our theorem.  
\end{proof}

With the aid of Theorem \ref{thm21zuidazhengze}, we can prove Theorem \ref{thmmaximalregularity}, which is

\vspace{2mm}

\textit{Proof of Theorem \ref{thmmaximalregularity}:}
\begin{proof}
The argument is similar as the proof of Theorem 1.3 in \cite{cirant2022local} and we give the sketch of proof here for the completeness.  

We first consider the case for $\frac{n}{r}<p<n$. Choosing  $\Omega''$ such that $\Omega'\subset\subset \Omega''\subset \subset \Omega$,   by invoking Theorem \ref{thm21zuidazhengze}, one has
$$\Vert \nabla u\Vert_{L^{r',q}(\Omega'')}\leq C,~q=r'\bigg(\frac{n}{p}-1\bigg).$$
where $C=C(M,\text{dist}(\Omega'',\partial\Omega),n,p, C_H,r)>0$ and $q<r$.  In addition, for ${\bar \Omega}'$, we have there exists a finite cover $\{B_{R}(x_k)\}_k$ such that $\bar \Omega'\subset \cup_k B_{R}(x_k)$ and $B_{2R}(x_k)\subset \Omega''$ for any $k.$  With the aid of Lemma \ref{lemma21prop33keylemma}, one obtains
$$\Vert D^2 u\Vert_{L^p(\Omega')}\leq C,$$
where $C=C(M,\text{dist}(\Omega',\partial\Omega''),n,p, C_H,r)>0$  is some constant depends on .  Moreover, thanks to Lemma \ref{lemma14maximalregularity}, we can get the gradient estimate of $u$.  By fixing the average of $u$, one finally gets the desired conclusion.  

If $p\geq n$, we pick up some $\frac{n}{r}<q<n$ at first, then follow the discussion shown above to arrive at 
$$\Vert D^2u\Vert_{L^q(\Omega')}\leq C,$$
where $C>0$ is some constant.  Next, we perform the bootstrap argument and complete the proof of this theorem since $q>\frac{n}{r}.$
\end{proof}
\begin{remark}
We remark that our results of $L^p$ estimates shown in Theorem \ref{thmmaximalregularity} also hold for $r'>2$, which covers the results of Theorem 1.3 in \cite{cirant2022local}.  Indeed, we find when $r'\geq 2$, solution $u$ still satisfies $\nabla u\in L^{r', r}_{\text{loc}}(\Omega)$ provided with $p>\frac{n}{r}.$  By establishing the key Lemma \ref{lemmavitalingredientmaximalregularity} and performing the blow-up argument, we show that $\nabla u\in L^{r',q}_{\text{loc}}(\Omega)$ with $q<r$ given in Theorem \ref{thm21zuidazhengze}, which is a higher regularity compared to the conclusion shown in Theorem 1.2 of \cite{cirant2022local}.  In addition, we give a unified argument to prove the $L^p$ maximal regularity of the solution to (\ref{newsectioneq1}) for the case of $\frac{n}{n-1}\leq r'\leq 2$ and $r'>2.$
\end{remark}
Next, we focus on the existence and asymptotic behaviors of least energy solutions to (\ref{MFG-SS}) under with critical mass exponent.

\section{Existence and Regularities: Hamliton-Jacobi and Fokker-Planck equations}\label{preliminaries}
In this section, we shall state some key lemmas and important properties satisfied by the solutions to (\ref{MFG-SS}).  First of all, we collect existence and regularity results of Hamilton-Jacobi-Bellman (HJB) equations and Fokker-Planck equations, which are summarized in Subsection \ref{subsection1} and \ref{subsection2}, respectively.
\subsection{Hamilton-Jacobi Equations}\label{subsection1}
Consider the following form of Hamilton-Jacobi equations:
\begin{align}\label{HJB-regularity}
-\Delta u_k+C_H|\nabla u_k|^{r'}+\lambda_k=V_k(x)+f_k(x),\ \ x\in\mathbb R^n,
\end{align}
where $r'>1$ is fixed, $C_H$ is a given positive  constant independent of $k$ and $(u_k,\lambda_k)$ denote the solutions to (\ref{HJB-regularity}).  Focusing on the regularities of $u_k$, one has
\begin{lemma}\label{sect2-lemma21-gradientu}
Suppose that $f_k\in L^{\infty}(\mathbb R^n)$ satisfies  $\Vert f_k\Vert_{L^\infty}\leq C_f$, $|\lambda_k|\leq \lambda$, and the potential functions $V_k(x)\in C^{0,\theta}_{\rm loc}(\mathbb R^n)$ with $\theta\in(0,1)$ satisfy  $0\leq V_k(x)\rightarrow +\infty$ as $|x|\rightarrow +\infty,$ and $\exists~ R>0$  sufficiently large such that 
\begin{align}\label{themoregeneralvkcondition}
0< C_1\leq \frac{V_k(x+y)}{V_k(x)}\leq C_2,\text{~for~all~}k\text{~and~all~}|x|\geq R \text{~with~}|y|<2,
\end{align}
where the positive constants $C_f$, $\lambda$, $R$, $C_1$ and $C_2$ are independent of $k$.  Let $(u_k,\lambda_k)\in C^2(\mathbb R^n)\times \mathbb R$ be a sequence of solutions to (\ref{HJB-regularity}).  Then,  for all $k$,
\begin{align}\label{usolutiongradientestimatepre1}
|\nabla u_k(x)|\leq C(1+V_k(x))^{\frac{1}{r'}}, \text{ for all } x\in\mathbb{R}^n,
\end{align}
where constant $C$ depends on $C_H$, $C_1$, $C_2$, $\lambda$, $r$, $n$ and $C_f.$

In particular, if there exist   $b\geq 0$ and $C_{F}>0$  independent of $k,$ such that  following conditions hold on $V_k$
\begin{align}\label{cirant-Vk}
 C_F^{-1}(\max\{|x|-C_F,0\})^b\leq V_k(x)\leq C_F(1+|x|)^b,~~\text{for all }k\text{ and }x\in\mathbb R^n,
 \end{align}
 then  we have 
\begin{align}\label{usolutiongradientestimatepre}
|\nabla u_k|\leq C(1+|x|)^{\frac{b}{r'}}, ~\text{for all }k\text{ and }x\in\mathbb R^n,
\end{align}
where constant $C$ depends on $C_H$, $C_{F}$, $b$, $\lambda$, $r$, $n$ and $C_f.$ 
\begin{proof}
The approach what we shall employ is based on Theorem 2.5 in \cite{cesaroni2018concentration}.  Indeed, when $V_k$ satisfy (\ref{cirant-Vk}), (\ref{usolutiongradientestimatepre}) hold and the arguments are stated in \cite{cesaroni2018concentration}.  

Next, we focus on the proof of (\ref{usolutiongradientestimatepre1}) for the more general $V_k$ satisfying (\ref{themoregeneralvkcondition}).  It is shown in  (2-6)  of \cite{cesaroni2018concentration} that if 
$$|-\Delta v+|\nabla v|^{r'}|\leq K~\text{in~}B_2(0)\text{~with~positive constant~} K,$$
then
\begin{align}\label{26lemmagradientestimategeneralvu}
\Vert \nabla v\Vert_{L^\gamma(B_1(0)}\leq \tilde C,~\forall \gamma\in[1,\infty],
\end{align}
where positive constant $\tilde C$ depends on $K$, $r$, $n$ and $C_H.$ For any fixed  $x_0 \in \mathbb{R}^n$, 
let $\delta=(1+V_n(x_0))^{-\frac{1}{r}}$   and define
$$v_k(y)=\delta^{\frac{2-r'}{r'-1}}u_k(x_0+\delta y),$$
then we have $v_k$ solves
\begin{align*}
-\Delta v_k+C_H|\nabla v_k|^{r'}=\delta^{r}[V_k(x_0+\delta y)-f_k(x_0+\delta y)-\lambda_k].
\end{align*}
Since $\delta>0$ is sufficiently small and $|x_0|>R$, one finds from (\ref{themoregeneralvkcondition}) that 
\begin{align*}
\delta^{r}|V_k(x_0+\delta y)-f_k(x_0+\delta y)-\lambda_k|\leq \frac{\bar C V_k(x_0+\delta y)+C_f+\lambda}{1+V_k(x_0)}\leq \hat C\text{~for~}|y|\leq 2,
\end{align*}
where the positive constants $\bar C$ and $\hat C$ are independent of $k.$  Then it follows from \eqref{26lemmagradientestimategeneralvu} that 
$$\Vert \nabla v_k(y)\Vert_{L^\infty(B_1(0))}\leq \tilde C$$ 
for some $\tilde C>0.$  In particular, choosing $y=0$, we arrive at
$$|\nabla u_k(x_0)|=\delta^{-\frac{1}{r'-1}}|\nabla v_k(0)|\leq \tilde C(1+V_k(x_0))^{\frac{1}{r'}},$$
which gives the desired estimate \eqref{usolutiongradientestimatepre1}.
In addition, noting that $u_k\in C^2(\mathbb R^n),$ we further obtain the desired conclusion (\ref{usolutiongradientestimatepre1}).
\end{proof}
\end{lemma}
Besides the gradient estimates of $u_k$, we also have the following results for the lower bounds of $u_k$: 
\begin{lemma}\label{lowerboundVkgenerallemma22}
Suppose all conditions in Lemma \ref{sect2-lemma21-gradientu} hold. 
Let $u_k$ be a family of $C^2$ solutions and assume that $u_k(x)$ are bounded from below uniformly.  Then there exist positive constants $C_3$ and $C_4$ independent of $k$ such that 
\begin{align}\label{29uklemma22}
u_k(x)\geq C_3V^{\frac{1}{r'}}_k(x)-C_4,\text{~}\forall x\in\mathbb R^n,~\text{for all }k.
\end{align}
In particular, if the following conditions hold on $V_k$
\begin{align}\label{cirant-Vk-1}
 C_F^{-1}(\max\{|x|-C_F,0\})^b\leq V_k(x)\leq C_F(1+|x|)^b,~~\text{for all }k\text{ and }x\in\mathbb R^n,
 \end{align}
 where constants $b> 0$ and $C_{F}$ are independent of $k,$ then we have 
\begin{align}\label{usolutionlowerestimatepre-11}
 u_k(x)\geq C_3|x|^{1+\frac{b}{r'}}-C_4,\text{~for all }k, x\in\mathbb R^n.
\end{align}
If $b=0$ in (\ref{cirant-Vk-1}) and there exist $R>0$ and $\hat \delta>0$ independent of $k$ such that 
\begin{align}\label{lemma22holdsbeforeconclusion}
f_k+V_k-\lambda_k>\hat \delta>0\text{~for~all }|x|>R,
\end{align}
then \eqref{usolutionlowerestimatepre-11} also holds.
\end{lemma}
\begin{proof}
When $V_k$ is assumed to satisfy (\ref{cirant-Vk-1}) or (\ref{lemma22holdsbeforeconclusion}) holds with $b=0$, the proofs are the same as in \cite{cesaroni2018concentration}, Theorem 2.6. 

Next, we focus on the case of general potential $V_k$ satisfying (\ref{themoregeneralvkcondition}).  Since $u_k$ are bounded from below uniformly, we assume that $u_k(x)\geq 0.$  Note that we only need to prove the conclusion for $|x|$ large since it can be shown straightforward if there exists $R_0>0$ independent of $k$ such that $|x|<R_0$. 
 When $|x|$ is sufficiently large, we argue by contradiction and assume up to a subsequence, there exists $|x_l|\rightarrow +\infty$ such that 
\begin{align}\label{eq2.13}
\lim_{l\rightarrow +\infty}\frac{u_{k_l}(x_l)}{V^{\frac{1}{r'}}_{k_l}(x_l)}=0.
\end{align}
Let 
$$v_l(x):=\frac{1}{\mu_l}u_{k_l}(x_l+ x), ~\text{ where}~\mu_l:=V^{\frac{1}{r'}}_{k_l}(x_l)\rightarrow+\infty.$$
Then, we have from (\ref{HJB-regularity}) that $v_l(x)$ solves
\begin{align*}
-\mu_l^{1-r'}\Delta v_l(x)+C_H|\nabla v_l(x)|^{r'}=\mu_l^{-r'}[f_{k_l}(x_l+ x)+V_{k_l}(x_l+ x)-\lambda_{k_l}].
\end{align*}
Note that 
\begin{align*}
\mu_l^{-r'}|f_{k_l}(x+x_l)-\lambda_{k_l}+V_{k_l}(x_l+x)|\geq \frac{V_{k_l}(x+x_l)-C_f-\lambda}{V_{k_l}(x_l)}\geq \delta>0,\text{~for~}|x|<2,
\end{align*}
where we have used (\ref{themoregeneralvkcondition}) and $\delta>0$ independent of $k.$
On the other hand, Lemma \ref{sect2-lemma21-gradientu} implies
$$|\nabla u_{k_l}(x)|\leq C\bigg[1+V^{\frac{1}{r'}}_{k_l}(x)\bigg].$$
Moreover, $v_l$ satisfies
\begin{align}\label{214gradientuniformboundvnew}
|\nabla v_l(x)|=\frac{1}{V^{\frac{1}{r'}}_{k_l}(x_l)}|\nabla u_{k_l}(x_l+x)|\leq \frac{C[1+V_{k_l}(x_l+x)]}{V^{\frac{1}{r'}}_{k_l}(x_l)}\leq \delta_2,~\text{for any~}|x|<2,
\end{align}
where $C>0$ and $\delta_2>0$ are constants independent of $k$.  In particular, \eqref{eq2.13} implies 
$v_l(0)=\frac{u_{k_l}(x_l)}{\mu_l}\rightarrow 0,$
and thus, $v_l(x)\leq 3\delta_2,$ $\forall |x|<2.$  Letting $l\rightarrow \infty$ and invoking (\ref{214gradientuniformboundvnew}) one applies Arzel\`{a}-Ascoli theorem to get $v_l\rightarrow v\geq 0$ uniformly in $B_2(0).$  Thus, $v$ is a solution to
\begin{align*}
|\nabla v|^{r'}\geq \delta >0 \text{~in~}B_2(0),~~v\geq 0,
\end{align*}
in a viscosity sense.  Noting that $h(x)=\delta^{\frac{1}{r'}}(2-|x|)$ is a viscosity solution to $|\nabla h|^{r'}=\delta$ in $B_2(0)$ with $h(x)=0$ on $\partial B_2(0).$  By comparison,
$$v(x)\geq h(x),\text{~for any~}x\in B_2(0).$$
As a consequence, $v(0)\geq 2\delta^{\frac{1}{r'}}>0$
which reaches a contradiction to $v_l(0)\rightarrow v(0)=0$ as $l\rightarrow \infty.$  This completes the proof of (\ref{29uklemma22}).
\end{proof}

For the existence of the classical solution to (\ref{HJB-regularity}), we have the following results:
\begin{lemma}  \label{lemma22preliminary}
Suppose $V_k+f_k$ are locally H\"{o}lder continuous and bounded from below uniformly in $k$.  Define 
\begin{align}
\bar \lambda_k:=\sup\{\lambda\in\mathbb R~|~(\ref{HJB-regularity})\text{ has a solution }u_k\in C^2(\mathbb R^n)\}.
\end{align}
Then 
\begin{itemize}
    \item[(i).] $\bar \lambda_k$ are finite for every $k$ and (\ref{HJB-regularity}) admits a solution $(u_k,\lambda_k)\in C^2(\mathbb R^n)\times \mathbb R$  with $\lambda_k=\bar \lambda_k$ and $u_k(x)$ being bounded from below (may not uniform in $k$).  Moreover,
    $$\bar \lambda_k=\sup\{\lambda\in\mathbb R~|~(\ref{HJB-regularity})\text{ has a subsolution }u_k\in C^2(\mathbb R^n)\}.$$
    \item[(ii).] If $V_k$ satisfies (\ref{cirant-Vk}) with $b>0$, then $u_k$ is unique up to constants for fixed $k$ and there exists a positive constant $C$ independent of $k$ such that 
    \begin{align}\label{lowerboundusect2}
    u_k(x)\geq C|x|^{\frac{b}{r'}+1}-C, \forall x\in\mathbb R^n.
    \end{align}
    In particular, if $V_k\equiv 0$, $b=0$ in (\ref{cirant-V}) and there exists $\sigma>0$ independent of $k$ such that 
    \begin{align}\label{verifylemma22}
    f_k-\lambda_k\geq \sigma>0, \ \ \text{for~} |x|>K_2,
    \end{align}
    where $K_2>0$ is a large constant independent of $k$, then (\ref{lowerboundusect2}) also holds.
 \end{itemize}
 (iii). If $V_k$ satisfies \eqref{V2mainassumption_2} with $V$ replaced by $V_k$ and positive constants $C_1$, $C_2$ and $\delta$ independent of $k,$ then there exist uniformly bounded from below classical solutions $u_k$ to problem (\ref{HJB-regularity}) satisfying estimate (\ref{29uklemma22}).
\end{lemma}
\begin{proof}
The proof is the same as Theorem 2.7 shown in \cite{cesaroni2018concentration}.
\end{proof}

It is worthy mentioning that if locally H\"{o}lder continuous potential functions $V_k$ satisfies  $C_1e^{\delta |x|}\leq V_k\leq C_2e^{\delta|x|}$ for some $C_1$, $C_2>0$ independent of $k,$ (\ref{V2mainasumotiononv}) with $V$ replaced by $V_k$ also holds for $V_k$.  With a-priori estimates and existence results of solutions to HJ equations given by (\ref{HJB-regularity}), we next discuss the regularities of solutions to Fokker-Planck equations, which is exhibited in Subsection \ref{subsection2}.

\subsection{Fokker-Planck Equations}\label{subsection2}
Before stating the gradient estimates satisfied by solutions to Fokker-Planck equations, we recall the following key lemma for any function $m\in L^p(\mathbb R^n)$:
\begin{lemma}\label{proposition-lemma21-FP}
 Suppose $p>1$ and $m\in L^p(\mathbb R^n)$ such that
 \begin{align*}
 \Big|\int_{\mathbb R^n}m\Delta\varphi\,dx\Big|\leq N\Vert\nabla \varphi\Vert_{L^{p'}(\mathbb R^n)}\ \ \text{~for~all~}\varphi\in C_c^{\infty}(\mathbb R^n),
 \end{align*}
 where $N>0$ is a positive constant.  Then we have $m\in W^{1,p}(\mathbb R^n)$ and 
 \begin{align*}
 \Vert\nabla m\Vert_{L^p(\mathbb R^n)}\leq C_pN,
 \end{align*}
 where $C_p$ is a positive constant depending only on $p.$
\end{lemma}
\begin{proof}
See Proposition 2.4 in \cite{cesaroni2018concentration}.
\end{proof}
Now, we are concerned with the following Fokker-Planck equations:
\begin{align}\label{sect2-FP-eq}
-\Delta m+\nabla\cdot w=0,\ \ x\in\mathbb R^n,
\end{align}
where $w$ is given and $m$ denotes the solution.  By invoking Lemma \ref{proposition-lemma21-FP}, we can obtain the crucial a-priori estimates satisfied by $m$.  To begin with, we recall that $\hat q$ is defined as \eqref{hatqconstraint}, and set 
${\hat q}^*=\frac{n\hat q}{n-\hat q}$ if $\hat q<n$, and ${\hat q}^*=+\infty$ if $\hat q\geq n$. Choose
 $\beta\in [\hat q,{\hat q^*}]$ such that 
\begin{equation}\label{eq2-160}
 \frac{1}{\hat q}=\frac{1}{r}+\frac{1}{r'\beta}.\end{equation}
Then one can deduce from \eqref{hatqconstraint} that 
\begin{equation}\label{beta}
\beta=\begin{cases}{\hat q}^*, \ &\text{ if }r<n,\\
\in(\hat q,{\hat q}^*) \ &\text{ if }r=n,\\
\infty,&\text{ if }r>n.\end{cases}
\end{equation}
Set 
\begin{equation}\label{Gag}
0<\mathcal{S}_{\hat q,r}^{-1}:=\inf_{m\in W^{1, \hat q}(\mathbb R^n)} \frac{\|\nabla m\|^\theta_{L^{\hat q}(\mathbb R^n)}\|\nabla m\|^{1-\theta}_{L^{\hat q}(\mathbb R^n)}}{\|m\|_{L^{\beta}(\mathbb R^n)}}<\infty,\text{ where }\theta\in[0,1] \text{ satisfying }\frac{1}{\beta}=\theta(\frac{1}{{\hat q}}-\frac{1}{n})+1-\theta.
\end{equation}
Then we have the following lemma which addresses the regularity of solutions for equation \eqref{sect2-FP-eq}.
\begin{lemma}\label{lemma21-crucial}
Assume $(m,w)\in \left(L^1(\mathbb R^n)\cap W^{1, \hat q}(\mathbb R^n)\right)\times L^1(\mathbb R^n)$ is a  solution to (\ref{sect2-FP-eq}) and
\begin{equation*}
\Lambda_r:=\int_{\mathbb R^n}|m|\Big|\frac{w}{m}\Big|^{r}\, dx<\infty.
\end{equation*}
Then, we have $w\in L^{1}(\mathbb R^n)\cap L^{\hat q}(\mathbb R^n)$ and there exists constant $\mathcal{C}=\mathcal{C}(\Lambda_r,\|m\|_{L^1(\mathbb R^n)})>0$ such that
$$\|m\|_{W^{1,\hat q}(\mathbb R^n)}, \|w\|_{L^1(\mathbb R^n)},\|w\|_{L^{\hat q}(\mathbb R^n)}\leq \mathcal{C}.$$ More precisely, we have
\begin{equation}\label{eq240}
\|\nabla m\|_{L^{\hat q}(\mathbb R^n)}\leq \mathcal{S}_{\hat q,r}^{\frac{1}{r'-\theta}}\left(C_{\hat q} \Lambda_r^\frac{1}{r}\right)^{\frac{r'}{r'-\theta}}\|m\|_{L^1(\mathbb R^n)}^\frac{1-\theta} {r'-\theta},~
~
\|m\|_{L^{\hat q}(\mathbb R^n)}\leq  \mathcal{S}_{\hat q,r}^{\frac{1}{r'-\theta}} \left(C_{\hat q} \Lambda_r^\frac{1}{r}\right)^{\frac{\theta}{r'-\theta}}\|m\|_{L^1(\mathbb R^n)}^{\frac{1-\theta} {r'-\theta}+\frac{1}{r}}.
\end{equation}
and 
\begin{equation}\label{eq2.31}
\|w\|_{L^1(\mathbb R^n)}\leq  \Lambda_r^\frac{1}{r}\|m\|_{L^1(\mathbb R^n)}^\frac{r-1}{r},~~\|w\|_{L^{\hat q}(\mathbb R^n)}\leq \Lambda_r^\frac{1}{r} \left(\mathcal{S}_{\hat q,r}\right)^{\frac{1}{r'-\theta}}\left(C_{\hat q} \Lambda_r^\frac{1}{r}\right)^{\frac{\theta}{r'-\theta}}\|m\|_{L^1(\mathbb R^n)}^\frac{1-\theta} {r'-\theta},
\end{equation}
where 
\begin{equation}\label{eq2.201}
\theta=\frac{nr(\hat q-1)}{(r-1)(nq-n+q)}=\begin{cases} 1, \ &\text{ if }r<n,\\
\frac{n^2(\hat q-1)}{(n-1)(n\hat q-n+\hat q)} \ &\text{ if }r=n,\\
\frac{nr}{nr-n+r},&\text{ if }r>n.
\end{cases}
\end{equation}

\end{lemma}

\begin{proof}
We refer the readers to Lemma 2.8 in \cite{cesaroni2018concentration} and Proposition 2.5 in \cite{cirant2023ergodic}.  For the sake of completeness, we give the proof of this lemma as follows:

Let $\beta\in [\hat q,{\hat q^*}]$ satisfy \eqref{eq2-160} and \eqref{beta}.
Then we deduce from (\ref{sect2-FP-eq}) that
\begin{equation}\label{eq2-16}
\Big|\int_{\mathbb R^n}\nabla m \cdot \nabla\varphi dx\Big|=\Big|\int_{\mathbb R^n}w \cdot\nabla\varphi dx\Big|\leq \Lambda_r^\frac{1}{r}\|m\|^\frac{1}{r'}_{L^\beta(\mathbb R^n)}\|\nabla\varphi\|_{L^{\hat q'}(\mathbb R^n)}\text{~for~all~}\varphi\in C_c^{\infty}(\mathbb R^n).
\end{equation}
In view of Lemma \ref{proposition-lemma21-FP}, one has there exists $C_{\hat q}>0$ such that
\begin{equation}\label{eq2-17}
\|\nabla m\|_{L^{\hat q}(\mathbb R^n)}\leq  C_{\hat q} \Lambda_r^\frac{1}{r}\|m\|^\frac{1}{r'}_{L^\beta(\mathbb R^n)}.
\end{equation}
Moreover, we apply  \eqref{Gag} to get
\begin{equation}\label{eq2-19}
\|m\|_{L^\beta(\mathbb R^n)}\leq \mathcal{S}_{\hat q,r}\|\nabla m\|_{L^{\hat q}(\mathbb R^n)}^\theta\|m\|_{L^1(\mathbb R^n)}^{1-\theta} \text{ where }\frac{1}{\beta}=\theta(\frac{1}{{\hat q}}-\frac{1}{n})+1-\theta.
\end{equation}
In light of \eqref{eq2-160}, one obtains
\begin{equation}\label{eq2.20}
\theta=\frac{nr(\hat q-1)}{(r-1)(nq-n+q)}
\end{equation}

Invoking \eqref{eq2-17} and \eqref{eq2-19}, we find the following inequalities hold:
\begin{equation}\label{eq2.22}
\|m\|_{L^\beta(\mathbb R^n)}\leq \left(\mathcal{S}_{\hat q,r}\right)^{\frac{r'}{r'-\theta}} \left(\mathcal{S}_{\hat q,r}C_{\hat q} \Lambda_r^\frac{1}{r}\right)^{\frac{r'\theta}{r'-\theta}}\|m\|_{L^1(\mathbb R^n)}^\frac{(1-\theta)r'} {r'-\theta}
\end{equation}
and
\begin{equation}\label{240eq-in-priori-FP}
\|\nabla m\|_{L^{\hat q}(\mathbb R^n)}\leq \mathcal{S}_{\hat q,r}^{\frac{1}{r'-\theta}}\left(C_{\hat q} \Lambda_r^\frac{1}{r}\right)^{\frac{r'}{r'-\theta}}\|m\|_{L^1(\mathbb R^n)}^\frac{1-\theta} {r'-\theta}.
\end{equation}
Then letting $\tau:=\frac{1-\frac{1}{\hat q}}{1-\frac{1}{\beta}}=\frac{1}{r'}$ since \eqref{eq2-160} holds, we apply H\"older's inequality to obtain
\begin{equation}\label{241eq-in-priori-FP}
\|m\|_{L^{\hat q}(\mathbb R^n)}\leq \|m\|_{L^\beta(\mathbb R^n)}^{\tau}\|m\|_{L^1(\mathbb R^n)}^{1-\tau}\leq  \mathcal{S}_{\hat q,r}^{\frac{1}{r'-\theta}} \left(C_{\hat q} \Lambda_r^\frac{1}{r}\right)^{\frac{\theta}{r'-\theta}}\|m\|_{L^1(\mathbb R^n)}^{\frac{1-\theta} {r'-\theta}+\frac{1}{r}}.
\end{equation}
Moreover, using \eqref{hatqconstraint}  and \eqref{eq2.20} again, we obtain \eqref{eq2.201}

From \eqref{240eq-in-priori-FP} and  \eqref{241eq-in-priori-FP} we obtain \eqref{eq240}.

Now, we focus on the estimates of $w.$ 
 Noting that for any $\nu\in[1,\hat q]$, we have $\frac{r}{\nu}>\frac{r}{\hat q}>1$.  Then, by H\"{o}lder's inequality,
\begin{equation*}
\int_{\mathbb R^n}|w|^\nu dx=\int_{\mathbb R^n}|w|^\nu |m|^{-\frac{(r-1)\nu}{r}} |m|^{\frac{(r-1)\nu}{r}}dx\leq \left(\int_{\mathbb R^n}|m|\Big|\frac{w}{m}\Big|^{r}\, dx \right)^\frac{\nu}{r}\left(\int_{\mathbb R^n}|m|^{\frac{r-1}{r-\nu}\nu}dx\right)^\frac{r-\nu}{r},
\end{equation*}
which implies
\begin{equation}\label{eq2.29}
\|w\|_{L^\nu(\mathbb R^n)}\leq  \Lambda_r^\frac{1}{r}\|m\|_{L^{\frac{r-1}{r-\nu}\nu}(\mathbb R^n)}^\frac{r-1}{r} \text{ for all }\nu\in[1,\hat q].
\end{equation}
Choosing $\nu=\hat q$, it follows from   \eqref{eq2-160} and \eqref{eq2.22} that
\begin{equation}\label{eq2.30}
\|w\|_{L^{\hat q}(\mathbb R^n)}\leq  \Lambda_r^\frac{1}{r}\|m\|_{L^{\beta}(\mathbb R^n)}^\frac{r-1}{r}\leq \Lambda_r^\frac{1}{r} \left(\mathcal{S}_{\hat q,r}\right)^{\frac{1}{r'-\theta}}\left(C_{\hat q} \Lambda_r^\frac{1}{r}\right)^{\frac{\theta}{r'-\theta}}\|m\|_{L^1(\mathbb R^n)}^\frac{1-\theta} {r'-\theta}.
\end{equation}
Taking $\nu=1$ in \eqref{eq2.29}, together with \eqref{eq2.30}, we obtain   \eqref{eq2.31}.  We complete the proof of this lemma.

\end{proof}

By the same arguments of \eqref{eq2-160}, \eqref{eq2-16} and \eqref{eq2-17}, we  have the following corollary:
\begin{corollary}\label{lemma21-crucial-cor}
Assume that $(m,w)\in (L^1(\mathbb R^n)\cap L^{1+\alpha}(\mathbb R^n)\cap W^{1, q}(\mathbb R^n))\times L^1(\mathbb R^n)$ be the solution to (\ref{sect2-FP-eq}) with
$$\frac{1}{ q}=\frac{1}{r}+\frac{1}{r'(1+\alpha)}.$$
Then for $\alpha\in(0,\frac{r}{n}\big]$, there exists a positive constant $C$ depending only on $n$ and $\alpha$ such that
\begin{align}\label{lemma24eq28w1q}
\Vert \nabla m\Vert_{L^{q}(\mathbb R^n)}\leq C\Big(m\Big|\frac{w}{m}\Big|^{r}\, dx\Big)^\frac{1}{r}\Vert m\Vert_{L^{1+\alpha}}^{\frac{1}{r'}}.
\end{align}
Moreover, there exists a positive constant $C$ only depending on $r,$ $n$ and $\alpha$ such that   
\begin{align}
\Vert m\Vert^{1+\alpha}_{L^{1+\alpha}(\mathbb R^n)}\leq  C\bigg(\int_{\mathbb R^n}m\,dx\bigg)^{\frac{(\alpha+1)r-n\alpha}{r}}\bigg(\int_{\mathbb R^n}m\Big|\frac{w}{m}\Big|^{r}\,dx\bigg)^{\frac{n\alpha}{r}}.
\end{align}
\end{corollary}

Next, we turn our attention to system (\ref{MFG-SS}), a coupled system consisting of a HJ equation and a Fokker-Planck Equation.  Indeed, with some assumptions imposed on population density $m$ and Lagrange multiplier $\lambda,$ we have the following lemma for the decay property of $m$:

\begin{lemma}[ C.f. Proposition 5.3 in \cite{cesaroni2018concentration} ]\label{mdecaylemma}
Assume that $(u,\lambda, m)\in C^2(\mathbb R^n)\times \mathbb R\times \big(W^{1,p}(\mathbb R^n)\cap L^1(\mathbb R^n)\big)$ with $u$ bounded from below, $ p>n$ and $\lambda<0$ is the solution of the following potential-free problem  
\begin{align}\label{26preliminaryfinal}
\left\{\begin{array}{ll}
-\Delta u+C_H|\nabla u|^{r'}+\lambda=-m^{\alpha}, &x\in\mathbb R^n,\\
\Delta m+C_Hr'\nabla\cdot(m|\nabla u|^{r'-2}\nabla u )=0, &x\in\mathbb R^n.
\end{array}
\right.
\end{align}
 Then, there exist $\kappa_1,\kappa_2>0$ such that 
\begin{align}\label{exponentialdecaym}
m(x)\leq \kappa_1 e^{-\kappa_2|x|}  ~\text{ for all } x\in \mathbb R^n.
\end{align}
\end{lemma}
\begin{proof}
Since $p>n$, we see that  $m\in W^{1,p}(\mathbb R^n)\hookrightarrow C^{0,\theta}(\mathbb R^n)$ for some $\theta\in(0,1)$, which indicates   $m\rightarrow 0$ as $|x|\rightarrow +\infty$.  In addition, noting $-\lambda>0,$ we obtain
\begin{align}\label{eq2.45}
\liminf_{|x|\rightarrow \infty}(-m^{\alpha}-\lambda)\geq -\frac{\lambda}{2}>0,
\end{align}
which satisfies (\ref{verifylemma22}).  Now, we fix $u(0)=0\leq u(x)$ for $x\in\mathbb R^n$ and deduce from \eqref{usolutiongradientestimatepre} with $b=0$ that 
\begin{align}\label{29ulowerboundpotentialfree}
|\nabla u(x)|\leq C_1, \ \ x\in\mathbb R^n ~\text{ for some }~C_1>0.
\end{align}
 To show (\ref{exponentialdecaym}), we consider the Lyapunov function $\Phi(x)=e^{\kappa u}$ with $0<\kappa<-\frac{\lambda}{4}$.  By using the $u$-equation in (\ref{26preliminaryfinal}) we obtain from \eqref{eq2.45} and \eqref{usolutiongradientestimatepre} that, $\exists \ R>0$ large enough such that for $|x|>R$, 
\begin{equation*}
\begin{split}
-\Delta\Phi+C_Hr'|\nabla u|^{r'-2}\nabla u \cdot\nabla\Phi&=\kappa(C_H(r'-1)|\nabla u|^{r'}-\lambda-\kappa|\nabla u|^2-m^{\alpha})\Phi\\
&\geq \kappa(-\lambda-\kappa|\nabla u|^2-m^{\alpha})\Phi\geq -\frac{\kappa\lambda}{4} \Phi.
\end{split}
\end{equation*}
Then by using (\ref{29ulowerboundpotentialfree}), as shown in \cite{cesaroni2018concentration}, we finish the proof of (\ref{exponentialdecaym}).
\end{proof}
We next collect the Pohozaev identities satisfied by the solution to (\ref{26preliminaryfinal}) in the following lemma: 
\begin{lemma}[C.f. Proposition 3.1 in \cite{cirant2016stationary}]\label{poholemma}
Let  $(u,\lambda, m)$ satisfy the assumptions of Lemma  \ref{mdecaylemma} and denote $w=-C_Hr'm|\nabla u|^{r'-2}\nabla u$. 
 Then the following identities hold:
\begin{align}\label{eq2.49}
\left\{\begin{array}{ll}
\lambda\int_{\mathbb R^n}m\, dx=-\frac{(\alpha+1)r-n\alpha}{(\alpha+1)r}\int_{\mathbb R^n}m^{\alpha+1}\,dx,\\
C_L\int_{\mathbb R^n}m\big|\frac{w}{m}\big|^r\, dx=\frac{n\alpha}{(\alpha+1)r}\int_{\mathbb R^n}m^{\alpha+1}\, dx=(r'-1)C_H\int_{\mathbb R^n} m|\nabla u|^{r'}\, dx.
\end{array}
\right.
\end{align}

\end{lemma}
\begin{proof}
From Lemma \ref{mdecaylemma} we see that  $m\leq \kappa_1e^{-\kappa_2 |x|}$ is exponential decay.  In addition, one can obtain from  \eqref{29ulowerboundpotentialfree} that there exists $R>0$ such that $|u|\leq C|x|$ for  $|x|>R$.  It is necessary to mention that if $1<r'<2$, the Fokker-Planck equation holds in the weak sense.  In this case, we take the approximation argument and let $H_{\epsilon}(p):=C_H(\epsilon+|p|^2)^{\frac{r'}{2}}$ to approximate $H$ given by (\ref{MFG-H}).  After performing the computations on $m_{\epsilon}$, we take the limit $\epsilon\rightarrow 0$ to obtain our desired conclusion.

 We test the $u$-equation in (\ref{26preliminaryfinal}) against $m$ and integrate it by parts to obtain
\begin{align}\label{sect3-eq1-poho}
\int_{\mathbb R^n}\nabla u\cdot\nabla m\, dx+C_H\int_{\mathbb R^n}|\nabla u|^{r'}m\, dx+\lambda \int_{\mathbb R^n}m\,dx=-\int_{\mathbb R^n}m^{\alpha+1}\,dx.
\end{align}
Similarly, multiplying the $m$-equation in (\ref{eq-attained-sub}) by $u$, we integrate it to find
\begin{align}\label{sect3-eq2-poho}
\int_{\mathbb R^n}\nabla u\cdot\nabla m\,dx=-C_Hr'\int_{\mathbb R^n} m|\nabla u|^{r'}\,dx.
\end{align}
Subtracting (\ref{sect3-eq1-poho}) from \eqref{sect3-eq2-poho}, we arrive at
\begin{align}\label{sect3-combine11}
(1-r')C_H\int_{\mathbb R^n}m|\nabla u|^{r'}\,dx+\lambda\int_{\mathbb R^n}m\,dx =-\int_{\mathbb R^n}m^{\alpha+1}\, dx.
\end{align}

We next prove that 
\begin{align}\label{pohoequation2}
- n\lambda\int_{\mathbb R^n}m\,dx-\frac{n }{\alpha+1}\int_{\mathbb R^n}m^{\alpha+1}\,dx+C_H\frac{n-r}{r-1}\int_{\mathbb R^n} m|\nabla u|^{r'}\,dx=0,
\end{align}
 Multiply the $u$-equation in (\ref{26preliminaryfinal}) by $\nabla m\cdot x$ and integrate it by parts to get
\begin{align}\label{pohocombine1}
\int_{\mathbb R^n}(-m^{\alpha}-\lambda)\nabla m\cdot x\,dx=&-\int_{\mathbb R^n}\Delta u(\nabla m\cdot x)\,dx+{C_H}\int_{\mathbb R^n}|\nabla u|^{r'}(\nabla m\cdot x)\,dx\nonumber\\
=&\overbrace{\int_{\mathbb R^n}\nabla u\cdot \nabla(\nabla m\cdot x)\,dx}^{I_1}-{C_H}\int_{\mathbb R^n}\nabla\cdot(|\nabla u|^{r'}x)m\,dx.
\end{align}
Test the $m$-equation in (\ref{26preliminaryfinal}) against $\nabla u\cdot x$, then we use the integration by parts to obtain
\begin{align}\label{pohocombine2}
-{C_H}\int_{\mathbb R^n}\nabla(|\nabla u|^{r'})\cdot xm\,dx=\int_{\mathbb R^n}\nabla m\cdot \nabla(\nabla u\cdot x)\,dx+C_Hr'\int_{\mathbb R^n}|\nabla u|^{r'}m\,dx,
\end{align}
where we have used
$${C_H}\nabla(|\nabla u|^{r'})\cdot x=C_Hr'|\nabla u|^{r'-2}u_{x_i}u_{x_ix_j}x_j=C_Hr'|\nabla u|^{r'-2}\nabla u\cdot\nabla(\nabla u\cdot x)-r'C_H|\nabla u|^{r'}.$$

For $I_1$, we have from the integration by parts that 
\begin{align}\label{pohocombine3}
\int_{\mathbb R^n}\nabla u\cdot \nabla(\nabla m\cdot x)\,dx=&\int_{\mathbb R^n} u_{x_i}m_{x_ix_j}x_j\,dx+\int_{\mathbb R^n}\nabla u\cdot \nabla m\,dx\nonumber\\
=&-\int_{\mathbb R^n}m_{x_i}u_{x_ix_j}x_j\,dx+(1-n)\int_{\mathbb R^n}\nabla u\cdot\nabla m\,dx\nonumber\\
=&-\int_{\mathbb R^n} \nabla m\cdot\nabla(\nabla u\cdot x)\,dx+(2-n)\int_{\mathbb  R^n}\nabla u\cdot \nabla m\,dx,
\end{align}
Combining (\ref{pohocombine1}), (\ref{pohocombine2}) and (\ref{pohocombine3}), one finds
\begin{align}\label{pohofinalbefore}
\overbrace{\int_{\mathbb R^n}(-m^{\alpha}-\lambda)\nabla m\cdot x\,dx}^{I_2}=C_H\big(r'-{n}\big)\int_{\mathbb R^n}|\nabla u|^{r'}m\,dx+(2-n)\int_{\mathbb R^n}\nabla u\cdot \nabla m\,dx.
\end{align}
For $I_2$, we integrate by parts again to find
\begin{align}\label{pohofinalbefore2}
I_2=n\int_{\mathbb R^n}\Big(\frac{1}{\alpha+1}m^{\alpha+1}+\lambda m\Big)\,dx.
\end{align}
By using (\ref{pohofinalbefore}) and (\ref{pohofinalbefore2}), we have shown 
\begin{align*}
- n\lambda\int_{\mathbb R^n}m\,dx-\frac{n }{\alpha+1}\int_{\mathbb R^n}m^{\alpha+1}\,dx+ C_H\Big(r'-{n}\Big)\int_{\mathbb R^n} |\nabla u|^{r'}m\,dx+(2-n)\int_{\mathbb R^n}\nabla u\cdot \nabla m\,dx=0,
\end{align*}
This together with \eqref{sect3-eq2-poho} indicates 
(\ref{pohoequation2}).  Since $w=-C_Hr'm|\nabla u|^{r'-2}\nabla u$ and $C_L=\frac{1}{r}(r'C_H)^{\frac{1}{1-r'}}$, one  gets
\begin{align}\label{sect3-combine200}
C_L\int_{\mathbb R^n}m\bigg|\frac{w}{m}\bigg|^{r}\, dx=C_L (C_Hr')^{r}\int_{\mathbb R^n}m|\nabla u|^{r'}\,dx=(r'-1)C_H\int_{\mathbb R^n}m|\nabla u|^{r'}\,dx.
\end{align}
Finally, \eqref{eq2.49} follows directly  from \eqref{sect3-combine11}, (\ref{pohoequation2}) and \eqref{sect3-combine200}. 

\end{proof}
Now, we are ready to show Theorem \ref{thm11-optimal}, the attainability of problem (\ref{sect2-equivalence-scaling}). 
 We would like to recall that minimization problem (\ref{sect2-equivalence-scaling}) is equivalent to (\ref{optimal-inequality-sub}).
\section{Gagliardo-Nirenberg Type Inequality: Potential-free MFGs}\label{sect3-optimal}
This section is devoted to the existence of minimizers to problem (\ref{optimal-inequality-sub}).  Before studying the mass critical case, we consider the case of $\alpha\in(0,\frac{r}{n})$ and recall that with this condition, Cirant et al. \cite{cesaroni2018concentration} in Theorem 1.3 showed for any $M>0,$ the following minimization problem is attained by pair $( \bar m_{\alpha,M},\bar w_{\alpha,M}),$
\begin{align}\label{energy-epsilon0}
e_{0,\alpha,M}:=\inf\limits_{(m,w)\in {\mathcal A}_M}\mathcal E_0(m,w) ~\text{ where }~\mathcal E_0(m,w)=\int_{\mathbb R^n}\Bigg(C_Lm\bigg|\frac{w}{m}\bigg|^r-\frac{1}{\alpha+1}m^{\alpha+1}\Bigg) \,dx.
\end{align}
Moreover, $ \bar m_{\alpha,M}\in W^{1,p}(\mathbb R^n)$  $\forall \ p\geq 1$ satisfies  
\begin{equation}\label{eq3.2}
    0<\bar m_{\alpha,M}<c_{1,M}e^{-c_{2,M}|x|} \text{ for some }c_{1,M}, c_{2,M}>0.
\end{equation}
and 
\begin{equation}\label{eq3.3}
    \text{$\exists\ \bar u_{\alpha,M}\in C^2(\mathbb R^n)$ bounded from below s. t. $\bar w_{\alpha,M}=-C_Hr'\bar m_{\alpha,M}|\nabla\bar u_{\alpha,M}|^{r'-2}\nabla\bar u_{\alpha,M}$. }
\end{equation}
  In addition, $(\bar m_{\alpha,M},\bar u_{\alpha,M})$  solves the following equation with some $\lambda_{\alpha,M}<0$
\begin{align}\label{eq-attained-sub}
\left\{\begin{array}{ll}
-\Delta u+C_H|\nabla u|^{r'}+\lambda=-m^{\alpha},&x\in\mathbb R^n,\\
\Delta m+C_Hr'\nabla\cdot(m|\nabla u|^{r'-2}\nabla u)=0,&x\in\mathbb R^n,\\
\int_{\mathbb R^n}m\, dx=M,
\end{array}
\right.
\end{align}
and  from Lemma \ref{poholemma} we have the following Pohozaev identities
\begin{align}\label{sect3-poho-final}
\left\{\begin{array}{ll}
\lambda\int_{\mathbb R^n}\bar m_{\alpha,M}\, dx=-\frac{(\alpha+1)r-n\alpha}{(\alpha+1)r}\int_{\mathbb R^n}\bar m_{\alpha,M}^{\alpha+1}\,dx,\\
C_L\int_{\mathbb R^n}\bar m_{\alpha,M}\big|\frac{\bar w_{\alpha,M}}{\bar m_{\alpha,M}}\big|^r\, dx=\frac{n\alpha}{(\alpha+1)r}\int_{\mathbb R^n}\bar m_{\alpha,M}^{\alpha+1}\, dx=(r'-1)C_H\int_{\mathbb R^n} \bar m_{\alpha,M}|\nabla \bar u_{\alpha,M}|^{r'}\, dx.
\end{array}
\right.
\end{align}

We next show that $(\bar m_{\alpha,M},\bar w_{\alpha,M})$, the minimizer of \eqref{energy-epsilon0}, is also a minimizer of problem (\ref{optimal-inequality-sub}).
\begin{lemma}\label{sect3-lemma32}
For any fixed $\alpha\in\big(0,\frac{r}{n}\big)$ and $M>0$, problem \eqref{sect2-equivalence-scaling} is attained by $(\bar m_{\alpha,M},\bar w_{\alpha,M})$ with $e_{0,\alpha,M}=\mathcal E_0(\bar m_{\alpha,M},\bar w_{\alpha,M})$.  Moreover, we have
\begin{align}\label{sect3-relation-Calpha-emalpha}
\Gamma_{\alpha}=\frac{n\alpha (-e_{0,\alpha,M})^{\frac{n\alpha-r}{r}}M^{\frac{(\alpha+1)r-n\alpha}{r}}}{r(1+\alpha)}\bigg(\frac{r-n\alpha}{n\alpha}\bigg)^{\frac{r-n\alpha}{r}}.
\end{align}
\end{lemma}
\begin{proof}
For simplicity, we denote
\begin{align}\label{sect3-Galpha-39}
G_{\alpha}(m,w):=\frac{\Big(C_L\int_{\mathbb R^n}m\big|\frac{w}{m}\big|^{r}\, dx\Big)^{\frac{n\alpha}{r}}\Big(\int_{\mathbb R^n}m\,dx\Big)^{\frac{(\alpha+1)r-n\alpha}{r}}}{\int_{\mathbb R^n}m^{\alpha+1}\, dx},
\end{align}
and  because of (\ref{optimal-inequality-sub}), we rewrite minimization problem \eqref{energy-epsilon0} as
\begin{align}\label{sect3-Calpha-equivalent}
\Gamma_{\alpha}=\inf_{(m,w)\in{{\mathcal A}}_M}G_{\alpha}(m,w).
\end{align}

To show (\ref{sect3-Calpha-equivalent}) is attained by $(\bar m_{\alpha,M},\bar w_{\alpha,M})$, we first analyze the lower bound of $\mathcal E_0$ defined by (\ref{energy-epsilon0}).  For  this purpose, we first note that  $\mathcal E_0(m,w)=G_\alpha(m,w)=\infty$ provided that $\int_{\mathbb R^n}m\big|\frac{w}{m}\big|^{r}\, dx=+\infty$. Therefore, we just need to consider the case that $(m,w)\in{{\mathcal A}}_M$ satisfying $\int_{\mathbb R^n}m\big|\frac{w}{m}\big|^{r}\, dx<\infty$.
We define $(m_{\mu}(x),w_{\mu}(x))=(\mu^nm(\mu x),\mu^{n+1}w(\mu x))$ for $\mu\in\mathbb R^+\backslash\{0\}$, then substitute the pair into $\mathcal E_0$ to obtain
\begin{align}\label{sect3-311-mathcalE0}
\mathcal E_0(m_{\mu},w_{\mu})=&\mu^r\int_{\mathbb R^n}C_Lm\Big|\frac{w}{m}\Big|^{r}\, dx-\frac{\mu^{n\alpha}}{\alpha+1}\int_{\mathbb R^n}m^{1+\alpha}\, dx\nonumber\\
\geq&-\bigg(\frac{n\alpha}{r}\bigg)^{\frac{r}{r-n\alpha}}\bigg[\frac{r-n\alpha}{n\alpha}\bigg]\bigg(\frac{1}{\alpha+1}\int_{\mathbb R^n}m^{\alpha+1}\,dx\bigg)^{\frac{r}{r-n\alpha}}\bigg(C_L\int_{\mathbb R^n}m\Big|\frac{w}{m}\Big|^{r}\, dx\bigg)^{-\frac{n\alpha}{r-n\alpha}},
\end{align}
where the equality holds if and only if
\begin{align*}
\mu=\Bigg[\frac{n\alpha\int_{\mathbb R^n}m^{\alpha+1}\, dx}{(\alpha+1)C_L\int_{\mathbb R^n}m\big|\frac{w}{m}\big|^{r}\, dx}\Bigg]^{\frac{1}{r-n\alpha}}.
\end{align*}
 Recall the definition of $e_{0,\alpha,M}:=\inf\limits_{(m,w)\in { {\mathcal A}}_M}\mathcal E_0(m,w)$, then we find from (\ref{sect3-311-mathcalE0}) that
\begin{align*}
-\bigg(\frac{n\alpha}{r}\bigg)^{\frac{r}{r-n\alpha}}\bigg(\frac{r-n\alpha}{n\alpha}\bigg)\bigg(\frac{1}{\alpha+1}\int_{\mathbb R^n}m^{\alpha+1}\,dx\bigg)^{\frac{r}{r-n\alpha}}\bigg(C_L\int_{\mathbb R^n}m\Big|\frac{w}{m}\Big|^{r}\, dx\bigg)^{-\frac{n\alpha}{r-n\alpha}}\geq e_{0,\alpha,M},
\end{align*}
which implies
\begin{align}\label{313-estimate-sect3}
\frac{\bigg(C_L\int_{\mathbb R^n}m\Big|\frac{w}{m}\Big|^{r}\, dx\bigg)^{\frac{n\alpha}{r-n\alpha}}}{\bigg(\frac{1}{\alpha+1}\int_{\mathbb R^n}m^{\alpha+1}\,dx\bigg)^{\frac{r}{r-n\alpha}}}\geq (-e_{0,\alpha,M})^{-1}\Big(\frac{n\alpha}{r}\Big)^{\frac{r}{r-n\alpha}}\bigg(\frac{r-n\alpha}{n\alpha}\bigg).
\end{align}
By using the definition $G_{\alpha}$ given in (\ref{sect3-Galpha-39}), one obtains from (\ref{313-estimate-sect3}) that
\begin{align}\label{combine-sect3-314}
G_{\alpha}(m,w)=\frac{\Big(C_L\int_{\mathbb R^n}m\big|\frac{w}{m}\big|^r\, dx\Big)^{\frac{n\alpha}{r}}\Big(\int_{\mathbb R^n}m\,dx\Big)^{\frac{(\alpha+1)r-n\alpha}{r}}}{\int_{\mathbb R^n}m^{\alpha+1}\, dx}\geq\frac{ H_{\alpha,M}}{\alpha+1}M^{\frac{(\alpha+1)r-n\alpha}{r}},
\end{align}
where
\begin{align}\label{H-alpham-definition}
\int_{\mathbb R^n}m\,dx=M,~~H_{\alpha,M}:=\frac{n\alpha}{r}(-e_{0,\alpha,M})^{\frac{n\alpha-r}{r}}\bigg(\frac{r-n\alpha}{n\alpha}\bigg)^{\frac{r-n\alpha}{r}}.
\end{align}

Recall that $(\bar m_{\alpha,M},\bar w_{\alpha,M})$ is a minimizer of problem  (\ref{energy-epsilon0}), then we apply \eqref{sect3-poho-final} to get
\begin{align}\label{sect3-combine-315}
G_\alpha(\bar m_{\alpha,M},\bar w_{\alpha,M})=\frac{H_{\alpha,M}}{\alpha+1}M^{\frac{(\alpha+1)r-n\alpha}{r}}.
\end{align}
Combining (\ref{combine-sect3-314}) with (\ref{sect3-combine-315}), one can see that (\ref{sect3-Calpha-equivalent}) is attained by $(\bar m_{\alpha,M},\bar w_{\alpha,M}).$  Moreover, noting that $H_{\alpha,M}$ is defined by \eqref{H-alpham-definition}, we have (\ref{sect3-relation-Calpha-emalpha}) holds.
\end{proof}
With the aid of Lemma \ref{sect3-lemma32}, one can use \eqref{sect3-poho-final} to establish the relationship between Lagrange multiplier $\lambda$ and $\Gamma_{\alpha}$ defined by \eqref{energy-epsilon0}.  Indeed, invoking (\ref{sect3-poho-final}) and  (\ref{sect3-relation-Calpha-emalpha}), we   get that 
\begin{align}\label{sect3-lemma31-imply}
e_{0,\alpha,M}=\frac{n\alpha-r}{(\alpha+1)r-n\alpha}\lambda M,
\end{align}
and
\begin{align}\label{sect3-318lambdam}
\lambda M=-S_{\alpha,M}\frac{(\alpha+1)r-n\alpha}{n\alpha}\bigg(\frac{n\alpha}{r}\bigg)^{\frac{r}{r-n\alpha}}, \ \ S_{\alpha,M}:=\Bigg[\frac{M^{\frac{(\alpha+1)r-n\alpha}{r}}}{(1+\alpha)\Gamma_{\alpha}}\Bigg]^{\frac{r}{r-n\alpha}}.
\end{align}
Lemma \ref{sect3-lemma32} demonstrates that for all $M>0$, Gagliardo-Nirenberg type inequalities given by (\ref{optimal-inequality-sub}) can be attained under the mass subcritical exponent case $\alpha\in\big(0,\frac{r}{n}\big)$.  In addition, $\lambda$, $\Gamma_{\alpha}$ and $M$ satisfy (\ref{sect3-318lambdam}).  Next, we shall investigate the mass critical exponent case and prove Theorem \ref{thm11-optimal}.  To begin with, we show that $\Gamma_{\alpha}$ defined in (\ref{optimal-inequality-sub}) is uniformly bounded as  $\alpha\nearrow\frac{r}{n}$.
\begin{lemma}\label{uniformlyboundC1CalphaC2}
There exist positive constants $C_1$ and $C_2$ independent of $\alpha$ such that, for all $\alpha\in(\frac{r}{n}-\epsilon,\frac{r}{n}]$ with $\epsilon>0$ small,
\begin{align}\label{317inlemma33alphaC2}
0<C_1\leq \Gamma_{\alpha}\leq C_2.
\end{align}
\end{lemma}
\begin{proof}
To establish the upper bound uniformly in $\alpha$, we set $\tilde m=e^{-|x|}$ with $\tilde w=\nabla \tilde m$.  Noting that $(\tilde m,\tilde w)\in\mathcal A$ for any $\alpha\in(0,\frac{r}{n}),$ one has
\begin{align}\label{upperboundCalpha}
\Gamma_{\alpha}\leq G_{\alpha}(\tilde m,\tilde w)=\frac{\Big(C_L\int_{\mathbb R^n}\tilde m\Big|\frac{\tilde w}{\tilde m}\Big|^{r}\, dx\Big)^{\frac{n\alpha}{r}}\Big(\int_{\mathbb R^n}\tilde m\,dx\Big)^{\frac{(\alpha+1)r-n\alpha}{r}}}{\int_{\mathbb R^n}\tilde m^{\alpha+1}\, dx}
\leq C_2(C_L,r)<+\infty.
\end{align}

It is left to establish the lower bound satisfied by $\Gamma_{\alpha}$ uniformly in $\alpha$.  To this end, we argue by contradiction and assume
\begin{align}\label{contradiction-assumption}
\liminf_{\alpha\nearrow \frac{r}{n}} \Gamma_{\alpha}=0.
\end{align}
Because of Lemma \ref{sect3-lemma32}, we denote  $(m_{\alpha},w_{\alpha})\in \mathcal A$ as a minimizer of problem (\ref{sect2-equivalence-scaling}).  Since (\ref{sect2-equivalence-scaling}) is invariant under the scaling $s(t^{n}m(tx),t^{n+1}w(tx))$ for any $s>0$ and $t>0,$ we normalize $m_{\alpha}$ to get
\begin{align}\label{normalize-alpha1-alpha}
\int_{\mathbb R^n} m_{\alpha}\, dx=\int_{\mathbb R^n}m^{\alpha+1}_{\alpha}\, dx\equiv 1.
\end{align}
Then it follows from (\ref{contradiction-assumption}) and  (\ref{sect2-equivalence-scaling}) that as $\alpha\nearrow \frac{r}{n},$
\begin{align}\label{eq3.18}
\int_{\mathbb R^n}m_{\alpha}\bigg|\frac{w_{\alpha}}{m_{\alpha}}\bigg|^{r}\,dx\rightarrow 0.
\end{align}
We claim that there exists  $\alpha_0\in\big(0,\frac{r}{n}\big)$ such that
\begin{equation}\label{eq3.19}
  q_{\alpha_0}<1+\frac{r}{n} ~\text{ and }~  q^*_{\alpha_0}>1+\frac{r}{n}.
\end{equation}
where $q_{\alpha}$ and ${q}^*_{\alpha}$ with $\alpha>0$ are defined by
\begin{align}\label{defqstar}
\frac{1}{q_{\alpha}}:=\frac{1}{r}+\frac{1}{(1+\alpha)r'} ~\text{ and }~
{q}^*_{\alpha}:=\left\{\begin{array}{ll}
\frac{nq_{\alpha}}{n-q_{\alpha}},&q_{\alpha}<n,\\
+\infty, &q_{\alpha}\geq n.
\end{array}
\right.
\end{align}
To prove our claim, we choose $\alpha=\alpha^*:=\frac{r}{n}$ and compute
$q_{\alpha^*}=\frac{n+r}{n+1}<1+\frac{r}{n}.$
Moreover, one finds if $q_{\alpha^*}<n,$ then $q^*_{\alpha^*}=\frac{n(n+r)}{n^2-r}>1+\frac{r}{n};$ otherwise if $q_{\alpha^*}\geq n,$ then  $q^*_{\alpha^*}=+\infty>1+\frac{r}{n}.$  Hence, by the continuity of $q_{\alpha}$ and $q^*_{\alpha}$ with respect to $\alpha$, we finish the proof of  claim \eqref{eq3.19}.

With this claim, we invoke Corollary \ref{lemma21-crucial-cor} and choose $\alpha=\alpha_0$ in (\ref{lemma24eq28w1q}) to get
\begin{align}\label{G-Nbeforeinterpolation}
\Vert \nabla m_\alpha\Vert_{L^{q_{\alpha_0}}(\mathbb R^n)}\leq {\bar C}_{\alpha_0}\Big(\int_{\mathbb R^n}m_\alpha\Big|\frac{w_\alpha}{m_\alpha}\Big|^{r}\,dx\Big)^{\frac{1}{r}}\Vert m_\alpha\Vert_{L^{1+\alpha_0}(\mathbb R^n)}^{\frac{1}{r'}}.
\end{align}
 Noting that as $\alpha\nearrow \frac{r}{n},$ $1<1+\alpha_0<1+\alpha$.  By H\"{o}lder's inequality and (\ref{normalize-alpha1-alpha}), we have
\begin{align}\label{globalboundmq0}
\limsup_{\alpha\nearrow \frac{r}{n}}\Vert m_{\alpha}\Vert_{L^{1+\alpha_0}(\mathbb R^n)}\leq C, ~\text{where  $C>0$ does not depend on $\alpha.$}
\end{align}
  Invoking Gagliardo-Nirenberg's inequality, we obtain from \eqref{eq3.18}, (\ref{G-Nbeforeinterpolation})  and (\ref{globalboundmq0}) that
\begin{align*}
\Vert m_{\alpha}\Vert_{L^{1+\frac{r}{n}}(\mathbb R^n)}\leq &C_{\alpha_0} \Vert\nabla m_{\alpha}\Vert_{L^{q_{\alpha_0}}(\mathbb R^n)}^{\theta}\Vert m_{\alpha}\Vert_{L^1(\mathbb R^n)}^{1-\theta}
\leq \tilde C_{\alpha_0}\Big(\int_{\mathbb R^n}m_{\alpha}\Big|\frac{w_{\alpha}}{m_{\alpha}}\Big|^{r}\, dx\Big)^{\frac{\theta}{r}}\rightarrow 0,~~\text{as}~~\alpha\nearrow \frac{r}{n},
\end{align*}
where  $\theta\in(0,1)$ and $\tilde C_{\alpha_0}>0$ are  independent of $\alpha.$  Recall  (\ref{normalize-alpha1-alpha}) and  thanks  to H\"{o}lder's inequality, one has 
\begin{align*}
\Vert m_{\alpha}\Vert_{L^{\alpha+1}(\mathbb R^n)}\leq \Vert m_{\alpha}\Vert_{L^1(\mathbb R^n)}^{1-\theta_{\alpha}}\Vert m_{\alpha}\Vert_{L^{1+\frac{r}{n}}(\mathbb R^n)}^{\theta_{\alpha}}=\Vert m_{\alpha}\Vert_{L^{1+\frac{r}{n}}(\mathbb R^n)}^{\theta_{\alpha}}\rightarrow 0, \ \ \text{as~~}{\alpha}\nearrow \frac{r}{n},
\end{align*}
where constants $\theta_{\alpha}\rightarrow 1$, which reaches a contradiction to (\ref{normalize-alpha1-alpha}).  Thus, there exists  $C_1>0$ independent of $\alpha$ such that
\begin{align}\label{lemma33lowebounddesired20231027}
0<C_1\leq \Gamma_{\alpha}.
\end{align}
Combining (\ref{lemma33lowebounddesired20231027}) with (\ref{upperboundCalpha}), we complete the proof of (\ref{317inlemma33alphaC2}).
\end{proof}
With the uniform boundedness of $\Gamma_{\alpha}$, we next establish the uniform $L^\infty$ bound of $m_{\alpha}$ as $\alpha\nearrow\frac{r}{n}$, which is
\begin{lemma}\label{lemma34uniformlinfboundmalpha}
Let $(u_{\alpha},\lambda_{\alpha},m_{\alpha})\in C^2(\mathbb R^n)\times \mathbb R\times W^{1,p}(\mathbb R^n)$, $\forall p>1$ be the solution of
\begin{align}\label{lemma34mualphaeq}
\left\{\begin{array}{ll}
-\Delta u+C_H|\nabla u|^{r'}+\lambda=-m^{\alpha},&x\in\mathbb R^n,\\
-\Delta m-C_H r'\nabla\cdot(m|\nabla u|^{r'-2}\nabla u)=0,&x\in\mathbb R^n,\\
\int_{\mathbb R^n}m\, dx=M_{\alpha}.
\end{array}
\right.
\end{align}
Define $w_{\alpha}=-C_Hr'm_{\alpha}|\nabla u_{\alpha}|^{r'-2}\nabla u_{\alpha}$.
 Assume that each $u_{\alpha}$ is bounded from below  and there exists a constant $C>0$ independent of $\alpha$, such that
\begin{align}\label{energypohoboundedness}
\limsup_{\alpha\nearrow\frac{r}{n}}\int_{\mathbb R^n}m_{\alpha}|\nabla u_{\alpha}|^{r'}\,dx\leq C,\ \ \lim_{\alpha\nearrow \frac{r}{n}}\int_{\mathbb R^n}m_{\alpha}\, dx=\lim_{\alpha\nearrow \frac{r}{n}} M_{\alpha}\leq C, \ \ \limsup_{\alpha\nearrow \frac{r}{n}}|\lambda_{\alpha}|\leq C,
\end{align}
then there exists $C_1>0$  independent of $\alpha$ such that
\begin{align}\label{linfinityboundofmuniform}
\limsup_{\alpha\nearrow \frac{r}{n}}\Vert m_{\alpha}\Vert_{L^\infty(\mathbb R^n)}\leq C_1.
\end{align}

\end{lemma}
\begin{proof}
Motivated by the argument of \cite[Theorem 4.1]{cesaroni2018concentration},
to prove (\ref{linfinityboundofmuniform}), we argue by contradiction and suppose that up to a subsequence,
\begin{align}\label{blowupassumption}
\mu_{\alpha}:=\Vert m_{\alpha}\Vert_{L^\infty(\mathbb R^n)}^{-\frac{1}{n}}\rightarrow 0 \ \ \text{~as~}\alpha\nearrow\frac{r}{n}.
\end{align}
Since $u_{\alpha}$ is bounded from below, we fix $0=u_{\alpha}(0)=\inf\limits_{x\in \mathbb R^n} u_{\alpha}(x)$ without loss of generality.  Define
\begin{align}\label{blowupanalysisrescaling}
\bar u_{\alpha}:=\mu_{\alpha}^{\frac{2-r'}{r'-1}}u_{\alpha}(\mu_{\alpha}x)+1,\ 
\bar m_{\alpha}:=\mu_{\alpha}^n m_{\alpha}(\mu_{\alpha} x) ~\text{ and }~
\bar w_{\alpha}:=\mu_{\alpha}^{n+1}w_{\alpha}(\mu_{\alpha}x),
\end{align}
then we have from (\ref{energypohoboundedness}) and (\ref{blowupassumption}) that up to a subsequence,
\begin{align}\label{eq3.36}
\int_{\mathbb R^n}\bar m_{\alpha}\,dx=\int_{\mathbb R^n} m_{\alpha}\, dx=M_{\alpha},\ 
\int_{\mathbb R^n}\bar m_{\alpha}^{\alpha+1}\, dx=\mu_{\alpha}^{\alpha n}\int_{\mathbb R^n}m_{\alpha}^{\alpha+1}\, dx\rightarrow 0 \ \ \text{as~}\alpha\nearrow \frac{r}{n},
\end{align}
and
\begin{align}\label{338thevanishofmgradu}
\int_{\mathbb R^n}\bar m_{\alpha}|\nabla \bar u_{\alpha}|^{r'}\,dx=\mu_{\alpha}^{r}\int_{\mathbb R^n}m_{\alpha}|\nabla u_{\alpha}|^{r'}\, dx\rightarrow 0 \ \ \text{as~}\alpha\nearrow \frac{r}{n}.
\end{align}
Noting the definition of $w_{\alpha}$, \eqref{MFG-L} and (\ref{blowupanalysisrescaling}), we have
\begin{align}\label{operatorbodyconverge0}
C_L\int_{\mathbb R^n}\bigg|\frac{\bar w_{\alpha}}{\bar m_{\alpha}}\bigg|^{r} \bar m_{\alpha}\,dx=(r'-1)C_H\int_{\mathbb R^n}\bar m_{\alpha}|\nabla\bar u_{\alpha}|^{r'}\,dx\rightarrow 0, \ \ \text{as~}\alpha\nearrow \frac{r}{n}.
\end{align}
In light of (\ref{blowupassumption}), we obtain from (\ref{blowupanalysisrescaling}) that
\begin{align}\label{342malphainfequiv1always}
\Vert \bar m_{\alpha}\Vert_{L^\infty}\equiv 1.
\end{align}
This together with \eqref{eq3.36} indicates that for any $q>1+\alpha$, 
\begin{align}\label{343mL1plusgammaovernnorm}
\int_{\mathbb R^n}\bar m_{\alpha}^{q}\, dx\leq \bigg(\int_{\mathbb R^n}\bar m_{\alpha}^{1+\alpha}\, dx\bigg)\|\bar m_\alpha\|_{L^\infty(\mathbb \R^n)}^{q-\alpha-1}\rightarrow 0 \text{ as $\alpha\nearrow \frac{r}{n}$}.
\end{align}
On the other hand, by using (\ref{blowupanalysisrescaling}), we have from (\ref{lemma34mualphaeq}) that
\begin{align}\label{limitingproblembarubarm}
\left\{\begin{array}{ll}
-\Delta_x\bar u_{\alpha}+C_H|\nabla_x \bar u_{\alpha}|^{r'}+\lambda_{\alpha}\mu_{\alpha}^{r}=-\mu_{\alpha}^{r-n\alpha}\bar m_{\alpha}^{\alpha+1},&x\in\mathbb R^n,\\
-\Delta_x \bar m_{\alpha}-C_Hr'\nabla_x\cdot(\bar m_{\alpha}|\nabla_x \bar u_{\alpha}|^{r'-2}\nabla_x \bar u_{\alpha})=0,&x\in\mathbb R^n,\\
\int_{\mathbb R^n}\bar m_{\alpha}\, dx=M_{\alpha}.
\end{array}
\right.
\end{align}
In light of (\ref{energypohoboundedness}) and (\ref{blowupassumption}), one finds
$\lambda_{\alpha}\mu_{\alpha}^{r}\rightarrow 0\ \ \text{as~} \alpha\nearrow \frac{r}{n}.$
In addition, thanks to the fact $r\geq n\alpha$, we obtain
$$
0\leq \mu_{\alpha}^{r-n\alpha}\Vert\bar m_{\alpha}^{1+\alpha}\Vert_{L^\infty(\mathbb \R^n)}\leq  \Vert\bar m_{\alpha}\Vert_{L^\infty(\mathbb \R^n)}^{\alpha+1}\leq 1.
$$
Thus, one applies Lemma \ref{sect2-lemma21-gradientu} on the $\bar u_\alpha$-equation in (\ref{limitingproblembarubarm}) to arrive at
\begin{align}\label{346gradualphaestimate}
\limsup_{\alpha\nearrow \frac{r}{n}}\Vert\nabla \bar u_{\alpha}\Vert_{L^\infty(\mathbb \R^n)}\leq C<\infty.
\end{align}
  Since $\bar w_{\alpha}=-C_Hr'\bar m_{\alpha}|\nabla \bar u_{\alpha}|^{r'-2}\nabla \bar u_{\alpha}$,
we deduce  from (\ref{346gradualphaestimate}) that
\begin{align}\label{350barwalphaLinflessC}
\limsup_{\alpha\nearrow \frac{r}{n}}\Vert \bar w_{\alpha}\Vert_{L^\infty(\mathbb \R^n)}\leq C<\infty.
\end{align}

Now, we focus on H\"{o}lder estimates of $\bar m_{\alpha}$.  By using (\ref{operatorbodyconverge0}) and H\"{o}lder's inequality , we get as $\alpha\nearrow \frac{r}{n},$
\begin{align}\label{351barwalphaL1lessC}
\int_{\mathbb R^n}|\bar w_{\alpha}|\, dx=\int_{\mathbb R^n}|\bar w_{\alpha}|\bar m_{\alpha}^{\frac{1-r}{r}}\bar m_{\alpha}^{\frac{r-1}{r}}\, dx\leq &\bigg(\int_{\mathbb R^n}|\bar w_{\alpha}|^{r}|\bar m_{\alpha}|^{1-r}\, dx\bigg)^{\frac{1}{r}}\bigg(\int_{\mathbb R^n}\bar m_{\alpha}\, dx\bigg)^{\frac{1}{r'}}\rightarrow 0.
\end{align}
Combining (\ref{350barwalphaLinflessC}) with (\ref{351barwalphaL1lessC}), one has
\begin{align}\label{350inproveoptimal}
\int_{\mathbb R^n}|\bar w_{\alpha}|^p\,dx\rightarrow 0\ \ \text{as~}\alpha\nearrow \frac{r}{n}, \ \ \forall 1<p<+\infty.
\end{align}
For any $q>n$, by using the $\bar m_\alpha$-equation in (\ref{limitingproblembarubarm}), we obtain as $\alpha\nearrow \frac{r}{n},$
\begin{align}\label{350inproveoptimal1}
\Big|-\int_{\mathbb R^n}\bar m_{\alpha}\Delta \varphi \, dx\Big|=\Big|\int_{\mathbb R^n}\bar w_{\alpha}\cdot\nabla \varphi \, dx\Big|\leq \Big(\int_{\mathbb R^n}|\bar w_{\alpha}|^q\,dx\Big)^{\frac{1}{q}}\Vert \nabla \varphi\Vert_{L^{q'}}\rightarrow 0, \ \ \forall \varphi\in C_{c}^\infty(\mathbb R^n).
\end{align}
With (\ref{350inproveoptimal}) and (\ref{350inproveoptimal1}), one applies Lemma \ref{lemma21-crucial} to get
$
\Vert \nabla \bar m_{\alpha}\Vert_{L^q}\rightarrow 0 \ \ \text{as~}\alpha\nearrow \frac{r}{n}. 
$
Combine this with \eqref{343mL1plusgammaovernnorm} we obtain that, 
$\Vert \bar m_{\alpha}\Vert_{W^{1,q}(\mathbb R^n)}\rightarrow 0\ \ \text{as~}\alpha\nearrow \frac{r}{n}.$
Thanks to Sobolev embedding theorem, one further has for some $\theta'\in(0,1),$
\begin{align}\label{malphaholdercontinuousuniformlyinalpha}
\Vert \bar m_{\alpha}\Vert_{C^{0,\theta'}(\mathbb R^n)}\rightarrow 0 \ \ \text{as~}\alpha\nearrow \frac{r}{n}.
\end{align}

Let $x_{\alpha}$ be a maximum point of $\bar m_{\alpha}$, i.e.,  $\bar m_{\alpha}(x_{\alpha})=\Vert \bar m_{\alpha}\Vert_{L^\infty(\mathbb R^n)}=1.$  Then we obtain from (\ref{malphaholdercontinuousuniformlyinalpha}) that there exists $R_1>0$ independent of $\alpha$ such that
$|\bar m_{\alpha}(x)|\geq \frac{1}{2},\forall x\in B_{R_1}(x_{\alpha})$.
It follows that
\begin{align*}
\bigg(\frac{1}{2}\bigg)^{1+\frac{r}{n}}|B_{R_1}|\leq \int_{B_{R_1}(x_{\alpha})}\bar m_{\alpha}^{1+\frac{r}{n}}\, dx\leq \int_{\mathbb R^n}\bar m_{\alpha}^{1+\frac{r}{n}}\, dx,
\end{align*}
which is contradicted to  \eqref{343mL1plusgammaovernnorm}.  This completes the proof of (\ref{linfinityboundofmuniform}).
\end{proof}

Now, we are ready to show Theorem \ref{thm11-optimal} by using the approximation argument, which is stated as follows:

\medskip

\textbf{Proof of Theorem \ref{thm11-optimal}:}
\begin{proof}
Recall that $(\bar m_{\alpha,M},\bar{w}_{\alpha,M}, \lambda_{\alpha,M})$ denotes the minimizer of problem (\ref{sect3-Calpha-equivalent}) for any $M>0$, which satisfies system (\ref{eq-attained-sub}) and estimates \eqref{eq3.2} and \eqref{eq3.3}.

Choosing $$M=M_{\alpha}:=e^{\frac{r-n\alpha}{(\alpha+1)r-n\alpha}}\big[(\alpha+1)\Gamma_{\alpha}\big]^{\frac{r}{(\alpha+1)r-n\alpha}}$$ in \eqref{eq-attained-sub}, one can get from (\ref{sect3-318lambdam}) that
$$S_{\alpha,M_\alpha}:=\Bigg[\frac{M_\alpha^{\frac{(\alpha+1)r-n\alpha}{r}}}{(1+\alpha)\Gamma_{\alpha}}\Bigg]^{\frac{r}{r-n\alpha}}\equiv e.$$
Then as $\alpha\nearrow \frac{r}{n}$, we find up to a subsequence,
\begin{align}\label{356Malphagammaalpharelation}
\frac{M_{\alpha}^{\frac{(\alpha+1)r-n\alpha}{r}}}{(\alpha+1)\Gamma_{\alpha}}\rightarrow 1, \ \ \frac{(\alpha+1)r-n\alpha}{r}\rightarrow \frac{r}{n}.
\end{align}
Without confusing the readers, rewrite $(\bar m_{\alpha,M},\bar{w}_{\alpha,M}, \lambda_{\alpha,M})$ as $(\bar m_{\alpha,M_{\alpha}},\bar{w}_{\alpha,M_{\alpha}}, \lambda_{\alpha,M_{\alpha}})$ since $M=M_\alpha$ depends on $\alpha$. 
 We also recall from \eqref{eq-attained-sub} that $(\bar m_{\alpha,M_{\alpha}},\bar u_{\alpha,M_\alpha}, \bar{w}_{\alpha,M_{\alpha}}, \lambda_{\alpha,M_{\alpha}})$ satisfies
  \begin{align}\label{eq3.50}
\left\{\begin{array}{ll}
-\Delta u+C_H|\nabla u|^{r'}+\lambda_{\alpha, M_\alpha}=-m^{\alpha},&x\in\mathbb R^n,\\
\Delta m+C_Hr'\nabla\cdot(m|\nabla u|^{r'-2}\nabla u)=0,&x\in\mathbb R^n,\\
 w=-C_Hr' m|\nabla u|^{r'-2}\nabla u,\ \int_{\mathbb R^n}m\, dx=M_\alpha.
\end{array}
\right.
\end{align}
 By using Lemma \ref{uniformlyboundC1CalphaC2}, one can see that as $\alpha\nearrow \frac{r}{n}$, up to a subsequence, $\Gamma_{\alpha}\rightarrow \bar \Gamma_{\alpha^*}:=\liminf\limits_{\alpha\nearrow \frac{r}{n}} \Gamma_{\alpha}>0$.  Moreover, (\ref{356Malphagammaalpharelation}) implies $M_{\alpha}\rightarrow M_{\alpha^*}:=M^*,$ where
\begin{align}\label{mstarlimitfinal}
M^*=\Bigg[\Bigg(1+\frac{r}{n}\Bigg)\bar \Gamma_{\alpha^*}\Bigg]^{\frac{n}{r}}, \ \ \alpha^*=\frac{r}{n}.
\end{align}
In addition, by using (\ref{sect3-318lambdam}), we obtain as $\alpha\nearrow \frac{r}{n}$, up to a subsequence,
\begin{align}\label{lambdalimitbehavior}
\lambda_{\alpha, M_{\alpha}}\rightarrow \lambda_{\alpha^*}:=-\frac{r}{nM^*},
\end{align}
and it follows from (\ref{sect3-poho-final}) that
\begin{align}\label{eq3.58}
\int_{\mathbb R^n}{\bar m}_{\alpha,M_{\alpha}}\,dx=M_{\alpha}\rightarrow M^*>0,\ \ \int_{\mathbb R^n}{\bar m}_{\alpha,M_{\alpha}}^{\alpha+1}\, dx\rightarrow  1+\frac{r}{n}, \ \ C_L\int_{\mathbb R^n}\bar m_{\alpha,M_{\alpha}}\bigg|\frac{\bar w_{\alpha,M_{\alpha}}}{\bar m_{\alpha,M_{\alpha}}}\bigg|^{r}\,dx\rightarrow 1.
\end{align}
According to Lemma \ref{lemma34uniformlinfboundmalpha}, we obtain from  \eqref{lambdalimitbehavior} and \eqref{eq3.58} that 
\begin{equation}\label{eq3.54}
\limsup_{\alpha\nearrow \frac{r}{n}}\Vert \bar m_{\alpha, M_\alpha}\Vert_{L^\infty(\mathbb R^n)}<\infty.
\end{equation}
We then deduce from \eqref {usolutiongradientestimatepre} with $b=0$ that
\begin{align}\label{364supnablauupper}
\limsup_{\alpha\nearrow \frac{r}{n}}\Vert \nabla \bar u_{\alpha,M_{\alpha}}\Vert_{L^\infty}<\infty.
\end{align}
Noting from  the definition of $\bar w_{\alpha,M_{\alpha}}$, one further has
\begin{align}\label{361sandwichbefore}
\limsup_{\alpha\nearrow \frac{r}{n}}\Vert \bar w_{\alpha,M_{\alpha}}\Vert_{L^\infty}<\infty.
\end{align}
Similar to the argument of \eqref{malphaholdercontinuousuniformlyinalpha}, we can  use \eqref{eq3.58}-\eqref{361sandwichbefore} to  get  that 
\begin{align}\label{eq3.60}
\limsup_{\alpha\nearrow \frac{r}{n}}\Vert \bar m_{\alpha,M_{\alpha}}\Vert_{W^{1,q}(\mathbb R^n)}<+\infty \ \forall \ q>n, ~\text{ and }~\limsup_{\alpha\nearrow \frac{r}{n}}\Vert\bar m_{\alpha,M_{\alpha}}\Vert_{C^{0,\tilde\theta}(\mathbb R^n)}<\infty ~\text{ for some}~\tilde\theta\in(0,1).
\end{align}
since  each $\bar u_{\alpha,M_\alpha}\in C^2(\mathbb R^n)$ is bounded from below, we may assume that $\bar u_{\alpha,M_\alpha}(0)=0=\inf_{x\in \mathbb R^n}\bar u_{\alpha,M_\alpha}(x)$. We then deduce from $u$-equation of \eqref{eq3.50}  that $ \bar m_{\alpha,M_{\alpha}}^\alpha(0)\geq -\lambda_{\alpha, M_{\alpha}}$, which together with \eqref{eq3.58}  and \eqref{lambdalimitbehavior} implies that  there exist  $\delta_1, R>0$ independent of $\alpha$ such that,
\begin{equation}\label{eq3.62}
\bar m_{\alpha,M_{\alpha}}(x)>\frac{\delta_1}{2}>0,\text{ for }|x|<R.
\end{equation}

 We rewrite the $u$-equation of \eqref{eq3.50} as
     \begin{equation}\label{eq3.56}
     -\Delta \bar u_{\alpha,M_\alpha}=-C_H|\nabla \bar u_{\alpha,M_\alpha}|^{r'}+h_\alpha(x) ~\text{with }~h_\alpha(x):=-\lambda_{\alpha, M_\alpha}-\bar m_{\alpha,M_\alpha}^{\alpha},x\in\mathbb R^n\end{equation}
In light of   (\ref{364supnablauupper}), we deduce that $|\bar u_{\alpha,M_\alpha}(x)|\leq C(|x|+1)$ with $C>0$ independent of $\alpha$.
We then derive from the  classical $W^{2,p}$ estimate to derive  from \eqref{eq3.54} and \eqref{364supnablauupper} that, for any $\bar R>0$ and $p>n$, 
 \begin{equation*}
 \|\bar u_{\alpha,M_\alpha}\|_{W^{2,p}(B_{\bar R+1})}\leq C\left(\|\bar u_{\alpha,M_\alpha}\|_{L^{p}(B_{2\bar R}(0))}+\|h_\alpha\|_{L^{p}(B_{2\bar R}(0))}+\||\nabla\bar u_{\alpha,M_\alpha}|^{r'}\|_{L^{p}(B_{2\bar R}(0))}\right)
 \leq C_{p,\bar R}<\infty,
 \end{equation*}
  where the constant  $C_{p, \bar R}>0$ is independent of $\alpha$. It then follows from the Sobolev embedding theorem that 
$\|\bar u_{\alpha,M_\alpha}\|_{C^{1,\theta_1}(B_{\bar R+1}(0))}\leq C_{\theta_1,\bar R}<\infty \text{ for some $\theta_1\in(0,1)$.}$
This combines  with (\ref{eq3.60}) gives 
$$\||\bar u_{\alpha,M_\alpha}|^{r'}\|_{C^{0,\theta_2}(B_{\bar R+1}(0))} +\|h_\alpha\|_{C^{0,\theta_2}(B_{\bar R+1}(0))}\leq C_{\theta_2,\bar R}<\infty\text{ for some $\theta_2\in(0,1)$.}.$$
 Then by using Schauder's estimates, we have from \eqref{eq3.56} that
 \begin{equation}\label{eq4.54}
     \|\bar u_{\alpha,M_\alpha}\|_{C^{2,\theta_3}(B_{R}(0))} \leq C_{\theta_3,\bar R}<\infty, \text{ for some $\theta_3\in(0,1)$.}
 \end{equation}


 Now, letting $\bar R\to\infty$ and proceeding the standard diagonalization procedure, we can apply Arzel\`{a}-Ascoli theorem to get from \eqref{eq3.60} and\eqref{eq4.54} that there exists  $(m_{\alpha^*},u_{\alpha^*})\in W^{1,p}(\mathbb R^n)\times C^2(\mathbb R^n)$  such that
\begin{align}\label{269locallyuniform}
 \bar m_{\alpha,M_{\alpha}}\rightharpoonup m_{\alpha^*} \text{ in }W^{1,p}(\mathbb R^n), \text{ and }\bar  u_{\alpha,M_{\alpha}}\rightarrow  u_{\alpha^*} \ \ \text{in  } C^2_{\rm loc}(\mathbb R^n),  \text{ as }\alpha\nearrow \frac{r}{n}.
\end{align}
This together with \eqref{eq3.50} and \eqref{lambdalimitbehavior} implies that
 $(m_{\alpha^*},u_{\alpha^*})\in W^{1,p}(\mathbb R^n)\times C^2(\mathbb R^n)$ satisfies 
\begin{align}\label{limitingproblemminimizercritical}
\left\{\begin{array}{ll}
-\Delta u+C_H|\nabla u|^{r'}-\frac{r}{nM^*}=-m^{\frac{r}{n}},&x\in\mathbb R^n,\\
-\Delta m-r'C_H\nabla\cdot(m|\nabla u|^{r'-2}\nabla u)=0,&x\in\mathbb R^n,\\
w=-C_Hr'm|\nabla u|^{r'-2}\nabla u.
\end{array}
\right.
\end{align}
Thanks to \eqref{eq3.62} and Fatou's lemma, one finds that 
\begin{equation}\label{eq3.67}
\int_{\mathbb R^n}m_{\alpha^*}\,dx=a\in(0,M^*].
\end{equation} 
  Moreover, we deduce from Lemma \ref{mdecaylemma}  that  there exists some $\kappa,C>0$ such that $m_{\alpha^*}(x)<Ce^{-\kappa|x|}.$  In addition, from \eqref{364supnablauupper} we have $\Vert \nabla u_{\alpha^*}\Vert_{L^\infty}<\infty$.   Then, by applying Lemma \ref{poholemma}, Pohozaev identities, one has 
\begin{align}\label{370criticalpohoidentity}
C_L\int_{\mathbb R^n}\Bigg|\frac{w_{\alpha^*}}{m_{\alpha^*}}\Bigg|^{r}m_{\alpha^*}\,dx=\frac{n}{n+r}\int_{\mathbb R^n}m_{\alpha^*}^{1+\frac{r}{n}}\, dx.
\end{align}

Next, we discuss the relationship between $\bar\Gamma_{\alpha^*}:=\liminf\limits_{\alpha\nearrow\frac{r}{n}}\Gamma_{\alpha}$ and $\Gamma_{\alpha^*}$ with $\alpha^*=\frac{r}{n}.$  We claim that
\begin{align}\label{ourclaimtwolimit}
\bar \Gamma_{\alpha^*}=\Gamma_{\frac{r}{n}}.
\end{align}
 To show this, we first note from Lemma \ref{sect3-lemma32} that  
\begin{align}\label{337beforecomparecbarc}
\Gamma_{\alpha}=G_{\alpha}(\bar m_{\alpha,M_{\alpha}},\bar w_{\alpha,M_{\alpha}})=&G_{\frac{r}{n}}(\bar m_{\alpha,M_{\alpha}},\bar w_{\alpha,M_{\alpha}})\frac{\Big(C_L\int_{\mathbb R^n}\bar m_{\alpha,M_{\alpha}}\Big|\frac{\bar w_{\alpha,M_{\alpha}}}{\bar m_{\alpha,M_{\alpha}}}\Big|^r\, dx\Big)^{\frac{n\alpha}{r}}\Big(\int_{\mathbb R^n}\bar m_{\alpha,M_{\alpha}}\,dx\Big)^{\frac{(\alpha+1)r-n\alpha}{r}}}{\Big(C_L\int_{\mathbb R^n}\bar m_{\alpha,M_{\alpha}}\Big|\frac{\bar w_{\alpha,M_{\alpha}}}{\bar m_{\alpha,M_{\alpha}}}\Big|^r\, dx\Big)\Big(\int_{\mathbb R^n}\bar m_{\alpha,M_{\alpha}}\,dx\Big)^{\frac{r}{n}}}\cdot\frac{\int_{\mathbb R^n}\bar m_{\alpha,M_{\alpha}}^{\frac{r}{n}+1}\,dx}{\int_{\mathbb R^n}\bar m^{\alpha+1}_{\alpha,M_{\alpha}}\,dx}\nonumber\\
\geq& \Gamma_{\frac{r}{n}}\frac{\Big(C_L\int_{\mathbb R^n}\bar m_{\alpha,M_{\alpha}}\Big|\frac{\bar w_{\alpha,M_{\alpha}}}{\bar m_{\alpha,M_{\alpha}}}\Big|^r\, dx\Big)^{\frac{n\alpha}{r}}\Big(\int_{\mathbb R^n}\bar m_{\alpha,M_{\alpha}}\,dx\Big)^{\frac{(\alpha+1)r-n\alpha}{r}}}{\Big(C_L\int_{\mathbb R^n}\bar m_{\alpha,M_{\alpha}}\Big|\frac{\bar w_{\alpha,M_{\alpha}}}{\bar m_{\alpha,M_{\alpha}}}\Big|^r\, dx\Big)\Big(\int_{\mathbb R^n}\bar m_{\alpha,M_{\alpha}}\,dx\Big)^{\frac{r}{n}}}\cdot\frac{\int_{\mathbb R^n}\bar m_{\alpha,M_{\alpha}}^{\frac{r}{n}+1}\,dx}{\int_{\mathbb R^n}\bar m_{\alpha,M_{\alpha}}^{\alpha+1}\,dx}.
\end{align}
Invoking \eqref{eq3.58}, we deduce that  as $\alpha\nearrow \frac{r}{n},$ 
\begin{align*}
\frac{\Big(C_L\int_{\mathbb R^n}\bar m_{\alpha,M_{\alpha}}\Big|\frac{\bar w_{\alpha,M_{\alpha}}}{\bar m_{\alpha,M_{\alpha}}}\Big|^r\, dx\Big)^{\frac{n\alpha}{r}}\Big(\int_{\mathbb R^n}\bar m_{\alpha,M_{\alpha}}\,dx\Big)^{\frac{(\alpha+1)r-n\alpha}{r}}}{\Big(C_L\int_{\mathbb R^n}\bar m_{\alpha,M_{\alpha}}\Big|\frac{\bar w_{\alpha,M_{\alpha}}}{\bar m_{\alpha,M_{\alpha}}}\Big|^r\, dx\Big)\Big(\int_{\mathbb R^n}\bar m_{\alpha,M_{\alpha}}\,dx\Big)^{\frac{r}{n}}}\cdot\frac{\int_{\mathbb R^n}\bar m_{\alpha,M_{\alpha}}^{\frac{r}{n}+1}\,dx}{\int_{\mathbb R^n}\bar m_{\alpha,M_{\alpha}}^{\alpha+1}\,dx}\rightarrow 1.
\end{align*}
Then one takes the limit in \eqref{337beforecomparecbarc} to get
\begin{align}\label{gammaalphastar}
\bar \Gamma_{\alpha^*}:=\liminf_{\alpha\nearrow\frac{r}{n}}\Gamma_{\alpha}\geq \Gamma_{\frac{r}{n}}.
\end{align}
To finish the proof of our claim, it is left to show that the ``=" holds in \eqref{gammaalphastar}.  On the contrary, if  $\Gamma_{\frac{r}{n}}<\bar \Gamma_{\alpha^*}$,  then by the definition of $\Gamma_{\frac{r}{n}},$ one finds there exists $(\hat m,\hat w)\in \mathcal A$ given in (\ref{sect2-equivalence-scaling}) such that
\begin{align}\label{340limit1}
G_{\frac{r}{n}}(\hat m,\hat w)\leq \Gamma_{\frac{r}{n}}+\delta<\Gamma_{\frac{r}{n}}+2\delta<\bar \Gamma_{\alpha^*},
\end{align}
where $\delta>0$ is sufficiently small.  
On the other hand, from the definition of $\Gamma_{\alpha},$ we find
\begin{align}\label{342limit2}
G_{\frac{r}{n}}(\hat m,\hat w)=&G_{\alpha}(\hat m,\hat w)\frac{{\Big(C_L\int_{\mathbb R^n}\hat  m\Big|\frac{\hat w}{\hat m}\Big|^r\, dx\Big)^{\frac{n\alpha}{r}}\Big(\int_{\mathbb R^n}\hat m\,dx\Big)^{\frac{(\alpha+1)r-n\alpha}{r}}}}{\Big(C_L\int_{\mathbb R^n}\hat m\Big|\frac{\hat w}{\hat m}\Big|^r\, dx\Big)\Big(\int_{\mathbb R^n}\hat m\,dx\Big)^{\frac{r}{n}}}\cdot\frac{\int_{\mathbb R^n}{\hat m}^{\frac{r}{n}+1}\,dx}{\int_{\mathbb R^n}\hat m^{\alpha+1}\,dx}\nonumber\\
\geq& \Gamma_{\alpha}\frac{{\Big(C_L\int_{\mathbb R^n}\hat  m\Big|\frac{\hat w}{\hat m}\Big|^r\, dx\Big)^{\frac{n\alpha}{r}}\Big(\int_{\mathbb R^n}\hat m\,dx\Big)^{\frac{(\alpha+1)r-n\alpha}{r}}}}{\Big(C_L\int_{\mathbb R^n}\hat m\Big|\frac{\hat w}{\hat m}\Big|^r\, dx\Big)\Big(\int_{\mathbb R^n}\hat m\,dx\Big)^{\frac{r}{n}}}\cdot\frac{\int_{\mathbb R^n}{\hat m}^{\frac{r}{n}+1}\,dx}{\int_{\mathbb R^n}\hat m^{\alpha+1}\,dx}.
\end{align}
Since
\begin{align*}
\frac{{\Big(C_L\int_{\mathbb R^n}\hat  m\Big|\frac{\hat w}{\hat m}\Big|^r\, dx\Big)^{\frac{n\alpha}{r}}\Big(\int_{\mathbb R^n}\hat m\,dx\Big)^{\frac{(\alpha+1)r-n\alpha}{r}}}}{\Big(C_L\int_{\mathbb R^n}\hat m\Big|\frac{\hat w}{\hat m}\Big|^r\, dx\Big)\Big(\int_{\mathbb R^n}\hat m\,dx\Big)^{\frac{r}{n}}}\cdot\frac{\int_{\mathbb R^n}{\hat m}^{\frac{r}{n}+1}\,dx}{\int_{\mathbb R^n}\hat m^{\alpha+1}\,dx}\rightarrow 1 \ \ \text{~as~}\alpha\nearrow \frac{r}{n},
\end{align*}
we take the limit in (\ref{340limit1}) and (\ref{342limit2}) to obtain
\begin{align*}
\bar\Gamma_{\alpha^*}=\liminf_{\alpha\nearrow \frac{r}{n}}\Gamma_{\alpha}\leq \Gamma_{\frac{r}{n}}+\delta<\Gamma_{\frac{r}{n}}+2\delta\leq \liminf_{\alpha\nearrow \frac{r}{n}}\Gamma_{\alpha},
\end{align*}
which reaches a contradiction.  Now, one has \eqref{ourclaimtwolimit}, i.e., $\bar \Gamma_{\alpha^*}=\Gamma_{\frac{r}{n}}$ holds.

Next, we prove $(m_{\alpha^*},w_{\alpha^*})\in \mathcal A.$
 Since $(m_{\alpha^*},w_{\alpha^*})$ solves (\ref{limitingproblemminimizercritical}) and $m_{\alpha^*}\in C^{0,\theta}(\mathbb R^n)$ with $\theta\in(0,1)$, we conclude from (\ref{269locallyuniform}) and Lemma \ref{sect2-lemma21-gradientu} that $u_{\alpha^*}\in C^1(\mathbb R^n).$  Then by standard elliptic estimates, the boundedness of $\Vert\nabla u_{\alpha^*}\Vert_{L^\infty}$ and the exponentially decaying property of $m_{\alpha^*}$, one can prove that $(m_{\alpha^*},w_{\alpha^*}) \in \mathcal A$.

 On the other hand, 
we have from (\ref{eq3.58}) and \eqref{gammaalphastar} that
\begin{align}\label{combining1strongconverge}
\liminf_{\alpha\nearrow \frac{r}{n}}\Gamma_{\alpha}=\frac{n}{n+r}\big(M^*\big)^{\frac{r}{n}}=\Gamma_{\frac{r}{n}},
\end{align}
where $M^*$ is given in \eqref{mstarlimitfinal}.  Then with $(m_{\alpha^*},w_{\alpha^*}) \in \mathcal A$, it follows from \eqref{eq3.67}, (\ref{370criticalpohoidentity}) and (\ref{combining1strongconverge}) that  
\begin{align}\label{combining1strongconverge2}
\Gamma_{\frac{r}{n}}=\frac{n}{n+r}\big(M^*\big)^{\frac{r}{n}}\leq \frac{C_L\int_{\mathbb R^n}\Big|\frac{w_{\alpha^*}}{m_{\alpha^*}}\Big|^{r}m_{\alpha^*}\,dx\Big(\int_{\mathbb R^n}m_{\alpha^*}\, dx\Big)^{\frac{r}{n}}}{\int_{\mathbb R^n}m_{\alpha^*}^{1+\frac{r}{n}}\, dx}=\frac{n}{n+r}a^{\frac{r}{n}}\leq \frac{n}{n+r}\big(M^*\big)^{\frac{r}{n}} ,
\end{align}
which indicates $(m_{\alpha^*},w_{\alpha^*}) \in \mathcal A$ is an minimizer of $\Gamma_{\frac{r}{n}}$  as well as 
\begin{align*}
\int_{\mathbb R^n}m_{\alpha^*}\,dx=M^* ~\text{ and }~\bar  m_{\alpha,M_{\alpha}}\rightarrow m_{\alpha^*} \ \ \text{in~} L^1(\mathbb R^n) \text{ as }\alpha\nearrow \frac{r}{n}.
\end{align*}
This together with \eqref{limitingproblemminimizercritical} indicates \eqref{limitingproblemminimizercritical0}.  The proof  of Theorem \ref{thm11-optimal} is finished.

\end{proof}
With the proof of Theorem \ref{thm11-optimal}, we have studied the existence of ground states for potential-free MFG systems and established Gagliardo-Nirenberg type inequality under the mass critical exponent case.  In Section \ref{sect4-criticalmass}, \ref{sect5preciseblowup} and \ref{refinedblow}, we shall apply the inequality to investigate the blow-up behaviors of ground states to problem (\ref{minimization-problem-critical}) as $M\nearrow M^*$ when $\alpha=\frac{r}{n}$.

\section{Existence of Minimizers: Critical Mass Phenomenon}\label{sect4-criticalmass}
This section is devoted to the proof of Theorem \ref{thm11}.  More precisely, we intend to prove that  the minimization problem (\ref{minimization-problem-critical}) with energy $\mathcal E(m,w)$ being given by \eqref{41engliang}
has a minimizer $(m,w)\in \mathcal K_{M}$ if and only if $M<M^*$, where  $\mathcal K_{M}$ is defined by (\ref{constraint-set-K}). In addition, we show that there exists $(u,\lambda)\in C^2(\mathbb R^n)\times \mathbb R$ such that $(m,u,\lambda)\in W^{1,p}(\mathbb R^n)\times C^2(\mathbb R^n)\times \mathbb R$ is a solution to (\ref{125potentialfreesystem}) when $V$ is assumed to satisfy (\ref{V2mainasumotiononv}) when $r>1.$ Recall from 
the definition of $\Gamma_{\alpha^*}$ given in (\ref{sect2-equivalence-scaling}) that
 \begin{align}\label{optimalinequality415}
\int_{\mathbb R^n}m^{1+\frac{r}{n}}\,dx \leq \frac{C_L\big(1+\frac{r}{n}\big)}{(M^*)^{\frac{r}{n}}}\bigg(\int_{\mathbb R^n}\big|\frac{w}{m}\big|^{r}m\,dx\bigg)\bigg(\int_{\mathbb R^n}m\,dx\bigg)^{\frac{r}{n}},~~\forall (m,w)\in\mathcal A,
\end{align}
where $\mathcal A$ is given by (\ref{mathcalA-equivalence}).  

We would like to mention that there exists a threshold of $r$ while proving the existence of minimizers to problem (\ref{minimization-problem-critical}).  Indeed, as shown in Lemma \ref{lemma21-crucial}, if $r>n$, one can show the uniform boundedness of $m_k$ in $L^{\infty}(\mathbb R^n)$ and $C^{0,\theta}(\mathbb R^n)$ for some $\theta\in(0,1)$ if $(m_k,w_k)$ is a minimizing sequence.  However, when $r\leq n,$ we must follow the procedure shown in \cite{cesaroni2018concentration} to perform the regularization on (\ref{41engliang}) due to the worse regularity of $m$.


When $r\leq n$, we let $ \eta_{\epsilon}\geq 0$ be the standard mollifier satisfying
$$\int_{\mathbb R^n}\eta_{\epsilon}\,dx=1, \ \ \text{supp}\eta_{\epsilon}\in B_{\epsilon}(0),$$
where $\epsilon>0$ is sufficiently small, then consider the following auxiliary minimization problem:
\begin{align}\label{Eepsilonminimizerproblem1}
e_{\epsilon,M}:=\inf_{(m,w)\in\mathcal A_M}\mathcal E_{\epsilon}(m,w),
\end{align}
where ${\mathcal A}_M$ is given by (\ref{mathcalAMbegining}) and
\begin{align}\label{approxenergy}
\mathcal E_{\epsilon}(m,w):=C_L\int_{\mathbb R^n}\Big|\frac{w}{m}\Big|^{r}m\,dx+\int_{\mathbb R^n}V(x)m\,dx-\frac{1}{1+\frac{r}{n}}\int_{\mathbb R^n}(\eta_{\epsilon}\ast m)^{1+\frac{r}{n}}\, dx.
\end{align}
With the approximation energy (\ref{approxenergy}), we are able to study the existence of minimizers for \eqref{minimization-problem-critical} when $r\leq n$ by taking the limit.
 However, as discussed in \cite{cesaroni2018concentration}, it is necessary to study the uniformly boundedness of $m_{\epsilon}$ in $L^{\infty}$ when we assume $(m_{\epsilon},w_{\epsilon})$ as a minimizer for (\ref{approxenergy}).

Following the strategies shown above, we can prove Conclusion (i) stated in Theorem \ref{thm11} under the case $M<M^*$.  We would like to remark that with (\ref{V2mainassumption_2}), the assumption (V2) imposed on potential $V$, the condition $\int_{\mathbb R^n}|x|^bm\,dx<+\infty$ in (\ref{mathcalAMbegining}) must be satisfied for any minimizer.  Next, we state some crucial propositions and lemmas, which will be used in the proof of Theorem \ref{thm11}, as follows:

\begin{lemma}\label{lem4.1}
Let $p^*=\frac{np}{n-p}$ if $1\leq p<n$ and $p^*=\infty$ if $p\geq n$. Assume that $0\leq V(x)\in L_{\rm loc}^\infty(\mathbb R^n)$ satisfies $\liminf\limits_{|x|\to\infty}V(x)=\infty$ and define
$$\mathcal{W}_{p,V}:=\bigg\{m \big|\ m\in W^{1,p}(\mathbb R^n)\cap L^1(\mathbb R^n) \text{ and }\int_{\mathbb R^n}V(x)|m|\,dx<\infty\bigg\}.$$
Then, the embedding $\mathcal{W}_{p,V}\hookrightarrow L^q(\mathbb R^n)$ is compact for any $1\leq q<p^*$.
\end{lemma}
\begin{proof}
We refer the readers to \cite[Theorem XIII.67]{reed2012methods} or \cite[Theorem 2.1]{bartsch1995} for the detailed discussions.
\end{proof}

When $r\leq n,$ we have the following lemma for the uniformly boundedness of $\Vert m_{\epsilon}\Vert_{L^{\infty}}:$

{
\begin{lemma}\label{blowupanalysismlinfboundcritical}
Suppose that $V(x)$ is locally H\"{o}lder continuous and satisfies \eqref{V2mainasumotiononv}.
Let  $(u_{k},\lambda_{k},m_{k})\in C^2(\mathbb R^n)\times \mathbb R\times (L^1(\mathbb R^n)\cap L^{1+\alpha^*}(\mathbb R^n)) $ be solutions to the following systems
\begin{align}\label{eq-potential-newest}
\left\{\begin{array}{ll}
-\Delta u_k+C_H|\nabla u_k|^{r'}+\lambda_k=V-g_k[m_k], &x\in\mathbb R^n,\\
\Delta m_k+C_Hr'\nabla\cdot(m_k|\nabla u_k|^{r'-2}\nabla u_k)=0, &x\in\mathbb R^n,\\
\int_{\mathbb R^n}m_k\,dx=M,
\end{array}
\right.
\end{align}
where $\alpha^*=\frac{r}{n}$ with $1<r\leq n$, $g_k: L^1(\mathbb \R^n) \mapsto L^1(\mathbb \R^n)$ with $\theta\in(0,1)$ satisfies for all $m\in L^p(\mathbb R^n)$, $p\in[1,\infty]$ and $k\in \mathbb N$,
\begin{equation}\label{gkm-newest-estimate-crucial2024207}
\|g_k[m]\|_{L^p(\mathbb R^n)}\leq K\bigg(\|m^{\alpha^*}\|_{L^p(\mathbb R^n)}+1\bigg) \ \text{ for some }K>0,
\end{equation}
and 
\begin{equation}\label{gkm-newest-estimate-crucial202420711}
\|g_k[m]\|_{L^p(B_R(x_0))}\leq K\bigg(\|m^{\alpha^*}\|_{L^p(B_{2R}(x_0))}+1\bigg) \ \text{ for any }R>0\text{ and }x_0\in \mathbb R^n.
\end{equation}
Assume that
\begin{align}\label{assumptionsincrucialemmanewest}
\sup_k\|m_k\|_{L^1(\mathbb R^n)}<\infty,~~\sup_k\|m_k\|_{L^{1+\alpha^*}(\mathbb R^n)}<\infty,~~\sup_k\int_{\mathbb R^n}Vm_k\,dx<\infty,~~\sup_k|\lambda_k|<\infty,
\end{align}
and for all $k$, $u_k$ is bounded from below uniformly.  Then we have 
\begin{equation}\label{uniformlyboundmkinlinfty}
\limsup_{k\to\infty}\|m_k\|_{L^\infty(\mathbb \R^n)}<\infty.
\end{equation}
\end{lemma}
}

\begin{proof}
We first establish the local uniformly estimates of $m_k$.  To begin with, one finds from \eqref{assumptionsincrucialemmanewest} and (\ref{gkm-newest-estimate-crucial2024207}) that
\begin{align}
\|g_k[m]\|_{L^{1+\frac{1}{\alpha^*}}(\mathbb R^n)}\leq K\bigg(\|m^{\alpha^*}\|_{L^{1+\frac{1}{\alpha^*}}(\mathbb R^n)}+1\bigg) <\infty.
\end{align}
Then by maximal regularities shown in Theorem \ref{thmmaximalregularity} and the uniformly local H\"{o}lder continuity of $V$, we have
\begin{align}
\sup_k\Vert |\nabla u_k|^{r'}\Vert_{L^{1+\frac{1}{\alpha}}_{\text{loc}}(\mathbb R^n)}<\infty,
\end{align}
which implies
\begin{align}
\sup_{k}\Vert |\nabla u_k|^{r'-1}\Vert_{L^{{(1+\frac{1}{\alpha^*})r}}_{\text{loc}}(\mathbb R^n)}<\infty\text{~with~}\bigg(1+\frac{1}{\alpha^*}\bigg)r>n.
\end{align}
Focusing on Fokker-Planck equations, one applies Theorem 1.6.5 in \cite{bogachev2022fokker} to obtain that 
\begin{align}\label{localestimatemlinftyinlemma42crucial}
\sup_{k}\Vert m_k\Vert_{L_{\text{loc}}^\infty(\mathbb R^n)}<\infty,
\end{align}
which is local uniformly estimates satisfied by $m_k$.
\\
Next, we claim that 
\begin{align}\label{crucialandnewclaim}
\lim_{R\rightarrow +\infty}\sup_{k}\bigg\Vert \frac{m_k^{\alpha^*}}{V}\bigg\Vert_{L^\infty(\mathbb R^n\backslash B_R(0))}=0.
\end{align}
To show (\ref{crucialandnewclaim}), we argue by contradiction and assume there exist $\varepsilon>0,$ $|x_l|\rightarrow +\infty$ and $k_l\rightarrow +\infty$ such that
\begin{align}\label{411contradictionkeystuff}
\frac{m_{k_l}^{\alpha^*}}{V}(x_l)\geq \varepsilon.
\end{align}
Then we define
\begin{align}\label{rescaling2024206}
v_l(x)=a_l^{r-2}u_{k_l}(x_l+a_lx),~~\mu_l(x)=a_l^n m_{k_l}(x_l+a_lx),
\end{align}
where $a_l$ will be determined later.  Upon substituting (\ref{rescaling2024206}) into (\ref{eq-potential-newest}), one obtains 
\begin{align}\label{413rescale2024206}
\left\{\begin{array}{ll}
-\Delta v_l+C_H|\nabla v_l|^{r'}+a_l^{r}\lambda_{k_l}=a_l^{r}V(x_l+a_lx)-a^r_lg_l[ a_l^{-r/\alpha^*}\mu_l], &x\in\mathbb R^n,\\
\Delta \mu_l+C_Hr'\nabla\cdot(|\nabla v_l|^{r'-2}\nabla v_l\mu_l)=0,&x\in\mathbb R^n.
\end{array}
\right.
\end{align}
Choosing $a_l^{r}=\frac{1}{V(x_l)}$, one finds from (\ref{V2mainassumption_1}) that  
\begin{align}\label{eq5.17}
a_l^{r}=\frac{1}{V(x_l)}\rightarrow 0,
\end{align}
where $|x_l|\rightarrow \infty$.  In light of the assumption (\ref{V2mainassumption_2}), we have 
\begin{align}\label{mildassumptiononVinlemmacritical}
\Vert a_l^{r}V_{k_l}(x_l+a_l x)\Vert_{L^\infty(B_1(0))}\leq C_2.
\end{align}
In addition, we combine (\ref{gkm-newest-estimate-crucial2024207}) with (\ref{gkm-newest-estimate-crucial202420711}) to obtain for $l$ large,
\begin{align}\label{418combinecrucialnewest1}
\Vert a^r_l g_l[\mu_l a_l^{-r/\alpha^*}]\Vert_{L^{{1+\frac{1}{\alpha^*}}}(B_1(0))}= &{a^r_l}\Vert g_l[\mu_l a_l^{-r/\alpha^*}]\Vert_{L^{{1+\frac{1}{\alpha^*}}}(B_1(0))}\nonumber\\
\leq & a^r_lK(\Vert \mu_l^{\alpha^*} a_l^{-r}\Vert_{L^{1+\frac{1}{\alpha^*}}(B_2(0))}+1)\leq K\Vert \mu_l^{\alpha^*} \Vert_{L^{1+\frac{1}{\alpha^*}}(B_2(0))}+1.
\end{align}
On the other hand, from \eqref{assumptionsincrucialemmanewest} and \eqref{eq5.17}, one has the fact that
\begin{align}\label{419combinecrucialnewest2}
\Vert \mu_l^{\alpha^*}\Vert^{1+\frac{1}{\alpha^*}}_{L^{1+\frac{1}{\alpha^*}}(B_2(0))}=a_l^r\Vert m_{k_l}\Vert_{L^{1+\alpha^*}(B_{2a_l}(x_l))}\rightarrow 0\text{~as~}l\rightarrow +\infty.
\end{align}
We combine \eqref{mildassumptiononVinlemmacritical}, (\ref{418combinecrucialnewest1}) with (\ref{419combinecrucialnewest2}) and similarly use the maximal regularities shown in \cite{cirant2021problem} to get 
\begin{align*}
\Vert |\nabla v_l|^{r'}\Vert_{L^{1+\frac{1}{\alpha^*}}(B_{{1}/{2}}(0))}\leq C, \text{~for~}l~\text{large},
\end{align*}
where constant $C>0$ independent of $l.$  Then focusing on the second equation of (\ref{413rescale2024206}), we similarly apply the standard elliptic regularity estimates (See Theorem 1.6.5 in \cite{bogachev2022fokker}) to obtain $\mu_l\in C^{0,\theta}(B_{1/4}(0))$ with $\theta\in(0,1)$ independent of $l$.  With the local H\"{o}lder's continuity of $\mu_l$, we have from (\ref{411contradictionkeystuff}) that 
\begin{align*}
\mu^{\alpha^*}_l(0)=m_{k_l}^{\alpha^*}(x_l)a_l^{r}=\frac{m_{k_l}^{\alpha^*}}{V}(x_l)\geq \varepsilon,
\end{align*}
which implies there exists $\delta>0$ such that 
\begin{align*}
\mu_l\geq \delta \text{~in~}B_R(0),
\end{align*}
where $R>0$ is independent of $l.$ Then it follows that if $R$ is small, 
\begin{align*}
\int_{\mathbb R^n}m_{k_l} V\,dx\geq& \delta a_l^{-\frac{r}{\alpha^*}}\int_{B_{a_lR}(0)}V(x_l+y)\,dy
\geq \frac{\delta}{2}a_l^{-n}V(x_l)|B_{a_lR}(0)|
\geq C\delta V(x_l)\rightarrow +\infty,
\end{align*}
where $C>0$ is some constant.  This is contradicted to (\ref{assumptionsincrucialemmanewest}) then finishes the proof of claim (\ref{crucialandnewclaim}).
\\
With the aid of (\ref{crucialandnewclaim}), we find there exists constant $C>0$ such that 
\begin{align}\label{wefind2024207}
|V-g_k[m_k]|\leq C (V+1).
\end{align}
Proceeding the same argument shown in the proof of Lemma \ref{sect2-lemma21-gradientu}, one obtains  
\begin{align}\label{nablaukVknewestcriticalbound}
|\nabla u_k|\leq C(1+V^{\frac{1}{r'}}),
\end{align}
where $C>0$ is some constant.  
\\
Finally, we establish the global uniformly estimate of $m$ in $L^\infty$, i.e. prove \eqref{uniformlyboundmkinlinfty}.  To this end, we set $\phi_k=u^p_k$ for $p>1$ and show $\phi_k$ are Lyapunov functions.  Indeed, since $u_k$ solve HJ equations, we have
\begin{align*}
-\Delta \phi_k+C_Hr'|\nabla u_k|^{r'-2}\nabla u_k\cdot\nabla\phi_k=pu_k^{p-1}G_k(x),
\end{align*}
where $G_k(x)$ are defined by
\begin{align}\label{GKxuniformlylowerboundcriticalnew}
G_k(x):=-(p-1)\frac{|\nabla u_k|^{2}}{u_k}-C_H|\nabla u_k|^{r'}+C_Hr'|\nabla u_k|^{r'}-\lambda_k+V-g_k[m_k].
\end{align}
To estimate (\ref{GKxuniformlylowerboundcriticalnew}), we use (\ref{wefind2024207}) and (\ref{crucialandnewclaim}) to find
\begin{align*}
G_k(x)\geq (p-1)|\nabla u_k|^{r'}\bigg(\frac{C_H(r'-1)}{(p-1)}-\frac{|\nabla u_k|^{2-r'}}{u_k}\bigg)-\lambda_k+CV\geq 1\text{~for all~}|x|>R,
\end{align*}
where $C>0$ and $R>0$ independent of $k.$ 
 Similarly as results obtained in \cite{metafune2005global}, one gets 
 \begin{align}\label{upperbounduniformlyestimateinlemma42}
 \sup\limits_{k}\int_{\mathbb R^n} m_k|V|^p\,dx<+\infty,~\text{and}~\sup\limits_{k}\int_{\mathbb R^n} m_k|\nabla u_k|^p\,dx<+\infty,
 \end{align}
 for large $p>1$.  Next, we perform the global estimates of $m_k$ in $L^q$ with any $q>1$.  In fact, we claim that 
 \begin{align}\label{mkglobalestimatesinlqwholespace}
 \sup_{k}\int_{\mathbb R^n}m_k^q\,dx \leq C_q<+\infty,~\forall q>1,
 \end{align}
 where $C_q>0$ is some constant.  Thanks to (\ref{crucialandnewclaim}) and (\ref{localestimatemlinftyinlemma42crucial}), one has for any $k$, $\exists R_0>0$ such that 
\begin{align}\label{431upperboundmkcr0anyxgreatR0}
m_k(x)\leq C_{R_0}V^{\frac{1}{\alpha^*}},~~\forall |x|>R_0,
\end{align}
where $C_{R_0}>0$ is some constant.  Moreover, (\ref{431upperboundmkcr0anyxgreatR0}) together with (\ref{upperbounduniformlyestimateinlemma42}) gives us
\begin{align}\label{432incrucialemmaavoidblowup}
\int_{B_{R_0}^c}m_k^q\,dx=\int_{B_{R_0}^c}m_km_k^{q-1}\,dx\leq C_{R_0}^{q-1}\int_{\mathbb R^n}m_kV^{\frac{q-1}{\alpha^*}}\,dx\leq C_{R_0,q},
\end{align}
where $C_{R_0,q}$ is some positive constant.  Moreover, it follows from (\ref{localestimatemlinftyinlemma42crucial}) that 
\begin{align}\label{432incrucialemmaavoidblowup2}
\int_{B_{R_0}(0)}m_k^q\,dx\leq \Vert m_k\Vert^q_{L^\infty(B_{R_0}(0))}|B_{R_0}(0)|\leq {\tilde C}_{R_0,q},
\end{align}
where ${\tilde C}_{R_0,q}$ is some positive constant.  Hence, claim  (\ref{mkglobalestimatesinlqwholespace}) holds by invoking (\ref{432incrucialemmaavoidblowup}) and (\ref{432incrucialemmaavoidblowup2}).  We next prove for any $ \varphi\in C_c^{\infty}(\mathbb R^n)$ and $q>1$ that 
\begin{align}\label{434expectationtheglobalboundmkinlqdual}
\bigg|\int_{\mathbb R^n}m_k|\nabla u_k|^{r'-2}\nabla u_k\nabla\varphi\,dx\bigg|\leq C_q\Vert \nabla \varphi\Vert_{L^{q'}(\mathbb R^n)},
\end{align}
where $C_q>0$ is some constant.  Indeed, for any fixed $q>1$, we use H\"{o}lder's inequality to get that 
\begin{align}\label{threetermsholdersinequalitynew}
\int_{\mathbb R^n} m_k|\nabla u_k|^{r'-1}|\nabla\varphi|\,dx=&\int_{\mathbb R^n} m_k^{\frac{1}{p'}}m_k^{\frac{1}{p}}|\nabla u_k|^{r'-1}|\nabla\varphi|\,dx\nonumber\\
\leq &\Big(\int_{\mathbb R^n}m_k^{\frac{s}{p'}}\,dx\Big)^{\frac{1}{s}}\cdot\Big(\int_{\mathbb R^n}m_k|\nabla u_k|^{(r'-1)p}\,dx\Big)^{\frac{1}{p}}\cdot\Big(\int_{\mathbb R^n}|\nabla \varphi|^{q'}\,dx\Big)^{\frac{1}{q'}},
\end{align}
where $p$ and $s$ are chosen such that $p>\max\{q,\frac{1}{r'-1}\}$ and $s>p'$ with
\begin{align}
\frac{1}{s}+\frac{1}{p}=\frac{1}{q},~\frac{1}{p'}+\frac{1}{p}=1.
\end{align}
We combine (\ref{upperbounduniformlyestimateinlemma42}) with (\ref{mkglobalestimatesinlqwholespace}) and obtain from (\ref{threetermsholdersinequalitynew}) that \eqref{434expectationtheglobalboundmkinlqdual} holds.  On the other hand, we apply the integration by parts on Fokker-Planck equations to get
\begin{align}\label{combineintegralbypartsmsecondequation}
\int_{\mathbb R^n}m_k\Delta \varphi\,dx=C_Hr'\int_{\mathbb R^n}m_k|\nabla u_k|^{r'-2}\nabla u_k\cdot \nabla \varphi\,dx.
\end{align}
Combining (\ref{combineintegralbypartsmsecondequation}) with (\ref{434expectationtheglobalboundmkinlqdual}), we use Lemma \ref{proposition-lemma21-FP} to obtain
\begin{align}\label{this439inlemmacrucial2024213}
\sup_{k}\Vert\nabla m_k\Vert_{W^{1,q}(\mathbb R^n)}\leq C_q<+\infty,~\forall q>1.
\end{align}
(\ref{this439inlemmacrucial2024213}) together with (\ref{mkglobalestimatesinlqwholespace}) implies
\begin{align}
\sup_{k}\Vert m_k\Vert_{W^{1,q}(\mathbb R^n)}\leq C_q<+\infty.
\end{align}
By choosing $q>n$ and using the Sobolev embedding theorem, we find (\ref{uniformlyboundmkinlinfty}) holds. 
\end{proof}

{
We remark that the conclusion stated in Lemma \ref{blowupanalysismlinfboundcritical} also holds if $\alpha^*$ is replaced by $\alpha$ and satisfies $0<\alpha<\frac{r}{n}.$  Furthermore, it is worthy mentioning that our approach employed in the proof of Lemma \ref{blowupanalysismlinfboundcritical} is distinct from the blow-up analysis shown in  \cite[Theorem 4.1]{cesaroni2018concentration}.  With our novel arguments, we are able to relax the polynomial growth assumption on $V$ and analyze the existence and blow-up behaviors of minimizers under a class of $V$ satisfying (\ref{V2mainasumotiononv}) when $r\leq n.$ }

We next focus on the case of $r>n$ and give some preliminary results for the proof of Theorem \ref{thm11}.  First of all, if problem (\ref{minimization-problem-critical}) is attained by a minimizer $(m,w)$, then we have
\begin{proposition}\label{prop41relationminimiersolutionsect4}
Assume that $V$ satisfies (V1) and (V2) stated in Subsection \ref{mainresults11}. Let
 $(m,w)\in \mathcal{K}_M$ be a  minimizer to problem (\ref{minimization-problem-critical}). Then there exists a solution $(u,m,\lambda)$ to  (\ref{125potentialfreesystem}), which satisfies
\begin{align}\label{47lowerboundgeneralvu}
|\nabla u(x)|\leq C\Big(1+V^{\frac{1}{r'}}(x)\Big),~~u(x)\geq CV^{\frac{1}{r'}}(x)-C^{-1}, x\in \mathbb R^n,
\end{align}
for some constant $C>0.$
\end{proposition}
\begin{proof}
The proof  is similar as shown in  \cite[Proposition 3.4]{cesaroni2018concentration}.   Define the  test function space
\begin{align}\label{mathcalhatAfurtherseeprop3611}
\mathcal {  B}:=\bigg\{ \psi\in C^2(\mathbb R^n)\bigg|\limsup_{|x|\rightarrow \infty}\frac{|\nabla\psi(x)|}{V^{\frac{1}{r'}}(x)}<+\infty,~\limsup_{|x|\rightarrow \infty}\frac{|\Delta\psi(x)|}{V(x)}<+\infty\bigg\}.
\end{align}
By using the condition \eqref{V2mainassumption_3}, one can  obtain from (\ref{mathcalhatAfurtherseeprop3611}) that
\begin{equation}\label{eq4.9}
\limsup_{|x|\rightarrow \infty}\frac{|\psi(x)|}{|x|V^{\frac{1}{r'}}}<+\infty.
\end{equation}
We claim that
\begin{align}\label{withourclaimcirantnewprop34}
-\int_{\mathbb R^n}m\Delta \psi\,dx=\int_{\mathbb R^n}w\cdot\nabla \psi\,dx,\ ~
\text{for all  $\psi\in\mathcal {B}.$}\end{align}
Similarly as in \cite[Proposition 3.4]{cesaroni2018concentration}, we consider a radial smooth cutoff function $\chi$ satisfying
$\chi(x)=1$ if $x\in B_1(0),$ and $\chi(x)=0$ if $x\in B_2^c(0)$.
Define $\chi_R(x):=\chi\big(\frac{x}{R}\big),$ then we have $|\nabla \chi_R(x)|\leq CR^{-1}$ and $\Delta \chi_R(x)|\leq CR^{-2}$ for some $C>0.$  Noting that $\Delta m=\nabla\cdot w$ in the weak sense, we test the equation against $\psi\chi_R$ with $\psi\in\mathcal{B}$ and integrate it by parts to get
\begin{align}\label{ibp49beforetakelimitcirant}
-\int_{\mathbb R^n}m(\chi_R\Delta \psi+2\nabla\psi\cdot \nabla \chi_R+\psi\Delta \chi_R)\,dx=\int_{B_{2R}}w\cdot(\chi_R\nabla \psi+\psi\nabla \chi_R)\,dx.
\end{align}
Since $(m,w)\in \mathcal{K}_{M}$ is a minimizer of (\ref{minimization-problem-critical}), we obtain $\int_{\mathbb R^n}Vm\,dx<\infty $ and 
 \begin{equation*}
\int_{\mathbb R^n}|w|V^{\frac{1}{r'}}\,dx=\int_{\mathbb R^n}|w|m^{-\frac{1}{r'}}m^{\frac{1}{r'}}V^{\frac{1}{r'}}\,dx\leq \left(\int_{\mathbb R^n}\Big|\frac{w}{m}\Big|^{r}m\,dx\right)^r\left(\int_{\mathbb R^n}mV\,dx\right)^{r'}<\infty.\end{equation*}
We thus get from \eqref{mathcalhatAfurtherseeprop3611} that 
\begin{align*}
\int_{\mathbb R^n}|w\nabla\psi|\,dx\leq C\int_{\mathbb R^n}|w|V^{\frac{1}{r'}}\,dx<\infty,~~\int_{\mathbb R^n}m|\Delta \psi|\,dx\leq C\int_{\mathbb R^n}mV\,dx<\infty.
\end{align*}
Moreover, it follows  from  \eqref{mathcalhatAfurtherseeprop3611} and \eqref{eq4.9} that
\begin{align*}
\int_{R\leq |x|\leq 2R} m|\psi||\Delta\chi_R|\,dx,
\int_{\mathbb R^n}m|\nabla\psi||\nabla\chi_R|\,dx\leq \frac{C}{R}\int_{R\leq |x|\leq 2R}mV^{\frac{1}{r'}}\,dx\overset{R\to\infty}\longrightarrow 0,
\end{align*}
and
\begin{align*}
\int_{\mathbb R^n}|w||\psi||\nabla\chi_R|\,dx\leq \frac{C}{R}\int_{R\leq |x|\leq 2R}|w|V^{\frac{1}{r'}}|x|\,dx\leq 2C\int_{R\leq |x|\leq 2R}wV^{\frac{1}{r'}}\,dx\overset{R\to\infty}\longrightarrow 0.
\end{align*}
 Upon collecting the above two estimates,  we take the limit in (\ref{ibp49beforetakelimitcirant}) to prove our claim.

With the claim (\ref{withourclaimcirantnewprop34}), we can follow the subsequent arguments shown in \cite[Proposition 3.4]{cesaroni2018concentration} to complete the proof of this proposition and the detailed argument is omitted.
\end{proof}
With the aid of Proposition \ref{prop41relationminimiersolutionsect4}, we study the regularity of $m$ when $r>n$, which is
\begin{proposition}\label{prop41relationminimiersolutionsect4-next}
Let $(m,w, u, \lambda)$ be the solution given in Proposition \ref{prop41relationminimiersolutionsect4}, then  $m\in W^{1,p}(\mathbb R^n)$ for all $p>1.$
\end{proposition}
\begin{proof}
The proof is based on the argument shown in \cite[Proposition 3.6]{cesaroni2018concentration}.  In light of \eqref{47lowerboundgeneralvu}, we may assume that $u(x)\geq1$ and set $\phi=u^p$ for $p>1$.  Since $u$ solves the HJ equation, we have
\begin{align}
-\Delta \phi+C_Hr'|\nabla u|^{r'-2}\nabla u\cdot\nabla\phi=pu^{p-1}G(x),
\end{align}
where $G(x)$ is defined by
\begin{align}
G(x):=-(p-1)\frac{|\nabla u|^{2}}{u}-C_H|\nabla u|^{r'}+C_Hr'|\nabla u|^{r'}-\lambda+V-m^{\frac{r}{n}}.
\end{align}
Noting that $m\in W^{1,\hat q}(\mathbb R^n)\hookrightarrow L^{\infty}(\mathbb R^n)$ for $r>n$,  we find
\begin{align}
G(x)\geq (p-1)|\nabla u|^{r'}\bigg(\frac{C_H(r'-1)}{(p-1)}-\frac{|\nabla u|^{2-r'}}{u}\bigg)-C+V\geq 1\text{~for all~}|x|>R.
\end{align}
In view of the results obtained in \cite{metafune2005global}, we use (\ref{47lowerboundgeneralvu}) to get $\int_{\mathbb R^n} m|V|^p\,dx<+\infty$ for all $p>1$.  Then one finds $w=-C_Hr'm|\nabla u|^{r'-2}\nabla u\in L^p(\mathbb R^n)$ for all $p>1$  by using \eqref{47lowerboundgeneralvu} and the fact $m\in L^{\infty}(\mathbb R^n)$.  Finally, it follows from Lemma \ref{proposition-lemma21-FP} that $m\in W^{1,p}$ for all $p>1$.
\end{proof}

With the preliminary results shown above, we are ready to prove Theorem \ref{thm11}.

\medskip
\textbf{Proof of Theorem \ref{thm11}:}
\begin{proof}
Firstly, we shall show Conclusion (ii) in Theorem \ref{thm11}.  Recall that $\big(m_{\alpha^*},w_{\alpha^*}, u_{\alpha^*}\big)$ given in Theorem \ref{thm11-optimal} is a minimizer of problem (\ref{sect2-equivalence-scaling}) with $\alpha=\alpha^*=\frac{r}{n}$.  For the simplicity of notations, we rewrite $(m_{\alpha^*},w_{\alpha^*},u_{\alpha^*})$ as $(m_*,w_*,u_*)$, then define
\begin{align}\label{scalingwitht}
( m_{*}^t, w_*^t)=\bigg(\frac{M}{M^*}t^n m_*(t(x-x_0)),\frac{M}{M^*}t^{n+1} w_*(t(x-x_0))\bigg)\in \mathcal K_{M}, \ \ \forall t>0,~ x_0\in\mathbb R^n.
\end{align}
where  the constraint set $\mathcal K_{M}$ and $M^*>0$ are  defined by \eqref{constraint-set-K}  and (\ref{Mstar-critical-mass}), respectively.
Recall that $u_* \in C^2(\mathbb R^n)$ and $m_*$ decays exponentially   as stated in Theorem \ref{thm11-optimal}, then we employ Lemma \ref{poholemma} to get
\begin{align}\label{invoke45}
C_L\int_{\mathbb R^n}\bigg|\frac{ w_*}{ m_*}\bigg|^{r} m_*\,dx=\frac{n}{n+r}\int_{\mathbb R^n}m_*^{1+\frac{r}{n}}\,dx.
\end{align}
Invoking \eqref{invoke45}, we substitute (\ref{scalingwitht}) into (\ref{41engliang}) to obtain that if $M>M^*,$
\begin{align}\label{supercriticalmasscase}
e_{\alpha^*,M}\leq \mathcal E(m_*^t,w_*^t)=&\frac{M}{M^*}\bigg(C_Lt^{r}\int_{\mathbb R^n}\Big|\frac{ w_*}{ m_*}\Big|^{r} m_*\,dx+\int_{\mathbb R^n}V(x)m_*\, dx\bigg)-\frac{t^r}{1+\frac{r}{n}}\bigg(\frac{M}{M^*}\bigg)^{1+\frac{r}{n}}\int_{\mathbb R^n}m_*^{1+\frac{r}{n}}\,dx\nonumber\\
=&\frac{M}{M^*}\Big[1-\bigg(\frac{M}{M^*}\bigg)^{\frac{r}{n}}\Big]\frac{t^{r}}{1+\frac{r}{n}}\int_{\mathbb R^n} m_*^{1+\frac{r}{n}}\,dx +MV(x_0)+o_t(1)\nonumber\\
&\rightarrow -\infty \ \ \text{as~}t\rightarrow +\infty.
\end{align}
It immediately follows that $e_{\alpha^*,M}=-\infty$ for $M>M^*,$ there is no minimizer to problem (\ref{minimization-problem-critical}).

\vskip.1truein

To prove Conclusion (i) in Theorem \ref{thm11}, we divide the argument into two cases: $r\leq n$ and $r>n$.  For the former case, we first consider the auxiliary problem (\ref{Eepsilonminimizerproblem1}), 
with $\mathcal E_{\epsilon}(m,w)$ being given by (\ref{approxenergy}).
By using Young's inequality for convolution, we have the fact that
\begin{align}\label{inlightofsect4}
\int_{\mathbb R^n}m^{1+\frac{r}{n}}\,dx \geq \int_{\mathbb R^n}(m\ast\eta_{\epsilon})^{1+\frac{r}{n}}\,dx \overset{\epsilon\to0^+}\longrightarrow \int_{\mathbb R^n}m^{1+\frac{r}{n}}\,dx~ \text{ for any } m\in L^{1+\frac{r}{n}}(\mathbb R^n).
\end{align}
Then, one finds from \eqref{optimalinequality415}  that  
\begin{align}\label{418cruciallowerbound}
\mathcal E_{\epsilon}(m,w)\geq\mathcal E(m,w)\geq \Big[\Big(\frac{M^*}{M}\Big)^{\frac{r}{n}}-1\Big]\frac{n}{n+r}\int_{\mathbb R^n} m^{1+\frac{r}{n}}\,dx+\int_{\mathbb R^n}V(x)m\,dx.
\end{align}

Next, we show that there exist minimizers to problem (\ref{Eepsilonminimizerproblem1}).
Let $(m_{\epsilon,k},w_{\epsilon,k})\in\mathcal K_M$ be a minimizing sequence of (\ref{Eepsilonminimizerproblem1}).   Noting that if we take  $$(\hat m,\hat w):=\left(\frac{\|e^{-|x|}\|_{L^1(\mathbb R^n)}}{M} e^{-|x|}, \frac{\|e^{-|x|}\|_{L^1(\mathbb R^n)}}{M} \frac{x}{|x|}e^{-|x|}\right)\in\mathcal K_{M}, $$
then,
$$e_{\epsilon, M}\leq C_L\int_{\mathbb R^n}\Big|\frac{\hat w}{\hat m}\Big|^{r}\hat m\,dx+\int_{\mathbb R^n}V(x) \hat m\,dx<+\infty,$$
which implies that there exists $C>0$ independent of  $\epsilon$ such that
\begin{equation}\label{eq4.24}
\lim_{k\to\infty}\mathcal{E}_{\epsilon}(m_{\epsilon,k},w_{\epsilon,k})=e_{\epsilon, M}<C<+\infty.
\end{equation}
Since $M<M^*$, one concludes from \eqref{optimalinequality415}, (\ref{418cruciallowerbound}) and \eqref{eq4.24} that
\begin{align}\label{suanzitouguji}
\sup_{k\in\mathbb N^+}\int_{\mathbb R^n}m_{\epsilon,k}^{1+\frac{r}{n}}\,dx\leq C<+\infty,~~\sup_{k\in\mathbb N^+}\int_{\mathbb R^n}\bigg(\big|\frac{w_{\epsilon,k}}{m_{\epsilon,k}}\big|^{r}m_{\epsilon,k}+V(x)m_{\epsilon,k}\bigg)\,dx\leq C<+\infty,
\end{align}
where $C>0$ is  independent of  $\epsilon$.  The subsequent argument for proving Conclusion (i) is similar as discussed in \cite{cesaroni2018concentration}, for the sake of completeness, we give the proof briefly.  Indeed, with the aid of the key Lemma \ref{lemma21-crucial}, we obtain from (\ref{suanzitouguji}) that
\begin{align}\label{411mepsilonk}
\sup_{k\in\mathbb N^+}\Vert m_{\epsilon,k}\Vert_{W^{1,\hat q}(\mathbb R^n)}\leq C<+\infty \ \text{ and } \ \sup_{k\in\mathbb N^+}\Vert w_{\epsilon,k}\Vert_{L^{p}(\mathbb R^n)}\leq C<+\infty,\ \ \text{for any }p\in[1,\hat q],
\end{align}
where $\hat q$ is given in (\ref{hatqconstraint}) and $C>0$ is some constant independent of $\epsilon$. As a consequence, there exists $(m_{\epsilon},w_{\epsilon})\in W^{1,\hat q}(\mathbb R^n)\times  L^{\hat q}(\mathbb R^n)$ such that
\begin{align}\label{convergencesect4firstone}
(m_{\epsilon,k},w_{\epsilon,k})\overset{k}\rightharpoonup (m_{\epsilon},w_{\epsilon}) \text{~in~}  W^{1,\hat q}(\mathbb R^n)\times L^{\hat q}(\mathbb R^n).
\end{align}
In light of the assumption (V1), $\lim\limits_{|x|\rightarrow \infty}V(x)=+\infty$, given in Subsection \ref{mainresults11}, one can deduce from Lemma \ref{lem4.1} that
\begin{align}\label{convergencesect4secondone}
m_{\epsilon,k}\overset{k}\rightarrow m_{\epsilon} \text{~in~} L^1(\mathbb R^n)\cap L^{1+\frac{r}{n}}(\mathbb R^n) .
\end{align}
Therefore, up to a subsequence,
\begin{equation}\label{eq4.26}
\int_{\mathbb R^n}(\eta_{\epsilon}\ast m_{\epsilon,k})^{1+\frac{r}{n}}\,dx\overset{k}\rightarrow \int_{\mathbb R^n}(\eta_{\epsilon}\ast m_{\epsilon})^{1+\frac{r}{n}}\,dx.\end{equation}
In addition, thanks to the convexity of $\int_{\mathbb R^n}\big|\frac{w}{m}\big|^{r}m\,dx $,  by letting $k\to\infty$ in \eqref{suanzitouguji}, we see that there exists $C>0$  independent of  $\epsilon>0$ such that
\begin{align}\label{eq4.32}
\int_{\mathbb R^n}\Big|\frac{w_{\epsilon}}{m_{\epsilon}}\Big|^rm_{\epsilon}\,dx+\int_{\mathbb R^n}V(x)m_{\epsilon}\,dx\leq\liminf_{k\rightarrow +\infty }\int_{\mathbb R^n}\big|\frac{w_{\epsilon,k}}{m_{\epsilon,k}}\big|^{r}m_{\epsilon,k}\,dx+\int_{\mathbb R^n}V(x)m_{\epsilon,k}\,dx\leq C<+\infty.
\end{align}
It then follows that 
 \begin{equation}\label{eq4.33}
 \int_{\mathbb R^n}|w_\epsilon|V^{\frac{1}{r'}}\,dx\leq \left(\int_{\mathbb R^n}\Big|\frac{w_\epsilon}{m_\epsilon}\Big|^{r}m_\epsilon\,dx\right)^r\left(\int_{\mathbb R^n}V m_\epsilon \,dx\right)^{r'}\leq C<\infty.
 \end{equation}
 and 
  \begin{equation}\label{eq4.330}
 \int_{\mathbb R^n}|w_\epsilon|\,dx\leq \left(\int_{\mathbb R^n}\Big|\frac{w_\epsilon}{m_\epsilon}\Big|^{r}m_\epsilon\,dx\right)^r\left(\int_{\mathbb R^n} m_\epsilon \,dx\right)^{r'}\leq C<\infty.
 \end{equation}
From \eqref{convergencesect4firstone}, \eqref{convergencesect4secondone} and \eqref{eq4.330} we deduce that $(m_{\epsilon},w_{\epsilon})\in \mathcal K_{M}$. Moreover, it follows from \eqref{eq4.26} and \eqref{eq4.32} that
\begin{align*}
e_{\epsilon, M}=\lim_{k\rightarrow\infty}\mathcal E_{\epsilon}(m_{\epsilon,k},w_{\epsilon,k})\geq \mathcal E_{\epsilon}(m_{\epsilon},w_{\epsilon})\geq e_{\epsilon,M},
\end{align*}
which indicates $(m_{\epsilon},w_{\epsilon})\in \mathcal K_{M}$ is a minimizer of problem (\ref{Eepsilonminimizerproblem1}).  {Finally, similarly as the proof of Proposition 3.4 in \cite{cesaroni2018concentration} and the arguments shown in Proposition \ref{prop41relationminimiersolutionsect4} and Proposition \ref{prop41relationminimiersolutionsect4-next}, we apply Lemma \ref{lemma22preliminary} to obtain that}
there exists $u_{\epsilon}\in C^2(\mathbb R^n)$ bounded from below (depending on $\epsilon$)  and $\lambda_{\epsilon}\in \mathbb R$ such that
\begin{align}\label{eq4.31}
\left\{\begin{array}{ll}
-\Delta u_{\epsilon}+C_H|\nabla u_{\epsilon}|^{r'}+\lambda_{\epsilon}=V(x)-(\eta_{\epsilon}\ast m_{\epsilon})^{\frac{r}{n}}\ast \eta_{\epsilon},\\
\Delta m_{\epsilon}+C_Hr'\nabla\cdot (m_{\epsilon}|\nabla u_{\epsilon}|^{r'-2}\nabla u_{\epsilon})=0,\\
w_{\epsilon}=-C_Hr'm_{\epsilon}|\nabla u_{\epsilon}|^{r'-2}\nabla u_{\epsilon}, \ \int_{\mathbb R^n}m_{\epsilon}\,dx=M<M^*.
\end{array}
\right.
\end{align}

For each fixed $\epsilon>0$, we deduce from Lemma \ref{sect2-lemma21-gradientu} that, there exists $C_\epsilon>0$ depends on $\epsilon$ such that
$|\nabla u_{\epsilon}(x)|\leq C_\epsilon(1+V(x))^{\frac{1}{r'}}$. Since $u_\epsilon\in C^2(\mathbb R^n)$ satisfies the first equation of \eqref{eq4.31} in classical sense and $(\eta_{\epsilon}\ast m_{\epsilon})^{\frac{r}{n}}\ast \eta_{\epsilon}\in L^\infty(\mathbb R^n)$, we deduce  that $|\Delta u_{\epsilon}(x)|\leq C_\epsilon(1+V(x))$.  We claim that  
\begin{equation}\label{eq4.34}
|\lambda_\epsilon|\leq C<\infty, \text{ with $C>0$   independent of  $\epsilon>0$}.
\end{equation}
By using the $m$-equation of \eqref{eq4.31}, we have from \eqref{eq4.33} that
\begin{align*}
\int_{\mathbb R^n}m\Delta u_{\epsilon}\,dx=\int_{\mathbb R^n}w_{\epsilon}\cdot\nabla u_{\epsilon}\,dx=-C_Hr'\int_{\mathbb R^n}m_{\epsilon}|\nabla u_{\epsilon}|^{r'}\nabla u_{\epsilon}\,dx.
\end{align*}
In addition, we test the $u$-equation of \eqref{eq4.31} against $m_\epsilon$ and integrate to get
\begin{equation}\label{eq4.301}
\begin{split}
\lambda_\epsilon M&=-(1-r')C_H\int_{\mathbb R^n}m_\epsilon|\nabla u_\epsilon|^{r'}\,dx+\int_{\mathbb R^n}V m_\epsilon\,dx-\int_{\mathbb R^n}(\eta_{\epsilon}\ast m_{\epsilon})^{1+\frac{r}{n}}\, dx\\
&=C_L\int_{\mathbb R^n}m_\epsilon\bigg|\frac{w_\epsilon}{m_\epsilon}\bigg|^{r}\, dx+\int_{\mathbb R^n}V m_\epsilon\,dx-\int_{\mathbb R^n}(\eta_{\epsilon}\ast m_{\epsilon})^{1+\frac{r}{n}}\, dx\end{split}
\end{equation}
where we have used the fact that $C_L=\frac{1}{r}(r'C_H)^{\frac{1}{1-r'}}$ in the second equality.
Finally, we collect \eqref{eq4.32}, \eqref{eq4.33} and \eqref{eq4.301} to find (\ref{eq4.34}) holds.

{
Next, we let $\epsilon\rightarrow 0$ and shall find the minimizer $(m_{M},w_{M})$ to problem (\ref{minimization-problem-critical}).  Noting $(m_{\epsilon},u_{\epsilon},\lambda_{\epsilon})$ satisfies (\ref{assumptionsincrucialemmanewest}) with $k$ replaced by $\epsilon.$  
We apply Young's inequality to get
 \begin{align*}
\sup_{k}\Vert (\eta_k\ast m_k)^{\alpha^*}\ast \eta_{k}\Vert_{L^{1+\frac{1}{\alpha^*}}(\mathbb R^n)}\leq \sup_{k}\Vert  m_k^{\alpha^*}\Vert_{L^{1+\frac{1}{\alpha^*}}(\mathbb R^n)}<\infty,
\end{align*}
and
 \begin{align*}
\sup_{k}\Vert (\eta_k\ast m_k)^{\alpha^*}\ast \eta_{k}\Vert_{L^{1+\frac{1}{\alpha^*}}(B_R(x_0))}\leq \sup_{k}\Vert  m_k^{\alpha^*}\Vert_{L^{1+\frac{1}{\alpha^*}}(B_{2R}(x_0))}<\infty,~\text{for~any~}x_0\in\mathbb R^n\text{ and }R\text{ large.}
\end{align*}}
Then, with the aid of \eqref{eq4.32} and \eqref{eq4.34}, we can invoke Lemma \ref{blowupanalysismlinfboundcritical} to derive that
 \begin{equation}\label{eq4.35}
 \limsup_{\epsilon\to 0^+}\|m_\epsilon\|_{L^\infty(\mathbb R^n)}<\infty.
 \end{equation}
Then, by using Lemma \ref{sect2-lemma21-gradientu}, we find
 \begin{equation}\label{eq4.36}
 |\nabla u_{\epsilon}(x)|\leq C(1+V(x))^{\frac{1}{r'}}, \text{ where  $C>0$ is independent of $\epsilon$.}
 \end{equation}
 Since $u_{\epsilon}$ is bounded from below,  without loss of generality, we assume that $u_{\epsilon}(0)=0$.
Thanks to \eqref{29uklemma22}, one finds that $u_{\epsilon}(x)\geq C_\epsilon V^{\frac{1}{r'}}(x)-C_\epsilon\to +\infty$ as $|x|\to+\infty$, which indicates each $u_\epsilon(x)\in C^2( \mathbb R^n)$ admits its minimum at some finite point $x_{\epsilon}$. using  \eqref{eq4.34}, \eqref{eq4.35} and the coercivity of $V$, we can deduce from the $u$-equation of \eqref{eq4.31}  that $x_{\epsilon}$ is uniformly bounded with respect to $\epsilon$.  It then follows from $u_{\epsilon}(0)=0$ and \eqref{eq4.36} that
there exists $C>0$ independent of $\epsilon$ such that
\begin{equation*}
-C\leq u_\epsilon(x)\leq C|x|(1+V(x))^{\frac{1}{r'}} \text{ for all } x\in\mathbb R^n,
\end{equation*}
where we have used \eqref{V2mainassumption_3} in the second inequality.
Since $u_{\epsilon}$ are bounded from below uniformly, one can employ Lemma \ref{lowerboundVkgenerallemma22} to get that
 $u_{\epsilon}(x)\geq CV^{\frac{1}{r'}}(x)-C \text{ with $C>0$ independent of $\epsilon$.}$
 In summary, under the assumptions \eqref{V2mainasumotiononv}, we get that,
 \begin{equation}\label{eq4.37}
  C_1V^{\frac{1}{r'}}(x)-C_1\leq u_{\epsilon}\leq C_2|x|(1+V(x))^{\frac{1}{r'}},\text{ for all } x\in\mathbb R^n.
 \end{equation}
 where $C_1,C_2>0$ are independent of $\epsilon$.
 
 In light of \eqref{eq4.35} and \eqref{eq4.36}, we find for any $R>1$ and $p>1$,
 \begin{equation}\label{eq4.306}
 \|w_{\epsilon}\|_{L^p(B_{2R}(0))}=C_Hr'\|m_{\epsilon}|\nabla u_{\epsilon}|^{r'-1}\|_{L^p(B_{2R}(0))}\leq C_{p,R}<\infty,\end{equation}
 where the constant  $C_{p, R}>0$ depends only on $p$, $R$ and is independent of $\epsilon$.  By applying Lemma \ref{proposition-lemma21-FP}, we obtain from \eqref{eq4.306} that $\|m_\epsilon\|_{W^{1,p}(B_{2R}(0))}\leq C_{p,R}<\infty$. Taking $p>n$ large enough, we employ Sobolev embedding theorem to get
\begin{equation}\label{eq4.38}
\|m_\epsilon\|_{C^{0,\theta_1}(B_{2R}(0))}\leq C_{\theta_1,R}<\infty \text{ for some $\theta_1\in(0,1)$.}
\end{equation}

To estimate $u_{\epsilon},$ we rewrite the $u$-equation of \eqref{eq4.31} as
\begin{equation}\label{eq4.39}
-\Delta u_{\epsilon}=-C_H|\nabla u_{\epsilon}|^{r'}+f_\epsilon(x)\ \text{ with }f_\epsilon(x):=-\lambda_{\epsilon}+V(x)-(\eta_{\epsilon}\ast m_{\epsilon})^{\frac{r}{n}}\ast \eta_{\epsilon},
 \end{equation}
By using \eqref{eq4.35}, \eqref{eq4.36} and the fact that $V$ is locally H\"older continuous, we obtain that for any $p>1$,
 $$\|f_\epsilon\|_{L^{p}(B_{2R}(0))}+\||\nabla u_\epsilon|^{r'}\|_{L^{p}(B_{2R}(0))}\leq C_{p,R}<\infty.$$
Then by $W^{2,p}$ estimates, one gets from \eqref{eq4.39}  and \eqref{eq4.37} that
 \begin{equation}
 \|u_\epsilon\|_{W^{2,p}(B_{R+1})}\leq C_{p,R}\left(\|u_\epsilon\|_{L^{p}(B_{2R}(0))}+\|f_\epsilon\|_{L^{p}(B_{2R}(0))}+\||\nabla u_\epsilon|^{r'}\|_{L^{p}(B_{2R}(0))}\right)
 \leq \tilde C_{p,R}<\infty.
 \end{equation}
Taking $p>n$ large enough, we obtain that
\begin{equation}\label{eq4.41}
\|u_\epsilon\|_{C^{1,\theta_2}(B_{R+1}(0))}\leq C_{\theta_2,R}<\infty, \text{ for some $\theta_2\in(0,1)$.}
\end{equation}
Combining \eqref{eq4.38} with (\ref{eq4.41}), one finds
$$\||\nabla u_{\epsilon}|^{r'}\|_{C^{0,\theta_3}(B_{R+1}(0))} +\|f_\epsilon\|_{C^{0,\theta_3}(B_{R+1}(0))}\leq C_{\theta_3,R}<\infty, \text{ for some $\theta_3\in(0,1)$.}$$
 Then by using Schauder's estimates, we have from \eqref{eq4.39} that
 \begin{equation}\label{eq4.401}
     \|u_{\epsilon}\|_{C^{2,\theta_4}(B_{R}(0))} \leq C_{\theta_4,R}<\infty, \text{ for some $\theta_4\in(0,1)$.}
 \end{equation}
 Now, letting $R\to\infty$ and proceeding the standard diagonalization procedure, we can apply Arzel\`{a}-Ascoli theorem to get that there exists $u_M\in C^2(\mathbb R^n)$ such that
 \begin{equation}\label{eq4.42}
 u_{\epsilon}\overset{\epsilon\to0^+}\longrightarrow u_{M} \text{ in }C^{2,\theta_5}_{\rm loc}(\mathbb R^n) \text{ for some $\theta_5\in(0,1)$.}
 \end{equation}
On the other hand, it follows from Lemma \ref{lemma21-crucial} and \eqref{eq4.32} that,  there exists $(m_{M},w_{M})\in W^{1,\hat q}(\mathbb R^n)\times \big(L^{1}(\mathbb R^n)\cap L^{\hat q}(\mathbb R^n) \big)$ such that
\begin{align}\label{eq4.43}
m_{\epsilon}\overset{\epsilon\to0^+}\to m_{M} \text{ a.e. in $\mathbb R^n$,  \ and \ \ }(m_{\epsilon},w_{\epsilon})\overset{\epsilon\to0^+}\rightharpoonup (m_{M},w_{M}) \text{~in~}  W^{1,\hat q}(\mathbb R^n)\times  L^{\hat q}(\mathbb R^n).
\end{align}
Moreover, with the help of Lemma \ref{lem4.1}, we have
\begin{align}\label{eq4.44}
m_{\epsilon}\overset{\epsilon\to0^+}\rightarrow m_{M} \text{~in~} L^1(\mathbb R^n)\cap L^{1+\frac{r}{n}}(\mathbb R^n) .
\end{align}

Letting $\epsilon\to0^+$ in \eqref{eq4.31}, we then conclude from \eqref{eq4.34} and \eqref{eq4.42}-\eqref{eq4.44} that there exists $\lambda_M\in \mathbb R$ such that
$(m_M, u_M, w_M)$ satisfies  \eqref{125potentialfreesystem}.
In particular, we obtain from \eqref{eq4.36} and \eqref{eq4.37} that
 \begin{equation}\label{eq4.46}
 |\nabla u_{M}(x)|\leq C(1+V(x))^{\frac{1}{r'}}, C_1|x|^{1+\frac{b}{r'}}-C_1\leq u_{M}\leq C_2|x|^{1+\frac{b}{r'}}+C_2, \  \text{ for all } x\in\mathbb R^n.
 \end{equation}
  Recall that $m_{\epsilon}\to m_{M} $ a.e. as $\epsilon\to0^+$ in $\mathbb R^n$, it follows from \eqref{eq4.35}   that $m_M\in L^\infty(\mathbb R^n)$.  Then, proceeding the same argument as shown in the proof of Proposition \ref{prop41relationminimiersolutionsect4-next}, one can further obtain from \eqref{125potentialfreesystem} and \eqref{eq4.46} that
  \begin{equation}\text{$w_{M}=-C_Hr'm_{M}|\nabla u_{M}|^{r'-2}\nabla u_{M}\in L^p(\mathbb R^n)$ and $m_M\in W^{1,p}(\mathbb R^n)$
 for all $p>1$.}\end{equation}

 We finally prove that $(m_{M},w_{M})\in\mathcal K_{M}$ is a minimizer of $e_{\alpha^*,M}$.  To this end, we claim that for $M<M^*$, the following conclusion holds:
 \begin{equation}\label{eq4.406}
 \lim_{\epsilon\to0^+}e_{\epsilon, M}=e_{\alpha^*,M},
 \end{equation}
 where $e_{\alpha^*,M}$ is given in \eqref{minimization-problem-critical}.  Indeed, in view of \eqref{inlightofsect4}, one can easily find that $ \lim\limits_{\epsilon\to0^+}e_{\epsilon, M}\geq e_{\alpha^*,M}$. On the other hand, for any $\delta>0$, we choose $(m,w)\in\mathcal K_{M}$ such that $\mathcal{E}(m,w)\leq e_{\alpha^*,M}+\frac{\delta}{2}$.  Then by using \eqref{inlightofsect4}, we deduce  that for $\epsilon>0$  small enough,  $\mathcal{E}_\epsilon(m,w)\leq \mathcal{E}(m,w)+\frac{\delta}{2} $.  Hence,
 $$e_{\epsilon, M}\leq \mathcal{E}_\epsilon(m,w)\leq \mathcal{E}(m,w)+\frac{\delta}{2} \leq e_{\alpha^*,M}+{\delta}.$$
 Letting $\epsilon\to0^+$ and then $\delta\to 0^+$, one has $ \lim\limits_{\epsilon\to0^+}e_{\epsilon, M}\leq e_{\alpha^*,M}$.  Now, we finish the proof of \eqref{eq4.406}.  
 
  We collect  \eqref{eq4.43}, \eqref{eq4.44}, \eqref{eq4.406} and the convexity of $\int_{\mathbb R^n}\big|\frac{w}{m}\big|^{r}m\,dx $ to get
 $$e_{\alpha^*,M}=\lim_{\epsilon\to0^+}e_{\epsilon, M}=\lim_{\epsilon\to0^+}\mathcal{E}_\epsilon(m_\epsilon,w_\epsilon)\geq \mathcal{E}(m_M,w_M)\geq  e_{\alpha^*,M},$$
which implies $(m_{M},w_{M})\in\mathcal K_{M}$ is a minimizer of $e_{\alpha^*,M}$.  This completes the proof of Conclusion (i) for the case of $r\leq n$.

 \vskip.1truein

  Now, we consider the case of $r>n$. Compared to the former case of $r\leq n$,  since now $W^{1,\hat q}(\mathbb R^n)\hookrightarrow C^{0,\theta}(\mathbb R^n)$ for some $\theta\in(0,1)$, we can avoid discussing the auxiliary problem \eqref{Eepsilonminimizerproblem1} and consider  the minimization problem \eqref{minimization-problem-critical} directly.
  Let $(m_{k},w_{k})\in\mathcal K_{M}$ be a minimizing sequence of (\ref{ealphaM-117}).  Since $M<M^*$, we similarly obtain from (\ref{418cruciallowerbound}) that
  \begin{align*}
\sup_{k\in\mathbb N^+}\int_{\mathbb R^n}m_{k}^{1+\frac{r}{n}}\,dx\leq C<+\infty,~~\sup_{k\in\mathbb N^+}\int_{\mathbb R^n}\bigg(\big|\frac{w_{k}}{m_{k}}\big|^{r}m_{k}+V(x)m_{k}\bigg)\,dx\leq C<+\infty.
\end{align*}
Then we use Lemma \ref{lemma21-crucial} and Lemma \ref{lem4.1} to conclude that there exists $(m_{M},w_{M})\in W^{1,\hat q}(\mathbb R^n)\times \big(L^{1}(\mathbb R^n)\cap L^{\hat q}(\mathbb R^n) \big)$ such that
\begin{align*}
m_{k}\overset{k}\to m_{M} \text{ a.e. in $\mathbb R^n$,  \  \ \ }(m_{k},w_{k})\overset{k}\rightharpoonup (m_{M},w_{M}) \text{~in~}  W^{1,\hat q}(\mathbb R^n)\times  L^{\hat q}(\mathbb R^n),
\end{align*}
and
\begin{align*}
m_{k}\overset{k}\rightarrow m_{M} \text{~in~} L^1(\mathbb R^n)\cap L^{1+\frac{r}{n}}(\mathbb R^n) .
\end{align*}
Thus,
$$e_{\alpha^*,M}=\lim_{k\to\infty}\mathcal{E}(m_k,w_k)\geq \mathcal{E}(m_M,w_M)\geq  e_{\alpha^*,M},$$
which indicates that $(m_{M},w_{M})\in\mathcal K_{M}$ is a minimizer of $e_{\alpha^*,M}$. Moreover, by using Proposition \ref{prop41relationminimiersolutionsect4} and Proposition \ref{prop41relationminimiersolutionsect4-next}, one finds there exists $u_M\in C^2(\mathbb R^n)$ such that $(m_M,u_M,\lambda_M)\in W^{1,p}(\mathbb R^n) \times C^2(\mathbb R^n) \times\mathbb R$ is a solution to (\ref{125potentialfreesystem}).

 \vskip.1truein

It remains to study the critical case $M=M^*$ and show Conclusion (iii).  To this end, we first prove that up to a subsequence,
\begin{align}\label{criticalfirstwantshow}
\lim_{ M\nearrow M^*}e_{\alpha^*,M}=e_{\alpha^*,M^*}=0,
\end{align}
where $\inf_{x\in\mathbb R^n} V(x)=0$ was used, which is stated in (\ref{V2mainassumption_1}).  By the definition of $e_{\alpha^*,M^*}$ given in  \eqref{minimization-problem-critical}, for any $\delta>0$, $\exists (m,w)\in\mathcal A_{M^*}$ such that
\begin{align}\label{combine419-thm12}
e_{\alpha^*,M^*}\leq \mathcal E(m,w)\leq e_{\alpha^*,M^*}+\delta.
\end{align}
Noting that $\frac{M}{M^*}(m,w)\in\mathcal A_{M}$, we have
\begin{align}\label{combine420-thm12}
e_{\alpha^*,M}&\leq \mathcal E\bigg(\frac{M}{M^*}m,\frac{M}{M^*}w\bigg)\\
=&\mathcal E(m,w)\nonumber
+\Big(\frac{M}{M^*}-1\Big)\bigg[C_L\int_{\mathbb R^n}\Big|\frac{w}{m}\Big|^{r}m\,dx+\int_{\mathbb R^n}V(x)m\,dx\bigg]\nonumber
+\frac{n}{n+r}\bigg[1-\bigg(\frac{M}{M^*}\bigg)^{1+\frac{r}{n}}\bigg]\int_{\mathbb R^n}m^{1+\frac{r}{n}}\,dx.
\end{align}
By straightforward computation, one has as $M\nearrow M^*,$ 
\begin{align}\label{combine421-thm12}
\Big(\frac{M}{M^*}-1\Big)\Big[C_L\int_{\mathbb R^n}\bigg|\frac{w}{m}\bigg|^{r}m\,dx+\int_{\mathbb R^n}V(x)m\,dx\Big]+\frac{n}{n+r}\bigg[1-\bigg(\frac{M}{M^*}\bigg)^{1+\frac{r}{n}}\bigg]\int_{\mathbb R^n}m^{1+\frac{r}{n}}\,dx\rightarrow 0.
\end{align}
We collect (\ref{combine419-thm12}), (\ref{combine420-thm12}) and (\ref{combine421-thm12}) to find
\begin{align}\label{424limitbefore}
\limsup_{M\nearrow M^*}e_{\alpha^*,M}\leq e_{\alpha^*,M^*}+\delta, \ \ \forall \delta>0.
\end{align}
Letting $\delta\rightarrow 0,$ we obtain from (\ref{424limitbefore}) that
\begin{align}\label{combine4331critical}
\limsup_{M\nearrow M^*}e_{\alpha^*,M}\leq e_{\alpha^*,M^*}.
\end{align}
On the other hand, let $(\bar m_{\alpha^*,M},\bar w_{\alpha^*,M})\in\mathcal A_M$ be a minimizer of $e_{\alpha^*,M}=\inf_{(m,w)\in\mathcal A_M}\mathcal E(m,w)$ for any fixed $M\in(0,M^*)$.  Then we have $\frac{M^*}{M}({\bar m}_{\alpha^*,M},{\bar w}_{\alpha^*,M})\in\mathcal A_{M^*}$ and
\begin{align*}
e_{\alpha^*,M^*}\leq &\mathcal E\Big(\frac{M^*}{M}(\bar m_{\alpha^*,M},\bar w_{\alpha^*,M})\Big)\\
=&\frac{M^*}{M}\Big[C_L\int_{\mathbb R^n}\Big|\frac{\bar w_{\alpha^*,M}}{\bar m_{\alpha^*,M}}\Big|^{r}{\bar m}_{\alpha^*,M}\,dx+\int_{\mathbb R^n}V(x)\bar m_{\alpha^*,M}\, dx-\Big(\frac{M^*}{M}\Big)^{\frac{r}{n}}\Big(1+\frac{r}{n}\Big)\int_{\mathbb R^n}{\bar m}_{\alpha^*,M}^{1+\frac{r}{n}}\,dx\Big]\\
\leq &\frac{M^*}{M}\mathcal E(\bar m_{\alpha^*,M},\bar w_{\alpha^*,M})=\frac{M^*}{M}e_{\alpha^*,M},\ \ \forall M<M^*.
\end{align*}
  It follows that
\begin{align}\label{combine4332critical}
e_{\alpha^*,M^*}\leq \liminf_{M\nearrow M^*}\frac{M^*}{M}e_{\alpha^*,M}=\lim_{M\nearrow M^*} e_{\alpha^*,M}.
\end{align}
Combining (\ref{combine4331critical}) with (\ref{combine4332critical}), one has
\begin{align}\label{togetherwithenergycriticalvalue}
\lim_{M\nearrow M^*} e_{\alpha^*,M}=e_{\alpha^*,M^*}\geq 0.
\end{align}
By using assumptions (V1) and (V2) shown in Subsection \ref{mainresults11} for potential $V$, we set $M=M^*$ in (\ref{supercriticalmasscase}) to get
\begin{align*}
e_{\alpha^*,M^*}\leq \mathcal E(m_*^t,w_*^t)=M^*V(x_0)+o_t(1)\rightarrow 0,\text{~if~}V(x_0)=0\text{~and~}t\rightarrow +\infty.
\end{align*}
Hence $e_{\alpha^*,M^*}\leq 0$, which together with (\ref{togetherwithenergycriticalvalue}) implies (\ref{criticalfirstwantshow}).

Now, we argue by contradiction to show Conclusion (iii).
 Assume that $e_{\alpha^*,M^*}$ has a minimizer $(\hat m,\hat w)\in\mathcal A_{M^*}$, then one applies (\ref{criticalfirstwantshow}) to obtain
\begin{align*}
0=\mathcal E(\hat m,\hat w)=\int_{\mathbb R^n}C_L\Big|\frac{\hat w}{\hat m}\Big|^{r}\hat m\,dx -\frac{1}{1+\frac{r}{n}}\int_{\mathbb R^n}\hat m^{1+\frac{r}{n}}\, dx+\int_{\mathbb R^n}V(x)\hat m\,dx\geq 0,
\end{align*}
which together with (\ref{optimalinequality415})  implies
\begin{align}\label{eq5.85}
C_L\int_{\mathbb R^n}\Big|\frac{\hat w}{\hat m}\Big|^{r}\hat m\,dx=\frac{1}{1+\frac{r}{n}}\int_{\mathbb R^n}\hat m^{1+\frac{r}{n}}\,dx\text{~and~}\int_{\mathbb R^n}V(x)\hat m\,dx =0.
\end{align}
As a consequence, we obtain $\text{supp}V(x)\cap \text{supp~}\hat m=\emptyset.$  However, as exhibited in (\ref{V2mainassumption_3}), if $r\leq n,$ one has $\text{supp}V=\mathbb R^n,$ which indicates $\hat m=0$ a.e..  Otherwise if $r>n,$ we have facts that $(\hat m,\hat w)$ is a minimizer of $e_{\alpha^*,M^*}$ and $\hat m\in W^{1,r}(\mathbb R^n)\hookrightarrow C^{0,\theta}(\mathbb R^n)$ for some $\theta\in (0,1)$ by Lemma \ref{lemma21-crucial}.  Moreover, in light of Lemma \ref{lemma22preliminary}, one finds there exists a bounded from below solution $\hat u\in C^2(\mathbb R^n)$ and $\lambda\in \mathbb R$ such that 
\begin{align*}
\left\{\begin{array}{ll}
-\Delta \hat u+C_H|\nabla \hat u|^{r'}+\lambda=V(x)-\hat m^{\frac{r}{n}}, &x\in\mathbb R^n,\\
-\Delta \hat m-C_Hr'\nabla\cdot(m|\nabla \hat u|^{r'-2}\nabla\hat u)=0,&x\in\mathbb R^n,\\
\hat m\geq 0\text{~in~}\mathbb R^n,~~\int_{\mathbb R^n}\hat m\,dx=M^*>0.
\end{array}
\right.
\end{align*}
By standard elliptic estimates, one gets from the H\"{o}lder continuity of $V$ that $\hat u\in C^{2,\theta_1}_{\text{loc}}(\mathbb R^n)$ for some $\theta_1\in(0,1).$  Then, we apply the maximum principle as shown in \cite{bernardini2023ergodic} to find $\hat m>0$ in $\mathbb R^n$ when $r>n.$ This indicates $\int_{\mathbb R^n}V(x)\hat m\,dx >0$, which however contradicts \eqref{eq5.85}.   This completes the proof of Conclusion (iii).
\end{proof}

Theorem \ref{thm11} demonstrates that under some technical assumptions on potential $V$, there exists the critical mass phenomenon to problem (\ref{minimization-problem-critical}), where $\alpha=\frac{r}{n}.$  In detail, there exists $M^*>0$ such that when $M<M^*$, there exists at least one minimizer; otherwise if $M\geq M^*,$ problem (\ref{minimization-problem-critical}) has no minimzer.  A natural question is the asymptotic behaviors of minimizers to (\ref{minimization-problem-critical}) as $M\nearrow M^*$ and we investigate this in   Section \ref{sect5preciseblowup}.


\section{Blow-up Behaviors of Ground States as $M\nearrow M^*$ }\label{sect5preciseblowup}
In this section, we discuss the asymptotic profiles of least energy solutions to (\ref{MFG-SS}) as $M\nearrow M^*.$  More precisely, we shall prove Theorem \ref{thm13basicbehavior} and Theorem \ref{thm14preciseblowup}.  We would like emphasize that the results are parallel to Theorem 1.2 in \cite{cesaroni2018concentration} but they only consider the subcritical mass exponent case.
 \subsection{Basic Asymptotic Profiles of Ground States}

This subsection is devoted to the proof of Theorem \ref{thm13basicbehavior} and the investigation of the rough asymptotics.

\medskip
\textbf{Proof of Theorem \ref{thm13basicbehavior}:} 
\begin{proof}
To show Conclusion (i), we argue by contradiction and assume that
$$\limsup_{M\nearrow M^*}\int_{\mathbb R^n}\bigg|\frac{w_M}{m_M}\bigg|^{r}m_M\,dx<+\infty.$$
By using Lemma \ref{lemma21-crucial}, one further obtains
\begin{align}\label{inlightofsect5blowupbasicpro}
\limsup_{M\nearrow M^*}\Vert m_M\Vert_{W^{1,\hat q}(\mathbb R^n)}, ~\limsup_{M\nearrow M^*}\Vert w_M\Vert_{L^{\hat q}(\mathbb R^n)},~\limsup_{M\nearrow M^*}\Vert w_M\Vert_{L^1(\mathbb R^n)}<+\infty
\end{align}
 It follows that there exists $(m,w)\in  W^{1,\hat q}(\mathbb R^n)\times  L^{\hat q}(\mathbb R^n)  $ such that as $M\nearrow M^*$,
\begin{align}\label{combine1sect5}
m_M\rightharpoonup m \text{~in~} W^{1,\hat q}(\mathbb R^n),~\text{ and }~w_M\rightharpoonup w\text{~in~}  L^{\hat q}(\mathbb R^n).
\end{align}

We next prove that $(m,w)\in \mathcal K_{M^*}$ defined by (\ref{constraint-set-K}).  Firstly, in light of (\ref{inlightofsect5blowupbasicpro}), we have the fact that 
\begin{align}\label{combinesect5new53}
\limsup_{M\nearrow M^*}\int_{\mathbb R^n}V(x)m_M\,dx<+\infty.
\end{align}
  Noting that potential $V$ satisfies (V1), (V2) and (V3) given in Subsection \ref{mainresults11}, we combine (\ref{combine1sect5}) with (\ref{combinesect5new53}) to obtain from Lemma \ref{lem4.1} that 
\begin{align}\label{eq5.4}
m_M\rightarrow  m \text{~in~} L^1(\mathbb R^n)\cap L^{1+\frac{r}{n}}(\mathbb R^n), ~\text{as $M\nearrow M^*$},
\end{align}
which indicates $\int_{\mathbb R^n}m\,dx=M^*.$  In addition, we apply (\ref{combine1sect5}) to get $\Delta m= \nabla\cdot w$ {weakly.} 
 Then, similarly as in \eqref{eq2.31}, we have  $w\in L^1(\mathbb R^n)$.  By summarizing the argument above, we find $(m,w)\in \mathcal K_{M^*}$ and  
$\liminf\limits_{M\nearrow M^*} \mathcal E(m_M,w_M)\geq\mathcal E(m,w)$ thanks to \eqref{combine1sect5} and \eqref{eq5.4}.  Furthermore, we obtain from (\ref{criticalfirstwantshow}) that
$$e_{\alpha^*,M^*}\geq \mathcal E(m,w)\geq e_{\alpha^*,M^*}.$$
Thus, $(m,w)$ is a minimizer of $e_{\alpha^*,M^*},$ which is a contradiction to Conclusion (iii) in Theorem \ref{thm11}.  This finishes the proof of Conclusion (i).

(ii).  Recall that $\varepsilon_M$, which will be denoted as $\varepsilon$ for simplicity,  is defined as \eqref{thm51property1}. As stated in Conclusion (i) of Theorem \ref{thm11}, we have the facts that each $u_M\in C^2(\mathbb R^n)$ is bounded from below and satisfies $\lim\limits_{|x|\rightarrow+\infty} u_M(x)=+\infty$.  Hence, there exists $x_{\varepsilon}\in \mathbb R^n$ such that $u_M(x_{\varepsilon})=\inf\limits_{x\in\mathbb R^n}u_{M}(x)$, which implies $0=u_{\varepsilon}(0)=\inf\limits_{x\in\mathbb R^n}u_{\varepsilon}(x)$ thanks to the definition given in (\ref{thm51property2}).  

From \eqref{125potentialfreesystem} and (\ref{thm51property2})  we see that $(u_\varepsilon, m_\varepsilon, w_\varepsilon)$ satisfies 
 \begin{align}\label{eqscalingaftersect5mfgnew}
 \left\{\begin{array}{ll}
 -\Delta u_{\varepsilon}+C_H|\nabla u_{\varepsilon}|^{r'}+\lambda_M \varepsilon^{r}=-m_{\varepsilon}^{\frac{r}{n}}+\varepsilon^{r}V(\varepsilon x+x_{\varepsilon}),\\
 -\Delta m_{\varepsilon}-C_Hr'\nabla\cdot(m_{\varepsilon}|\nabla u_{\varepsilon}|^{r'-2}\nabla u_{\varepsilon})=-\Delta m_{\varepsilon}+\nabla \cdot w_{\varepsilon}=0,\\
 \int_{\mathbb R^n}m_{\varepsilon}\,dx=M.
 \end{array}
 \right.
 \end{align}  
By (\ref{thm51property1}), \eqref{optimalinequality415} and \eqref{criticalfirstwantshow}, we have 
\begin{align}\label{54equalitynormalize}
C_L\int_{\mathbb R^n}\bigg|\frac{w_{\varepsilon}}{m_{\varepsilon}}\bigg|^{r}m_{\varepsilon}\,dx=\varepsilon^{-r}C_L\int_{\mathbb R^n}\bigg|\frac{w_{M}}{m_{M}}\bigg|^{r}m_{M}\,dx\equiv 1,\ 
\int_{\mathbb R^n}m_{\varepsilon}^{1+\frac{r}{n}}\,dx=\varepsilon^{-r}\int_{\mathbb R^n}m_M^{1+\frac{r}{n}}\,dx\rightarrow \frac{r+n}{n},
\end{align}
and 
\begin{align}\label{59limitafterV}
\int_{\mathbb R^n} V(\varepsilon x+x_{\varepsilon})m_{\varepsilon}\,dx=\int_{\mathbb R^n}V(x)m_M\,dx\rightarrow 0\text{~as~}M\nearrow M^*.
\end{align}
Similar to \eqref{eq4.301}, we deduce from \eqref{eqscalingaftersect5mfgnew} and \eqref{54equalitynormalize}
that 
\begin{align*}
M\lambda_M=\mathcal E(m_M,w_M)-\frac{r}{n+r}\int_{\mathbb R^n}m_M^{1+\frac{r}{n}}\,dx=o(1)-\frac{r}{n+r}\varepsilon^r\int_{\mathbb R^n}m_{\varepsilon}^{1+\frac{r}{n}}\,dx,
\end{align*}
which implies 
\begin{align}\label{eq5.8}
\lambda_M\varepsilon^{r}\rightarrow -\frac{r}{M^*n}\text{~as~}M\nearrow M^*.
\end{align}

In light of the $u$-equation in (\ref{eqscalingaftersect5mfgnew}), we apply the maximum principle to obtain
\begin{align}\label{maximumprincipleconclusionsect5}
\lambda_M\varepsilon^{r}\geq -m^{\frac{r}{n}}_{\varepsilon}(0)+\varepsilon^rV(x_{\varepsilon})\geq -m^{\frac{r}{n}}_{\varepsilon}(0),
\end{align}
which implies 
\begin{align}\label{eq510mlowerbounduniform}
m^{\frac{r}{n}}_{\varepsilon}(0)\geq -\lambda_M\varepsilon^{r}\geq C>0,
\end{align}
 where and in the arguments below,  $C>0$ denotes a  constant independent of $\varepsilon,$ which may change line to line. 

{
We next claim that there exists some constant $C>0$ such that
\begin{align}\label{eq6.110}
\varepsilon^rV(x_{\varepsilon} )\leq C.
\end{align}
If not, we can find some subsequence $\varepsilon_l\rightarrow 0$ such that $\varepsilon_l^rV(x_{\varepsilon_l})\rightarrow +\infty$.  By using (\ref{maximumprincipleconclusionsect5}), one has
\begin{align*}
\frac{m_{\varepsilon_l}^{\alpha^*}(0)}{\varepsilon^r_l V(x_{\varepsilon_l})}\geq C,
\end{align*}
where $C>0$ is some constant independent of $\varepsilon_l.$ Similarly, let
\begin{align*}
v_{l}(x)=a_{l}^{r-2}u_{l}(a_{l}x),~~\mu_{l}(x)=a_{l}^{n}m_{l}(a_{l}x),~~a_{l}=\frac{1}{\varepsilon_l^{r}V(x_{\varepsilon_l})},
\end{align*}
then we have
\begin{align*}
a_{l}^r=\frac{1}{\varepsilon_l^rV(x_{\varepsilon_l})}\rightarrow 0,~~a_{l}^r\varepsilon_l^rV(x_{\varepsilon_l})=1.
\end{align*}
In light of (\ref{V2mainassumption_2}), one finds 
\begin{align*}
a_{l}^r\varepsilon^r_l V(a_{l}\varepsilon_l x+x_{\varepsilon_l})=\frac{V(a_{\varepsilon_l}\varepsilon_l x+x_{\varepsilon_l})}{V(x_{\varepsilon_l})}\leq C_2,
\end{align*}
where $C_2>0$ is some constant independent of $l.$  Noting that
\begin{align*}
\Vert \mu_l^{\alpha^*}\Vert^{1+\frac{1}{\alpha^*}}_{L^{1+\frac{1}{\alpha^*}}(B_1(0))}=a_l^r\Vert m_{\varepsilon_l}\Vert_{L^{1+\alpha^*}(B_{a_l}(x_l))}\rightarrow 0\text{~as~}l\rightarrow +\infty,
\end{align*}
we apply the maximal regularity shown in Theorem \ref{thmmaximalregularity} to obtain
\begin{align*}
\Vert |\nabla v_{l}|^{r'}\Vert_{L^{1+\frac{1}{\alpha^*}}(B_{1/2})}\leq C,
\end{align*}
where $C>0$ is some constant, which implies $\mu_{l}\in C^{0,\theta}(B_{1/4}(0))$ with $\theta\in(0,1).$
\\
since $\mu_{l}(0)=a_{l}^nm_{l}(0)$, one finds
\begin{align*}
\mu_{l}^{\alpha^*}(0)=\frac{m^{\alpha^*}_{l}(0)}{\varepsilon_l^{r}V(x_{\varepsilon_l})}\geq C>\delta>0.
\end{align*}
By using the H\"{o}lder's continuity of $\mu_l$, we obtain
\begin{align*}
\mu_{l}(x)\geq \delta>0\text{~in~}B_R(0),
\end{align*}
where $R\in(0,\frac14)$ is some constant independent of $l.$
We have
\begin{align*}
&\int_{\mathbb R^n} V(\varepsilon_l x+x_{\varepsilon_l})m_{\varepsilon_l} (x)\,dx\nonumber\\
=&\int_{\mathbb R^n}V(\varepsilon_l a_{\varepsilon_l }x+x_{\varepsilon_l})\mu_{l}\,dx\geq \delta\int_{B_{R}(0)}V(\varepsilon_l a_{l} x+x_{\varepsilon_l})\,dx\geq \frac{\delta}{\varepsilon_l^r}\rightarrow +\infty.\nonumber
\end{align*}
However, as $\varepsilon\rightarrow 0,$
\begin{align*}
\int_{\mathbb R^n} V(\varepsilon x+x_{\varepsilon})m_{\varepsilon} (x)\,dx\rightarrow 0,
\end{align*}
which is a contradiction.  This completes the proof of our claim \eqref{eq6.110}.
}




	On the other hand, noting that $V$ satisfies (\ref{V2mainassumption_2}), one further finds for $R>0$ large enough,
			\begin{align}\label{eq512insect5new}
				\varepsilon^rV(\varepsilon x+x_{\varepsilon})\leq C_R<+\infty,\text{ for all }|x|\leq 4R,
			\end{align}
			where positive constant $C_R$ depends on $R$ but is independent of $\varepsilon.$ 
			
			Similarly as shown in the proof of Theorem \ref{thm11}, we  estimate  $\nabla u_{\varepsilon}$ and rewrite the $u$-equation of \eqref{eqscalingaftersect5mfgnew} as
			\begin{equation}\label{negativedeltaueq6521}
				-\Delta u_{\varepsilon}=-C_H|\nabla u_{\varepsilon}|^{r'}+g_\varepsilon(x)\ \text{ with }g_\varepsilon(x):=-\lambda_{M}\varepsilon^r +\varepsilon^r  V(x_{\varepsilon}+\varepsilon x)-m^{\frac{r}{n}}_{\varepsilon}.
			\end{equation}
			Noting that $m^{\alpha^*}\in L^{1+\frac{1}{\alpha^*}}(\mathbb R^n)$, we have from the maximal regularity shown in Theorem \ref{thmmaximalregularity} that $|\nabla u|^{r'}\in L_{\rm loc}^{1+\frac{1}{\alpha^*}}(\mathbb R^n)$, i.e. $|\nabla u|^{r'-1}\in L_{\rm loc}^{\big(1+\frac{1}{\alpha^*}\big)r'}(\mathbb R^n)$, which implies $m\in C^{0,\theta}_{\text{loc}}(\mathbb R^n)$ for some $\theta\in(0,1)$ by using the $m$-equation in (\ref{eqscalingaftersect5mfgnew}). Thanks to Lemma \ref{sect2-lemma21-gradientu}, we further obtain 
			\begin{align*}
				|\nabla u_{\varepsilon}|\leq C_R<+\infty,~~\forall |x|<2R,
			\end{align*}
			which implies
			\begin{align*}
			\Vert |\nabla u_{\varepsilon}|^{r'}\Vert_{L^p(B_{2R}(0))}+\Vert g_{\varepsilon}\Vert_{L^p(B_{2R}(0)}\leq C_{p,R}<+\infty.
			\end{align*}
   Then, similarly as the derivation of \eqref{eq4.38}  and \eqref{eq4.401},  we can obtain that for some $\theta\in(0,1)$,
 \begin{align}\label{eq514upperboundmcalpha}
    \Vert m_{\varepsilon}\Vert_{C^{0,\theta}(B_{2R}(0))}\leq C<\infty,~~
 \end{align}
and
 \begin{align}\label{eq515C2thetaquestion}
\Vert u_{\varepsilon}\Vert_{C^{2,\theta}(B_R(0))}\leq C<\infty.
 \end{align}
  In light of (\ref{eq510mlowerbounduniform}), one has from (\ref{eq514upperboundmcalpha}) that there exists a constant $R_0\in(0,1)$ such that 
 \begin{align}\label{eq516lowerboundmindependent}
 m_{\varepsilon}(x)\geq C>0,~~\forall |x|<R_0.
 \end{align}

Now, we claim that up to a subsequence,
\begin{equation}\label{eq5.18}
    \lim_{\varepsilon\to0}x_{\varepsilon}= x_0 \ \text{ with }~V(x_0)=0.
\end{equation}
If not, one has either $|x_{\varepsilon}|\rightarrow +\infty$ or $x_{\varepsilon}\rightarrow x_0$ with $V(x_0)>0.$  In the two cases, we both get 
$\lim\limits_{x_\varepsilon\to0}V(\varepsilon x+x_\varepsilon)\geq A$ a.e. in $\mathbb R^n$ for some $A>0$.
   Hence, it follows from \eqref{eq516lowerboundmindependent} that 
\begin{align*}
\lim_{\varepsilon\to 0}\int_{\mathbb R^n}V(\varepsilon x+x_{\varepsilon})m_{\varepsilon}\,dx\geq \frac{A}{2}\int_{B_{R_0}(0)}m_{\varepsilon} (x)\, dx\geq \frac{AC}{2}|B_{R_0}(0)|,
\end{align*}
which reaches a contradiction to (\ref{59limitafterV}).  Therefore \eqref{eq5.18} is proved.

In light of (\ref{54equalitynormalize}), there exists $(m_0,w_0) \in W^{1,\hat q}(\mathbb R^n)\times \big(L^1(\mathbb R^n)\cap L^{\hat q}(\mathbb R^n)\big)$ such that 
\begin{align}\label{eq5.20519}
(m_{\varepsilon},w_{\varepsilon})\rightharpoonup (m_0,w_0) \ \ \text{in~} W^{1,\hat q}(\mathbb R^n)\times  L^{\hat q}(\mathbb R^n), \ \text{ as }\varepsilon\to 0,
\end{align}
where  $m_0\not\equiv 0$ by \eqref{eq516lowerboundmindependent}  and  $\hat q$ is given by (\ref{hatqconstraint}).  Furthermore, thanks to (\ref{eq515C2thetaquestion}), we find $u_{\varepsilon}\rightarrow u_0$ in $C^2_{\rm loc}(\mathbb R^n)$.  Moreover,  it follows from \eqref{eqscalingaftersect5mfgnew} and \eqref{eq5.8} that  $(m_0,u_0)$ satisfies
 \begin{align} \label{eq5.20}
 \left\{\begin{array}{ll}
 -\Delta u_0+C_H|\nabla u_0|^{r'}-\frac{r}{M^*n}=-m_{0}^{\frac{r}{n}},\\
 -\Delta m_0=-C_Hr'\nabla\cdot(m_0|\nabla u_0|^{r'-2}\nabla u_0)=-\nabla w_0,\\
 0<\int_{\mathbb R^n}m _0\,dx\leq M^*,~~w_0=-C_Hm_0|\nabla u_0|^{r'-2}\nabla u_0,
 \end{array}
 \right.
 \end{align}
 where similarly as in Section \ref{sect3-optimal} (see the proof of \eqref{combining1strongconverge2}), we have used Pohozaev identities given in Lemma \ref{poholemma} to obtain that $(m_0, w_0)$ is a minimizer of \eqref{sect2-equivalence-scaling}  and $\int_{\mathbb R^n}m _0\,dx=M^*$.  We thus see that $(u_0,m_0,w_0)$ satisfies \eqref{limitingproblemminimizercritical0} by recalling \eqref{eq5.20}.
 In addition, we find $m_\varepsilon\to m_0$ in $L^1(\mathbb R^n)$, and then by \eqref{eq5.20519}, we  obtain that  
 \begin{align*}
m_{\varepsilon}\to m_0  \ \ \text{in~} L^{p}(\mathbb R^n),~ \forall ~ p\in[1,{\hat q}^*)  \ \text{ as }\varepsilon\to 0,
\end{align*}
which inidicates \eqref{eq1.40}. 

It is left to prove \eqref{eq141realtionuminmmax} when $V$ satisfies (\ref{cirant-V}).  To this end, we argue by contradiction and assume that, up to a subsequence, 
 \begin{align}\label{eq5.23}
 \liminf_{\varepsilon\rightarrow 0}\frac{|\bar x_{\varepsilon}-x_{\varepsilon}|}{\varepsilon}=+\infty.
 \end{align}
Define 
\begin{align}\label{eq5.24}
\left\{\begin{array}{ll}
\bar m_{\varepsilon}(x):=\varepsilon^{n}m_M(\varepsilon x+\bar x_{\varepsilon})=m_{\varepsilon}\Big(x+\frac{\bar x_{\varepsilon}-x_{\varepsilon}}{\varepsilon}\Big),\\
\bar w_{\varepsilon}(x):=\varepsilon^{n+1}w_M(\varepsilon x+\bar x_{\varepsilon})=w_{\varepsilon}\Big(x+\frac{\bar x_{\varepsilon}-x_{\varepsilon}}{\varepsilon}\Big),\\
\bar u_{\varepsilon}(x):=\varepsilon^{\frac{2-r'}{r'-1}} u_M(\varepsilon x+\bar x_{\varepsilon})=u_{\varepsilon}\Big(x+\frac{\bar x_{\varepsilon}-x_{\varepsilon}}{\varepsilon}\Big).
\end{array}
\right.
\end{align}
We claim that there exist constants $R_0>0$ and $C>0$ independent of $\varepsilon$ such that 
\begin{align}\label{claimuniformlowerboundmcriticalsect5}
\bar m_{\varepsilon}(x)\geq C>0,\ \ \forall~ |x|<R_0.
\end{align}
In light of \eqref{eq5.24}, one finds (\ref{claimuniformlowerboundmcriticalsect5}) is equivalent to 
\begin{align}\label{eq5.25lowerboundsect5}
m_{\varepsilon}(x)\geq C>0, \ \ \forall~\Big|x-\frac{\bar x_{\varepsilon}-x_{\varepsilon}}{\varepsilon}\Big|<R_0.
\end{align}
Noting \eqref{eq510mlowerbounduniform}, we have the fact that 
\begin{equation}\label{eq5.27}
    \bar m_\varepsilon(0)=\|\bar m_\varepsilon\|_{L^\infty(\mathbb R^n)}=\| m_\varepsilon\|_{L^\infty(\mathbb R^n)}>C>0.
\end{equation}
To show (\ref{claimuniformlowerboundmcriticalsect5}) or \eqref{eq5.25lowerboundsect5}, we divide our discussions into two cases, $r\leq n$ and $r>n.$  Focusing on the case $r\leq n$, we impose (\ref{cirant-V}) on the potential $V.$  Noting the first equation in \eqref{eqscalingaftersect5mfgnew}, one finds $\bar u_{\varepsilon}$ satisfies
\begin{align}
-\Delta \bar u_{\varepsilon}+C_H|\nabla \bar u_{\varepsilon}|^{r'}=\bar g_{\varepsilon}(x):=-\lambda_M\varepsilon^r-\bar m_{\varepsilon}^{\frac{r}{n}}+\varepsilon^rV(\varepsilon x+\bar x_{\varepsilon}).
\end{align}  Similarly as the argument shown in  \cite[Theorem 4.1]{cesaroni2018concentration}, we consider the following two cases:  

\textbf{Case 1:} Assume that there exists some constant $C>0$ independent of $\varepsilon$ such that $\bar x_{\varepsilon}$ satisfies
 \begin{align*}
 \limsup_{\varepsilon\rightarrow 0}\varepsilon^{r}V(\bar x_{\varepsilon})\leq C<+\infty.
 \end{align*}
 Then by using \eqref{eq5.27}, one can proceed the same argument shown as \eqref{negativedeltaueq6521}-\eqref{eq514upperboundmcalpha} and \eqref{eq516lowerboundmindependent} to obtain the claim (\ref{claimuniformlowerboundmcriticalsect5}).
 
\textbf{Case 2:}  Assume that $\bar x_{\varepsilon}$ satisfies
\begin{align}\label{eq5.29}
\liminf_{\varepsilon\rightarrow 0} \varepsilon^rV(\bar x_{\varepsilon})=+\infty.
\end{align}
We define
\begin{align}\label{eq5.30}
\tilde m(x)=\varepsilon^n m_M(\varepsilon x)=m_{\varepsilon}\bigg(x-\frac{x_{\varepsilon}}{\varepsilon}\bigg),\ 
\tilde u(x)=\varepsilon^{\frac{2-r'}{r'-1}}u_M(\varepsilon x),
\tilde w(x)=\varepsilon^{n+1}w_M(\varepsilon x),
\end{align}
then find from \eqref{eqscalingaftersect5mfgnew} that $(\tilde m,\tilde u,\tilde w)$ satisfies 
\begin{align}
\left\{\begin{array}{ll}
-\Delta \tilde u+C_H|\nabla \tilde u|^{r'}+\lambda_M\varepsilon^{r}=\varepsilon^{r}V(\varepsilon x)-\tilde m_{\varepsilon}^{\frac{r}{n}},&x\in\mathbb R^n,\\
-\Delta \tilde m-C_Hr'\nabla\cdot(\tilde m|\nabla \tilde u|^{r'-2}\nabla \tilde u)=0,&x\in\mathbb R^n,\\
\int_{\mathbb R^n}\tilde m_{\varepsilon}\,dx =M\nearrow M^*, \ \ \tilde w_{\varepsilon}=-C_Hr'\tilde m|\nabla \tilde u|^{r'-2}\nabla \tilde u.
\end{array}
\right.
\end{align}
Noting that $V$ satisfies (\ref{cirant-V}), we employ Lemma \ref{sect2-lemma21-gradientu} to get 
\begin{align}\label{eq5.31sect5gradtildeu}
|\nabla \tilde u_\varepsilon|\leq C\big(1+\sigma_\varepsilon^{\frac{1}{r'}}|x|^{\frac{b}{r'}}\big), \ \ \sigma_\varepsilon:=\varepsilon^{r+b}.
\end{align}
 Denote $y_{\varepsilon}:=\frac{x_{\varepsilon}}{\varepsilon}$ and $\bar y_{\varepsilon}:=\frac{\bar x_{\varepsilon}}{\varepsilon}$, which are  the minimum and maximum points of $\tilde u_\varepsilon(x)$ and $\tilde m_\varepsilon(x),$ respectively. Thanks to \eqref{eq5.18}, we obtain 
$|y_{\varepsilon}|\leq C\varepsilon^{-1}.$  Then, it follows from  (\ref{eq5.31sect5gradtildeu}) that
\begin{align}\label{eq5.33}
|\tilde u_{\varepsilon}(0)|\leq |\tilde u_{\varepsilon}(y_{\varepsilon})|+| y_{\varepsilon}|\sup_{|y|\leq |y_{\varepsilon}|}|\nabla \tilde u_{\varepsilon}(y)|
\leq 1+C\varepsilon^{-1}+C\varepsilon^{-1}\sigma_\varepsilon^{\frac{1}{r'}}|y_{\varepsilon}|^{\frac{b}{r'}}\leq 1+C\varepsilon^{-1}.
\end{align}
As a consequence,
\begin{align}\label{eq534uniformupperboundsect5}
\tilde u_{\varepsilon}(x)\leq \tilde u_{\varepsilon}(0)+|x|\sup|\nabla u_{\varepsilon}(x)|
\leq 1+C\varepsilon^{-1}+\sigma_\varepsilon^{\frac{1}{r'}}|x|^{\frac{b}{r'}+1}.
\end{align}
 Invoking \eqref{eq5.29}, \eqref{eq5.33} and (\ref{eq534uniformupperboundsect5}), we proceed the same argument shown in \cite[Theorem 4.1]{cesaroni2018concentration} to get $\tilde m_{\varepsilon}\in C^{0,\theta}(B_R(\bar y_\varepsilon))$ with $\theta\in(0,1)$ and $R>0$ independent of $\varepsilon.$  Noting that $\bar y_\varepsilon$ is maximum point of $\tilde m_\varepsilon(x)$, we thus obtain from \eqref{eq5.27} and \eqref{eq5.30} that $\tilde m_{\varepsilon}(\bar y_{\varepsilon})\geq C>0.$
Hence, one finds there exists some $R_0>0$ independent of $\varepsilon$ such that
\begin{align*}
\tilde m_{\varepsilon}(x)>\frac{C}{2}>0,\forall~|x-\bar y_{\varepsilon}|<R_0.
\end{align*}
Since $\bar y_\varepsilon=\frac{\bar x_\varepsilon}{\varepsilon}$, the above estimate together with \eqref{eq5.30}  indicates that (\ref{eq5.25lowerboundsect5}) holds.

When $r>n,$ with the assumptions (\ref{V2mainasumotiononv}) imposed on potential $V$, using \eqref{54equalitynormalize} and noting that $W^{1,\hat q}(\mathbb R^n)\hookrightarrow C^{0,\theta}(\mathbb R^n)$ for some $\theta\in(0,1)$, we deduce from Lemma \ref{lemma21-crucial} that $m_{\varepsilon}$ is uniformly bounded in $C^{0,\theta}(\mathbb R^n)$.  Then claim (\ref{claimuniformlowerboundmcriticalsect5}) follows drectly by noting \eqref{eq5.27}. 

In summary, if $V$ satisfies (\ref{cirant-V}) when $r\leq n$ or (\ref{V2mainasumotiononv}) when $r>n$, we obtain (\ref{claimuniformlowerboundmcriticalsect5})  and \eqref{eq5.25lowerboundsect5} hold.  However, \eqref{eq5.25lowerboundsect5} together with \eqref{eq5.23} contradicts the fact that  $m_{\varepsilon}$ converges strongly  to $m_0$ in $L^1(\mathbb R^n)$.  Therefore, \eqref{eq141realtionuminmmax} is proved and we  complete the proof of Theorem \ref{thm13basicbehavior}. 

\end{proof}
Theorem \ref{thm13basicbehavior} indicates that as $M\nearrow M^*,$ the ground states $(m_M,w_M)$ to problem (\ref{minimization-problem-critical}) form some profile with the striking structure, where the leading stone is given by $(m_0,w_0),$ the minimizer to problem (\ref{optimal-inequality-sub}).  In fact, with the precise local expansion of potential $V$, we are able to capture the asymptotic of scaling factor and the location of concentrated profile, which will be shown in Subsection \ref{refinedblow}.


\subsection{Refined Asymptotic Profiles of Ground States}\label{refinedblow}
In this subsection, we shall discuss the refined blow-up behaviors of the rescaled minimizer $(m_{\varepsilon},w_{\varepsilon}).$  As shown in Section \ref{sect5preciseblowup}, $(m_{\varepsilon},w_{\varepsilon},u_{\varepsilon})$ converges to $(m_0,w_0,u_0)$ with $$\text{$m_{\varepsilon}\rightarrow m_0$ in $L^p(\mathbb R^n)\ \forall ~p\in[1,{\hat q}^*) $, $w_{\varepsilon}\rightharpoonup w_0$ $L^{\hat q}(\mathbb R^n)$ and $u_{\varepsilon}\rightarrow u_0$ in $C^2_{\text{loc}}(\mathbb R^n)$}.$$
where $(m_0,w_0)$ is a minimizer of $\Gamma_{\alpha^*}$ and $(u_0, m_0,w_0)$ satisfies \eqref{limitingproblemminimizercritical0}. In addition, we deduce from Lemmas \ref{mdecaylemma} and \ref{poholemma} that 
there exist  $\delta_1>0$ and $C_{\delta_1}>0$ such that,
\begin{equation}\label{eq6.1}
 m_0(x)\leq C_{\delta_1}C^{-\delta_1|x|},\end{equation}
and
\begin{align}\label{6dian2}
1=C_L\int_{\mathbb R^n}\Big|\frac{w_0}{m_0}\Big|^{r}m_0\,dx=\frac{n}{n+r}\int_{\mathbb R^n}m_0^{1+\frac{r}{n}}\,dx.
\end{align}
Before proving Theorem \ref{thm14preciseblowup}, we recall the following assumptions on potential $V$:

Suppose $V(x)$ has $l\in \mathbb N$ distinct zeros defined by $\{P_1,\cdots, P_l\}\subset \mathbb{R}^n$; moreover, $\exists a_i>0$, $q_i>0$, $d>0$ such that
\begin{align}\label{taylorexpansionV}
V(x)=a_i|x-P_i|^{q_i}+O(|x-P_i|^{q_i+1}), \ \ \text{~if~}|x-P_i|\leq d.
\end{align}
Define $q=\max\{q_1,\cdots,q_l\}$ and $Z=\{P_i~|~q_i=q, i=1,\cdots,l\}$, then we denote
\begin{equation}\label{eq6.3}
\mu=\min\{\mu_i~|~P_i\in Z, i=1,\cdots,l\} ~\text{ with }~\mu_i=\min\limits_{y\in\mathbb R^n}H_i(y),\ H_i(y)=\int_{\mathbb R^n} a_i|x+y|^{q_i}m_0(x)\,dx,
\end{equation}
and define set $Z_0=\{P_i~|~P_i\in Z\text{~and~} \mu_i=\mu, i=1,\cdots,l\},$ which consists of all weighted flattest zeros of $V(x).$  With the notations, we establish the precise upper bound of $e_{\alpha^*,M}$  as $M\nearrow M^*$, which is
\begin{lemma}\label{lemma61}
As $M\nearrow M^*,$ $e_{\alpha^*,M}$ given in (\ref{minimization-problem-critical}) satisfies
\begin{align}\label{eq6.5}
e_{\alpha^*,M}\leq [1+o(1)]\frac{q+r}{q}\bigg(\frac{q\mu}{r}\bigg)^{\frac{r}{r+q}}\bigg[1-\bigg(\frac{M}{M^*}\bigg)^{\frac{r}{n}}\bigg]^{\frac{q}{r+q}}.
\end{align}
\end{lemma}
\begin{proof}
Choose some $P_i\in Z_0$  and let  $y_i\in \mathbb R^n$ satisfy
\begin{align}\label{6dian3def}
\mu=\mu_i=\inf_{y\in\mathbb R^n} H_i(y)=H_i(y_i).
\end{align}
Define $$m_{\tau}=\frac{M}{M_*}\tau^{n}m_0\big[\tau\big(x-\frac{y_i}{\tau}-P_i\big)\big] ~\text{ and }~ w_{\tau}=\frac{M}{M^*}\tau^{n+1}w_0\big[\tau\big(x-\frac{y_i}{\tau}-P_i\big)\big] ~\text{ for~  $\tau>0.$ }$$
 Then we  find
\begin{align}\label{66estimatepotential}
\int_{\mathbb R^n}V(x)m_{\tau}(x)\,dx =&\frac{M}{M^*}\int_{\mathbb R^n}V\Big(\frac{x+y_i}{\tau}+P_i\Big)m_0(x)\,dx
=|\tau|^{q_i}\frac{M}{M^*}\int_{\mathbb R^n}\frac{V\Big(\frac{x+y_i}{\tau}+P_i\Big)}{\Big|\frac{x+y_i}{\tau}\Big|^{q_i}}|x+y_i|^{q_i}m_0(x)\,dx.
\end{align}
Next, we establish the desired estimate for the last equality in (\ref{66estimatepotential}).  To begin with, from \eqref{V2mainassumption_1}, we see that  $\exists \tilde K>|P_i|$ such that
\begin{align}\label{eq6.7}
V(x)<C_2e^{\delta|x|},\ \ \text{for~}|x|>\tilde K,
\end{align}
Define
$$II:=\frac{V\Big(\frac{x+y_i}{\tau}+P_i\Big)}{\Big|\frac{x+y_i}{\tau}\Big|^{q_i}}|x+y_i|^{q_i}m_0,$$
then we decompose the estimates of $II$ into the following three cases:

\textbf{Case 1:} When  $0\leq \frac{|x+y_i|}{\tau}<d$, we use (\ref{taylorexpansionV}) to get
\begin{align}\label{IIestimate1}
|II|=\frac{V\Big(\frac{x+y_i}{\tau}+P_i\Big)}{\Big|\frac{x+y_i}{\tau}\Big|^{q_i}}|x+y_i|^{q_i}m_0 \leq 2a_im_0.
\end{align}

\textbf{Case 2:} When  $d\leq \Big|\frac{x+y_i}{\tau}\Big|\leq 2\tilde K,$ one has from $\tilde K>|P_i|$ that
$
\Big|\frac{x+y_i}{\tau}+P_i\Big|\leq 2\tilde K+|P_i|\leq 3\tilde K,
$
which implies
\begin{align}\label{IIestimate2}
\vert II\vert \leq d^{-q_i}\bigg[\sup_{|x|<3\tilde K}V(x)\bigg]m_0(x).
\end{align}

\textbf{Case 3:} When $\Big|\frac{x+y_i}{\tau}\Big|\geq 2\tilde K$ holds, we find $\Big|\frac{x+y_i}{\tau}+P_i\Big|\geq 2\tilde K-|P_i|\geq \tilde K$ and $\Big|\frac{x}{\tau}\Big|\geq \Big|\frac{x+y_i}{\tau}\Big|-\Big|\frac{y_i}{\tau}\Big|\geq \tilde K$ provided $\tau>0$ large enough. As a consequence, we obtain from \eqref{eq6.7} that
\begin{align*}
&\frac{V\Big(\frac{x+y_i}{\tau}+P_i\Big)}{\Big|\frac{x+y_i}{\tau}\Big|^{q_i}}|x+y_i|^{q_i}=\tau^{q_i}V\Big(\frac{x+y_i}{\tau}+P_i\Big)\leq C_{2}\tau^{q_i}e^{\delta\big|\frac{x+y_i}{\tau}+P_i\big|}
\leq C_{2}\tau^{q_i}e^{3\delta|\frac{x}{\tau}|},\ \text{as~}\tau\rightarrow+\infty,
\end{align*}
which together with \eqref{eq6.1} gives that for $\tau$ sufficiently large,
\begin{align}\label{collect3sect6}
\vert II\vert\leq & C_{2}C_{\delta_1}\tau^{q_i}e^{3\delta|\frac{x}{\tau}|}e^{-\delta_1|x|}
\leq C_{\delta_1}\tau^{q_i}e^{-\frac{\delta_1}{2}|x|}\leq C_{\delta_1}\tau^{q_i}e^{-\frac{\delta_1}{4}|x|}e^{-\frac{\delta_1}{4}\tau \tilde K}
\leq C_{\delta_1}e^{-\frac{\delta_1}{4}|x|}.
\end{align}
  Upon collecting (\ref{IIestimate1}), (\ref{IIestimate2}) and (\ref{collect3sect6}), we arrive at
\begin{equation*}
\vert II\vert\leq \varphi(x):=\left\{\begin{array}{ll}
2a_im_0(x),&\big|\frac{x+y_i}{\tau}\big|<d,\\
d^{-q_i}\Bigg[\sup\limits_{|x|<3\tilde K}V(x)\Bigg]m_0(x),&d<\big|\frac{x+y_i}{\tau}\big|<2\tilde K,\\
C_{\delta_1}e^{-\frac{\delta_3}{4}|x|},&\big|\frac{x+y_i}{\tau}\big|>2\tilde K.
\end{array}
\right.
\end{equation*}
In addition, by using (\ref{taylorexpansionV}), one has the fact that $\lim\limits_{\tau\rightarrow+\infty}II=a_i|x+y_i|^{q_i}m_0(x)$ \text{~a.e.~} in $\mathbb R^n.$
Hence, we employ the Lebesgue dominated convergence theorem to obtain that
\begin{align*}
\lim_{\tau\rightarrow+\infty}|\tau|^{-q_i}\int_{\mathbb R^n} V(x)m_{\tau}(x)\,dx
=\frac{M}{M^*}\int_{\mathbb R^n}a_i|x+y_i|^{q_i}m_0(x)\,dx=\mu,
\end{align*}
where we have used the definition of $\mu$ given by \eqref{6dian3def}. From the  definition of $(m_{\tau},w_{\tau})$ and \eqref{6dian2} we have
\begin{align*}
C_L\int_{\mathbb R^n}\bigg|\frac{w_{\tau}}{m_{\tau}}\bigg|^{r}m_{\tau}\,dx-\frac{1}{1+\frac{r}{n}}\int_{\mathbb R^n}m_{\tau}^{1+\frac{r}{n}}\,dx
=\frac{M \tau^r}{M^*}\bigg[1-\bigg(\frac{M}{M^*}\bigg)^{\frac{r}{n}}\bigg].
\end{align*}
Let $\tau=\Bigg(\frac{q\mu}{r[1-(\frac{M}{M^*})^{\frac{r}{n}}]}\Bigg)^{\frac{1}{r+q}} \to\infty$ as $M\nearrow M^*$, and note that $q_i=q$, we then  deduce from   the above two estimates  that 
\begin{align*}
\mathcal E(m_{\tau},w_{\tau})
=\frac{M}{M^*}\bigg[\bigg(1-\bigg(\frac{M}{M^*}\bigg)^{\frac{r}{n}}\bigg)\tau^{r}+\mu\tau^{-q}\bigg]+o(\tau^{-q})=\frac{q+r}{q}\Big(\frac{q\mu}{r}\Big)^{\frac{r}{r+q}}\bigg[1-\bigg(\frac{M}{M^*}\bigg)^{\frac{r}{n}}\bigg]^{\frac{q}{r+q}}(1+o(1)),
\end{align*}
which completes the proof of the lemma.
\end{proof}
With the upper bound of $e_{\alpha^*,M}$ given in Lemma \ref{lemma61}, we shall prove Theorem \ref{thm14preciseblowup}, which is

\medskip
\textbf{Proof of Theorem \ref{thm14preciseblowup}:}
\begin{proof}
As shown in Theorem \ref{thm13basicbehavior}, one has $x_{\varepsilon}\rightarrow P_i$ for some $1\leq i\leq l.$  Moreover, since $(m_M,w_M)$ denotes the minimizer of (\ref{minimization-problem-critical}), we obtain
\begin{align}
e_{\alpha^*,M}=\mathcal E(m_M,w_M)=&\varepsilon^{-r}C_L\int_{\mathbb R^n}\bigg|\frac{w_{\varepsilon}}{m_{\varepsilon}}\bigg|^{r}m_{\varepsilon}\,dx-\frac{\varepsilon^{-r}}{1+\frac{r}{n}}\int_{\mathbb R^n}m_{\varepsilon}^{1+\frac{r}{n}}\,dx+\int_{\mathbb R^n}V(\varepsilon x+x_{\varepsilon})m_{\varepsilon}(x)\,dx\nonumber\\
\geq &\varepsilon^{-r}\Big[\Big(\frac{M^*}{M}\Big)^{\frac{r}{n}}-1\Big]\frac{n}{n+r}\int_{\mathbb R^n}m_{\varepsilon}^{1+\frac{r}{n}}\,dx+\int_{\mathbb R^n}V(\varepsilon x+x_{\varepsilon})m_{\varepsilon}(x)\,dx.\label{eq6.10}
\end{align}
We have the facts that 
\begin{align}\label{eq6.11}
\int_{\mathbb R^n}V(\varepsilon x+x_{\varepsilon})m_{\varepsilon}(x)\,dx=\varepsilon^{q_i}\int_{\mathbb R^n}\frac{V(\varepsilon x+x_{\varepsilon})}{|\varepsilon x+x_{\varepsilon}-P_i|^{q_i}}\Big|x+\frac{x_{\varepsilon}-P_i}{\varepsilon}\Big|^{q_i}m_{\varepsilon}(x)\,dx,
\end{align}
and in view of $x_{\varepsilon}\rightarrow P_i$,
\begin{align}\label{eq6.12}
\lim_{\varepsilon\rightarrow 0}\frac{V(\varepsilon x+x_{\varepsilon})}{|\varepsilon x+x_{\varepsilon}-P_i|^{q_i}}=\lim_{\varepsilon\rightarrow 0}\frac{a_i|\varepsilon x+x_{\varepsilon}-P_i|^{q_i}+O(|\varepsilon x+x_{\varepsilon}-P_i|^{q_i+1})}{|\varepsilon x+x_{\varepsilon}-P_i|^{q_i}}=a_i, \ \text{~a.e.~in~}\mathbb R^n.
\end{align}
 Now, we claim that
\begin{align}\label{629uniformlyboundxp}
\text{$q_i=q=\max\{q_1,\cdots,q_l\}$ and }~\limsup_{\varepsilon\rightarrow 0}\Big|\frac{x_{\varepsilon}-P_i}{\varepsilon}\Big| \text{~is uniformly bounded.}
\end{align}
If not, we have either $q_i<q$ or up to  a subsequence, $\lim_{\varepsilon\rightarrow 0}\big|\frac{x_{\varepsilon}-P_i}{\varepsilon}\big|=+\infty$.  With the help of \eqref{eq1.40}, \eqref{eq6.11} and \eqref{eq6.12}, we apply Fatou's lemma to get for any constant $\beta\gg1$ large enough,
\begin{align*}
\lim_{\varepsilon\to 0}\varepsilon^{-q}\int_{\mathbb R^n}V(\varepsilon x+x_{\varepsilon})m_{\varepsilon}\,dx
=\lim_{\varepsilon\to 0}\varepsilon^{q_i-q}\int_{\mathbb R^n}\frac{V(\varepsilon x+x_{\varepsilon})}{|\varepsilon x+x_{\varepsilon}-P_i|^{q_i}}\Big|x+\frac{x_{\varepsilon}-P_i}{\varepsilon}\Big|^{q_i}m_{\varepsilon}\,dx\geq \beta\gg1.
\end{align*}
As a consequence, we obtain from \eqref{54equalitynormalize} and \eqref{eq6.10} that 
\begin{align*}
e_{\alpha^*,M}\geq& \varepsilon^{-r}\Big[\Big(\frac{M^*}{M}\Big)^{\frac{r}{n}}-1\Big]\frac{1}{1+\frac{r}{n}}\int_{\mathbb R^n}m_{\varepsilon}^{1+\frac{r}{n}}\,dx+\beta\varepsilon^{q}
=[1+o_{\varepsilon}(1)]\Big[\Big(\frac{M^*}{M}\Big)^{\frac{r}{n}}-1\Big]\varepsilon^{-r}+\beta\varepsilon^{q}\\
\geq &(1+o_{\varepsilon}(1))\frac{q+r}{q}\bigg(\frac{q\beta}{r}\bigg)^{\frac{r}{r+q}}\Bigg[\Big(\frac{M^*}{M}\Big)^{\frac{r}{n}}-1\Bigg]^{\frac{q}{r+q}}, ~\text{ where }~ \beta\gg1.
\end{align*}
This contradicts to Lemma \ref{lemma61}, which finishes the proof of our claim \eqref{629uniformlyboundxp}.

In view of (\ref{629uniformlyboundxp}), we find  $\exists y_0\in\mathbb R^n$ such that up to a subsequence,
\begin{align*}
\lim_{\varepsilon\rightarrow 0}\frac{x_{\varepsilon}-P_i}{\varepsilon}=y_0.
\end{align*}
We next show $y_0$ satisfies  (\ref{132thm14}). 
Since  $q_i=q,$ it follows from \eqref{taylorexpansionV}, (\ref{eq6.3}), \eqref{eq1.40} and  Fatou's lemma that 
\begin{align}
\lim_{\varepsilon\rightarrow 0}\varepsilon^{-q}\int_{\mathbb R^n} V(\varepsilon x+x_{\varepsilon})m_{\varepsilon}\,dx&=\lim_{\varepsilon\rightarrow 0}\int_{\mathbb R^n} \frac{V\Big(\varepsilon \big(x+\frac{x_{\varepsilon}-P_i}{\varepsilon}\big)+P_i\Big)}{|\varepsilon \big(x+\frac{x_{\varepsilon}-P_i}{\varepsilon}\big)|^q}\big| x+\frac{x_{\varepsilon}-P_i}{\varepsilon}\big|^q m_{\varepsilon}\,dx\nonumber\\
&\geq \int_{\mathbb R^n}a_i|x+y_0|^{q}m_0\,dx  \label{byusingsect61}
\geq \mu,
\end{align}
where  the last two equalities hold if and only if (\ref{132thm14}) holds. As a consequence, we obtain
\begin{equation}\label{eq6.16}
\begin{split}
e_{\alpha^*,M}\geq &\varepsilon^{-r}\Big[\Big(\frac{M^*}{M}\Big)^{\frac{r}{n}}-1\Big](1+o(1))+\varepsilon^{q}\mu(1+o(1))\\
\geq &(1+o(1))\frac{q+r}{q}\Big(\frac{q\mu}{r}\Big)^{\frac{r}{r+q}}\Big[\Big(\frac{M^*}{M}\Big)^{\frac{r}{n}}-1\Big]^{\frac{q}{r-q}}\\
=&(1+o(1))\frac{q+r}{q}\Big(\frac{q\mu}{r}\Big)^{\frac{r}{r+q}}\Big[1-\Big(\frac{M}{M^*}\Big)^{\frac{r}{n}}\Big]^{\frac{q}{r-q}}\Big(\frac{M^*}{M}\Big)^{\frac{rq}{n(r-q)}}\\
=&(1+o(1))\frac{q+r}{q}\Big(\frac{q\mu}{r}\Big)^{\frac{r}{r+1}}\Big[1-\Big(\frac{M}{M^*}\Big)^{\frac{r}{n}}\Big]^{\frac{q}{r-q}},
\end{split}
\end{equation}
where the ``=" holds in the second inequality  if and only if
\begin{align}\label{eq6.17}
\varepsilon=\left[\frac{r}{q\mu}\left[1-\big(\frac{M}{M^*}\big)^{\frac{r}{n}}\right]\right]^{\frac{1}{r+q}}(1+o(1)).
\end{align}
We collect the lower bound \eqref{eq6.16} and the upper bound \eqref{eq6.5} to find all equalities in \eqref{eq6.16} must hold.  It follows that all ``=" in \eqref{byusingsect61} and \eqref{eq6.17} also hold.  Therefore, we obtain
 (\ref{131thm14}) and (\ref{132thm14}), which finishes the proof of Theorem \ref{thm14preciseblowup}.
\end{proof}
Theorem \ref{thm14preciseblowup} demonstrates that when potential $V$ admits typical local expansions e.g. power asymptotic function and the nonlinearity exponent in system (\ref{MFG-SS}) is critical, there exist ground states to problem (\ref{minimization-problem-critical}) for $M<M^*$, which concentrate at the location where the potential $V$ is weighted flattest as $M\nearrow M^*.$  Moreover, we have captured the asymptotics of scaling factor $\varepsilon.$

\section{Discussion}

First of all, we have analyzed the $W^{2,p}$ regularity of the solution to the HJ equation (\ref{newsectioneq1}) with subquadratic gradient terms by performing blow-up analysis.  With the aid of the regularity results, we have employed the variational method to study the existence of least energy solutions to system \eqref{MFG-SS} with the decreasing cost and some technical assumptions imposed on potential $V$ when the nonlinearity exponent in the power local coupling is mass critical.  Based on the seminal work of Cirant et al. \cite{cesaroni2018concentration}, we have classified the ground states to the associated energy (\ref{energy-dual}) in terms of the total mass of population density $m$ when the coupling exponent $\alpha=\frac{r}{n}$.  More precisely, we have obtained there exists a critical mass $M^*$ such that if the total mass of density denoted by $M$ satisfying $M<M^*$, problem (\ref{minimization-problem-critical}) admits minimizers; otherwise, there is no minimizer to (\ref{minimization-problem-critical}) by using the assumptions (V3) imposed on potential $V$ when the exponent in the local coupling is mass critical.  Furthermore, we have showed there exists the concentration phenomenon as $M\nearrow M^*$ and captured the asymptotic behaviors of minimizers.  In particular, while proving our main results, Gagliardo-Nirenberg type's inequality corresponding to the ground states for the potential-free MFG system has been established.  In our argument, some innovation ideas what we have proposed is the application of Pohozaev identities on proving strongly $L^1(\mathbb R^n)$ convergence of population density, which avoids the discussion of concentration compactness as shown in \cite{cesaroni2018concentration}.  On the other hand, while studying the existence of least energy solutions, we relax the $|x|^b$ restriction on potential $V$ by applying maximal regularities shown in Theorem \ref{thmmaximalregularity} and \cite{cirant2022local}. 

We would like to point out that there are also some interesting questions that deserve investigations in the future.   While establishing the optimal inequality in Section \ref{sect3-optimal}, we imposed some technical assumption on $m$, which is the boundedness of $\int_{\mathbb R^n}m|x|^b\,dx$ for sufficiently small $b>0.$  We believe that this condition is unnecessary since our techniques arise from the continuity method but the existence of minimizers to problem (\ref{sect2-equivalence-scaling}) in the mass subcritical exponent case must be established with this technical assumption on $m$ due to the approaches taken in \cite{cesaroni2018concentration}.  To remove this assumption, we must propose some new ideas and it seems significant to understand the properties of principal part in $\mathcal E$ given by (\ref{energy-dual}) more clearly.  Other interesting question is the properties, such as radially symmetry, uniqueness and positivity of ground states, when system (\ref{MFG-SS}) is potential-free.  Since the potential-free system is also coupled, it seems need other ideas rather than moving plane methods to show the symmetry of solutions.     One of interesting directions is the extension of the potential functions $V$.  We impose some mild assumptions (\ref{V2mainasumotiononv}) on $V$, which does not include the super exponential growth case and the logarithmic potential.  It seems a challenge to remove the technical assumptions due to the lower bound estimates of the value function $u.$ 

\begin{appendices}
\setcounter{equation}{0}
\renewcommand\theequation{A.\arabic{equation}}

\section{Ground States of Schr\"{o}dinger equations and MFGs with $r'=2$}\label{appendixA}
This Appendix is devoted to 
the relationship between the ground states to (\ref{MFG-SS}) and (\ref{nonlinear-Schrodinger}), which is
\begin{proposition}\label{corequivschromfg}
Assume $r=r'=2,$ $H$ is given by (\ref{MFG-H}), $f(m)=-C_fm^{\alpha}$ with $\alpha\in\big(0,\frac{r}{n}\big)$ and $V$ satisfies (\ref{cirant-V}) in system (\ref{MFG-SS}). Let $(m_M,w_M)$ be minimizers to problem (\ref{ealphaM-117}).  Correspondingly, $(m_M,u_M,\lambda_M)\in W^{1,p}(\mathbb R^n)\times C^2(\mathbb R^n)\times \mathbb R$ solves (\ref{MFG-SS}) for any fixed $M>0.$  In addition, let $v_M$ be minimizers to problem 
\begin{align}\label{appendixa1problem}
\min\limits_{v\in H^1(\mathbb R^n),\int_{\mathbb R^n}v^2\,dx=M}\mathcal F(v)
\end{align}
with $\mathcal F(v)$ defined by (\ref{variation-schrodinger}). 
 Then we have that on one hand, $\bar v_M:=\sqrt{m_M}$ give rise to minimizers of the problem (\ref{appendixa1problem}); on the other hand, $(\bar m_M,\bar w_M):=(v_M^2,2v_M\nabla v_M)$ give rise to minimizers of problem (\ref{ealphaM-117}). 
\end{proposition}
\begin{proof}
First of all, we show that $\nabla\sqrt{m_M}\in L^2(\mathbb R^n).$  To this end, by applying Theorem 3.1 in \cite{metafune2005global}, one obtains
\begin{align}\label{followsfrombeforeinrocor1}
\int_{\mathbb R^n}\frac{|\nabla m_M|^2}{m_M}\,dx\leq \hat C\int_{\mathbb R^n}m_M|\nabla u_M|^2\, dx,
\end{align}
where $\hat C>0$ is some constant.  As shown in \cite{cesaroni2018concentration}, when potential $V$ satisfies (\ref{cirant-V}), $u_M$ at most grows algebraically but $m_M$ has the exponential decay property.  It follows from (\ref{followsfrombeforeinrocor1}) that 
\begin{align}\label{boundinvokejfacompare}
\int_{\mathbb R^n}\frac{|\nabla m_M|^2}{m_M}\,dx\leq \hat C\int_{\mathbb R^n}m_M|\nabla u_M|^2\, dx<+\infty,
\end{align}
which indicates $\nabla\sqrt{m_M}\in L^2(\mathbb R^n).$

On the other hand, since $(m_M,w_M)\in W^{1,p}(\mathbb R^n)\times L^p(\mathbb R^n)$ for any $p>1$, are minimizers to (\ref{ealphaM-117}) \cite{cesaroni2018concentration}, we find for any fixed $M>0,$ 
\begin{align}
II_1:=\int_{\mathbb R^n}\frac{|\nabla m_M+g_M|^2}{m_M}\,dx
\end{align}
exists with $w_M=\nabla m_M+g_M$ and $\nabla \cdot g_M=0$ weakly.  In addition, invoking (\ref{boundinvokejfacompare}) we have the facts that $\int_{\mathbb R^n}\frac{|\nabla m_M|^2}{m_M}\,dx$ and $\int_{\mathbb R^n}\frac{|g_M|^2}{m_M}\,dx$ are well-defined.  Then, it follows that for any $R\rightarrow +\infty,$
\begin{align}
\lim_{R\rightarrow +\infty}\int_{B_R(0)}\frac{|\nabla m_M+g_M|^2}{m_M}\, dx=II_1<+\infty.
\end{align}
Noting that $m>0$ for any $x\in\mathbb R^n$, by straightforward computation, one has
\begin{align}\label{real119relationIIfcompare}
II_1=\lim_{R\rightarrow+\infty}\bigg[\int_{B_R(0)}\frac{|\nabla m_M|^2}{m_M}\,dx +2\int_{B_R(0)}\frac{\nabla m_M}{m_M}\cdot g_M\,dx+\int_{B_R(0)}\frac{|g_M|^2}{m_M}\,dx\bigg].
\end{align}
where $\frac{\nabla m_M}{m_M}=\nabla \log m_M$ is well-defined.  Now, we claim that there exists a sequence $ R_k\rightarrow +\infty$ as $k\rightarrow +\infty$ such that 
\begin{align}\label{keyclaim119after-compare-schormfg}
\int_{B_{R_k}(0)}(\nabla \log m_M)\cdot g_M\, dx \rightarrow 0.
\end{align}
Since $g_M\in L^2(B_{R}(0))$ and $\nabla \log m_M \in L^2(B_R(0))$ for any large but fixed $R>0,$ we apply the integration by parts to get
\begin{align}\label{120-schromfgcompare}
\int_{B_{R_k}(0)}(\nabla \log m_M)\cdot g_M\, dx=\int_{\partial B_{R_k}(0)}\log m_M g_M \cdot \nu\,dS,
\end{align}
where we have used $\nabla\cdot g_M=0$ weakly.  Next, we shall show that 
\begin{align}\label{121-schrodingerequivalmfg}
\int_{\partial B_{R_k}(0)}|\log m_M|\cdot|g_M|\,dS\rightarrow 0,\text{~as~}k\rightarrow\infty.
\end{align}
Indeed, if the following fact holds: 
\begin{align}\label{toshowthisequilschromfg}
\int_{\mathbb R^n}|\log m_M|\cdot|g_M|\,dx<+\infty,
\end{align}
then we obtain the desired conclusion \eqref{121-schrodingerequivalmfg} since we have the fact that 
\begin{align*}
\int_{\mathbb R^n}|\log m_M|\cdot |g_M|\,dx=\int_0^{+\infty}\bigg(\int_{\partial B_R(0)} |\log m_M|\cdot |g_M|\,dS\bigg)\, dR.
\end{align*}
To prove (\ref{toshowthisequilschromfg}), we apply H\"{o}lder's inequality to get
\begin{align}\label{123rightsideschrmfg}
   \bigg(\int_{\mathbb R^n}|\log m_M|\cdot|g_M|\,dx\bigg)^2 \leq\bigg(\int_{\mathbb R^n}m_M|\log m_M|^2\,dx \bigg)\cdot\bigg(\int_{\mathbb R^n}\frac{|g_M|^2}{m_M}\,dx\bigg).
\end{align}
To estimate the first term in Right Hand Side of (\ref{123rightsideschrmfg}), we use the fact $m_M|\log m_M|^2\leq C \big(m_M^2+\sqrt{m_M}\big)$ with $m_M\geq 0$ and positive constant $C$ to obtain 
\begin{align}
\int_{\mathbb R^n}m_M|\log m_M|^2\,dx\leq C\Big(  \int_{\mathbb R^n} m_M^2\,dx +\int_{\mathbb R^n} \sqrt{m_M}\,dx\Big).
\end{align}
Since $m_M\in L^2(\mathbb R^n)$ and $0<m_M<e^{-\delta|x|}$ for some $\delta>0$, we have $\sqrt{m_M}\in L^1(\mathbb R^n),$ which completes the proof of (\ref{121-schrodingerequivalmfg}).  In view of (\ref{120-schromfgcompare}), one finds as $k\rightarrow +\infty$, the claim \eqref{keyclaim119after-compare-schormfg} holds.  Focusing on (\ref{real119relationIIfcompare}), we have $g_M\equiv 0$ to guarantee $(m_M,w_M)$ are minimizers to problem (\ref{ealphaM-117}), which implies the energy (\ref{energy-dual}) satisfied by $(m_M,w_M)$ is exactly reduced to (\ref{variation-schrodinger}) by using the transform $m=v^2$.  Now, we have the fact that the infima of problem (\ref{ealphaM-117}) and problem (\ref{appendixa1problem}) are the same.

Conversely, for any minimizer $v_M$ to problem (\ref{appendixa1problem}), it is straightforward to verify that $(\bar m_M,\bar w_M):=(v_M^2,2v_M\nabla v_M)\in \mathcal K_M$ given by (\ref{constraint-set-K}), which is a minimizer to problem (\ref{ealphaM-117}).  It completes the proof of our proposition.
\end{proof}
\end{appendices}




\bibliographystyle{abbrv}
\bibliography{ref}

\end{document}